\documentclass[11pt,a4paper]{article}
\usepackage[T1]{fontenc}
\usepackage[utf8]{inputenc}
\usepackage{authblk}

\usepackage[margin=1in]{geometry}  
\usepackage{graphicx}              
\usepackage{amsmath}               
\usepackage{amsfonts}              
\usepackage{amsthm}                

 \newtheorem{theorem}{Theorem}[section]
 \newtheorem{lemma}[theorem]{Lemma}
 
 \newtheorem{cor}[theorem]{Corollary}
 
 \newtheorem{definition}[theorem]{Definition}
 
\theoremstyle{remark}
\newtheorem{remark}{\bf Remark}

   \usepackage{algorithmic}
\usepackage{algorithm}

\title{Local principle satisfying high order total variation diminishing approximation for non-sonic data extrema} 
\author{Ritesh Kumar Dubey\footnote{\texttt{mail-to: riteshkd@gmail.com}}, Biswarup Biswas \\
  {\small Research Institute, SRM University, Tamilnadu, India}\\ \& \\Vikas Gupta\\{\small LNMIIT Jaipur, Rajsthan, India}}
\date{}
\begin{document}
\maketitle

   \begin{abstract}
     The main contribution of this work is to construct higher than
     second order accurate total variation  diminishing (TVD) schemes which can preserve high accuracy at
     non-sonic extrema with out induced local oscillations. It is done
     in the framework of local maximum principle (LMP) and
     non-conservative formulation. The representative uniformly second
     order accurate schemes are converted in to their
     non-conservative form using the ratio of consecutive
     gradient. These resulting schemes are analyzed for their
     non-linear LMP/TVD stability bounds using the local maximum
     principle.  Based on the bounds, second order accurate
     hybrid numerical schemes are constructed using a shock
     detector. Numerical results are presented to show that such
     hybrid schemes yield TVD approximation with second or higher
     order convergence rate for smooth solution with extrema. 
    \end{abstract} 

    {\bf keyword} Hyperbolic conservation laws;Smoothness parameter; Non-sonic critical
      point, Total variation stability, Finite difference schemes.

  \section{Introduction}\label{sec1} We consider the 1D scalar
  conservation law associated to the conserved variable
  $u(x,t)$,
 \begin{eqnarray} \displaystyle \frac{\partial}{\partial t}
    u(x,t) + \frac{\partial}{\partial x}f(u(x,t)) = 0,\;\; (x,\,t)\in
    \mathbb{R}\times\,
    \mathbb{R}^{+} \label{nonlin}\\ u(x,0)=u_{0}(x)\nonumber 
\end{eqnarray}
  where $f(u)$ is a non-linear flux function. The numerical
  approximation for the solution of (\ref{nonlin}) is done by the
  discretization of the spatial and temporal space into $N$ equispaced
  cells $I_{i} = [x_{i-\frac{1}{2}}, x_{i+\frac{1}{2}}],\; i=0,1,\dots
  N$ of length $\Delta x$ and $M$ equispaced intervals $[t^n,
    t^{n+1}],\; n=0,1,\dots, M$ of length $\Delta t$ respectively.
  Let $x_{i}=i\Delta x$ and $t^{n}=n\Delta t$ denote the cell center
  of cell $I_{i}$ and the $n^{th}$ time level respectively then a
  conservative numerical approximation for (\ref{nonlin}) can be defined
  by \begin{equation} {u}^{n+1}_{i} = {u}_{i}^{n} - \lambda
    \left(\mathcal{F}^n_{i+\frac{1}{2}}-
    \mathcal{F}^n_{i-\frac{1}{2}}\right),\;\;\lambda = \frac{\Delta
      t}{\Delta x}. \label{s2eq2} \end{equation} where
  ${u^n_{i}=u(x_{i},t^{n})}$ and $\mathcal{F}^n_{i\pm
    \frac{1}{2}}$ is the numerical flux function defined at the cell
  interface $x_{i \pm \frac{1}{2}}$ at time level $n$.  The characteristics speed $a(u)= \frac{\partial
    f(u)}{\partial u}$ associated with (\ref{nonlin}) can be approximated as, 
  \begin{equation}
    \displaystyle a^n_{i+\frac{1}{2}}= \left\{\begin{array}{clc}
    \displaystyle \frac{F^n_{i+1}-F^n_{i}}{{u}^n_{i+1} - {u}^n_{i}}& \mbox{if}
    &{u}^n_{i+1} \neq {u}^n_{i},\\ & &\\ \displaystyle
    \left. \frac{\partial f}{\partial u}\right|_{{u}^n_{i}} & \mbox{if}
    &{u}^n_{i+1}=
         {u}^n_{i} \end{array}\right.,\label{speed}
\end{equation} 
where $F_{i}= f(u^n_{i})$. In general due to non-linearity of
(\ref{nonlin}), beyond a small finite time, even for a smooth initial
data the evolution of discontinuities in the solution $u(x,t)$ is
inevitable. Therefore, it is required to have a conservative
approximation of the solution with high accuracy and crisp resolution
of such discontinuities with out numerical oscillations. Contrary to
this need, most classical high order schemes despite of being linearly
Von-Neumann stable give oscillatory approximation for discontinuities
even for the trivial case of transport equation i.e., $f(u)=au,\,
0\neq a\in R $.  Such oscillatory approximation can not be considered
as admissible solution since it violets the following global maximum
principle satisfied by the physically correct solution $u(x,t)$ of
(\ref{nonlin}) i.e.,
\begin{equation} \min(u_{0}(x)) \leq u(x,t) \leq \max(u_{0}(x)),
  \forall (x,\;t)\in \mathbb{R}\times\, \mathbb{R}^{+}. 
\label{mp} 
\end{equation} In
order to overcome these undesired numerical instabilities, various
notion of non-linear stability are developed in the light of maximum
principle (\ref{mp}). Examples of Maximum principle satisfying schemes
are monotone schemes \cite{vanLeer1974, Crandall}, total variation
diminishing (TVD) schemes
\cite{harten1984,sweby1984,Yee1987,sanders1988,davis1987,
  goodman1988,rkd2007,zhang2010}. Some uniformly high order
maximum-principle satisfying and positivity preserving schemes are
\cite{LaxLiu1998, zhangjcp2010,zhang2011}.  There are other
non-oscillatory schemes which do not strictly follow maximum principle
but practically give excellent numerical results e.g., Essentially
non-oscillatory (ENO) and weighted ENO schemes see \cite{cwshu1999}
and references therein.  It is known that among global maximum
principle satisfying schemes, the monotone and total variation
diminishing (TVD) schemes experience difficulties at data extrema.  On
the one hand, such high order schemes locally degenerate to first
order accuracy at non-sonic data extrema and on the other hand, even
such a uniformly first order accurate schemes may exhibit induced
local oscillations at data extrema. In this work the focus is on the
construction of improved TVD schemes at smooth data extrema.
 
  \section{Global maximum principle and data extrema}\label{sec2} 
  The above global maximum principle (\ref{mp}) satisfying monotone
  and TVD schemes have been of great interest mainly due to excellent
  convergence proofs for entropy solution
  \cite{Sanders1983,Lefloach99} and \cite{chakra,yang1996}
  respectively.  The key idea is, any maximum principle satisfying
  scheme produce a bounded solution sequence and convergence follows
  due to compactness of solution sequence space \cite{LeVeque1992}. It
  can be shown that monotone stable scheme $\Rightarrow$ TVD scheme
  $\Rightarrow$ monotonicity preserving scheme (or Local extremum
  diminishing (LED)) scheme \cite{harten1983,jameson1995}. Unfortunately,
  monotone as well TVD schemes experience difficulty at data
  extrema. The monotone stability relies on monotone data and
  therefore a monotone scheme preserves the monotonicity of a data set
  by mapping it to a new monotone data set but fails to preserve the
  non-monotone solution region i.e., at data extrema. These monotone
  schemes are criticized mainly due to barrier theorem which state
  that a 'linear' three point monotone scheme can be at most first
  order accurate \cite{godunov}. Later, second order 'non-linear'
  conservative monotone schemes are constructed using limiters but
  again by compromising on second order accuracy at extrema,
  e.g. \cite{vanLeer1974}. The TVD stability mimics the maximum
  principle as it relies on the condition that global extremum values
  of solution must remain be bounded by global extremum values of
  initial solution. In \cite{harten1983}, Harten gave the concept of
  total variation diminishing scheme by measuring the variation of the
  grid values as follows
  \begin{definition}\label{def1} Conservative scheme
    (\ref{s2eq2}) is called total variation diminishing if
    \begin{equation}\label{tvddef} \displaystyle
      \sum_{i=-\infty}^{\infty}\left|\Delta_{-}u_{i}^{n+1}\right|\leq
      \sum_{i=-\infty}^{\infty} \left|\Delta_{-}u_{i}^{n}\right|
    \end{equation} where
    $\Delta_{+}{u}_{i} = \Delta_{-}{u}_{i+1} = {u}_{i+1} -{u}_{i}$.
  \end{definition}
  Note that that the definition \ref{def1} is global as it is defined on the whole
  computational domain and ensures that global maxima or minima of
  initial solution $u_{0}(x)$ will not increase or decrease
  respectively.  Such conservative TVD schemes are heavily
  criticized because, even if they are higher order accurate in most
  solution region, they give up second order of accuracy at non-sonic
  critical values of the solution \cite{tadmor1988,chakra}\footnote{Sanders also defined the total
    variation by measuring the variation of the reconstructed
    polynomials and such TVD schemes can be uniformly high order
    accurate \cite{sanders1988,zhang2010}.}.  {\it We emphasize that these depressing results on
    degeneracy of accuracy of TVD method are given for {\bf
      conservative} schemes and in the above {\bf global sense}}. More
  precisely the global nature of TVD definition (\ref{tvddef}) allows
  shift in indices technique in $\sum$ sign and is extensively used
  in different terms of the infinite sums in the TVD proofs of various
  schemes and results in the literature including the following one
  due to Harten \cite{harten1983}.
  \begin{lemma}\label{lem1} A conservative scheme in
    Incremental form (I-form)
    \begin{equation}
      {u}_{i}^{n+1}=u_{i}^{n}+\alpha_{i+\frac{1}{2}}({u}^{n}_{i+1}-{u}^{n}_{i})-
      \beta_{i-\frac{1}{2}}({u}^{n}_{i}-{u}^{n}_{i-1})\label{iform}
    \end{equation}
    is TVD iff $\alpha_{i+\frac{1}{2}}\geq0, \beta_{i+\frac{1}{2}}\geq
    0\;\mbox{and}\; \alpha_{i+\frac{1}{2}}+ \beta_{i+\frac{1}{2}}\leq
    1,\; \forall i$. \cite{Thomas}
  \end{lemma}
\subsection{Degenerate accuracy at extrema:} 
In \cite{chakra}, proof for degeneracy to first order accuracy at
non-sonic critical points of solution i.e., points $u^{*}(x,t)$
s.t. $f^{'}(u^{*}(x,t)) \neq 0= u^{*}_{x}(x,t)$ is mainly based on
modified equation analysis and a {\it conservative} semi-discrete
version of Lemma \ref{lem1}. In \cite{tadmor1988}, using a trade off
between second order accuracy and TVD requirement along with shift in
indices technique, it is shown that second order accuracy must be
given up by a {\it conservative} TVD scheme at non sonic critical
values $u_{i}=u^{*}(x_{i},t)$ which corresponds to extreme values i.e.,
$[u(x_{i}+\Delta\,x,t)-
  u(x_{i},t)].[u(x_{i},t)-u(x_{i}-\Delta\,x,t)]<0 \neq f^{'}(u_{i})$.
It is also worthy to note that problem of degenerate accuracy by
modern high resolution TVD schemes is also due to their construction
procedure. For example, the numerical flux function of flux limiters or
slope limiters based TVD schemes is essentially design in such a way
that it reduces to first order accuracy at extrema and high
gradient region by forcing limiter function to be zero see
\cite{Piperno,Dubey2013} and references therein. This makes it
impossible for a limiter based TVD schemes to achieve higher than
first order accuracy at solution extrema as well at steep gradient
region \cite{LeVeque1992}. Thus every high order TVD (in global sense
(\ref{tvddef}) ) scheme suffers from clipping error and cause flatten
approximation for smooth extrema though they sharply capture
discontinuities\cite{Laney}.

\subsection{Induced local oscillations:}  Apart from compromise in uniform high accuracy, it is notable that
  global maximum principle satisfying monotone and TVD schemes do not
  necessarily ensure preservation of non-monotone data set i.e., for a
  data set with extrema as demonstrated in Figure \ref{F1a}(b). In
  particular first order monotone and TVD schemes with {\bf large
    coefficient of numerical viscosity} can allow the occurrence of
  induced oscillations at data extrema and formation of new local
  extremum values as shown in Figure \ref{F1b}(a). This phenomena of
  generation of local oscillations at extrema is reported and analyzed
  for well known monotone and TVD three point {\it Lax-Friedrichs}
  scheme in \cite{Breuss2010,jiequan2009,jiequan2011} similar to
  Figure \ref{F1b}(a).
  \begin{figure}[h]
    \begin{tabular}{cc}
      \hspace{-0cm} \includegraphics[%
      scale=0.3]{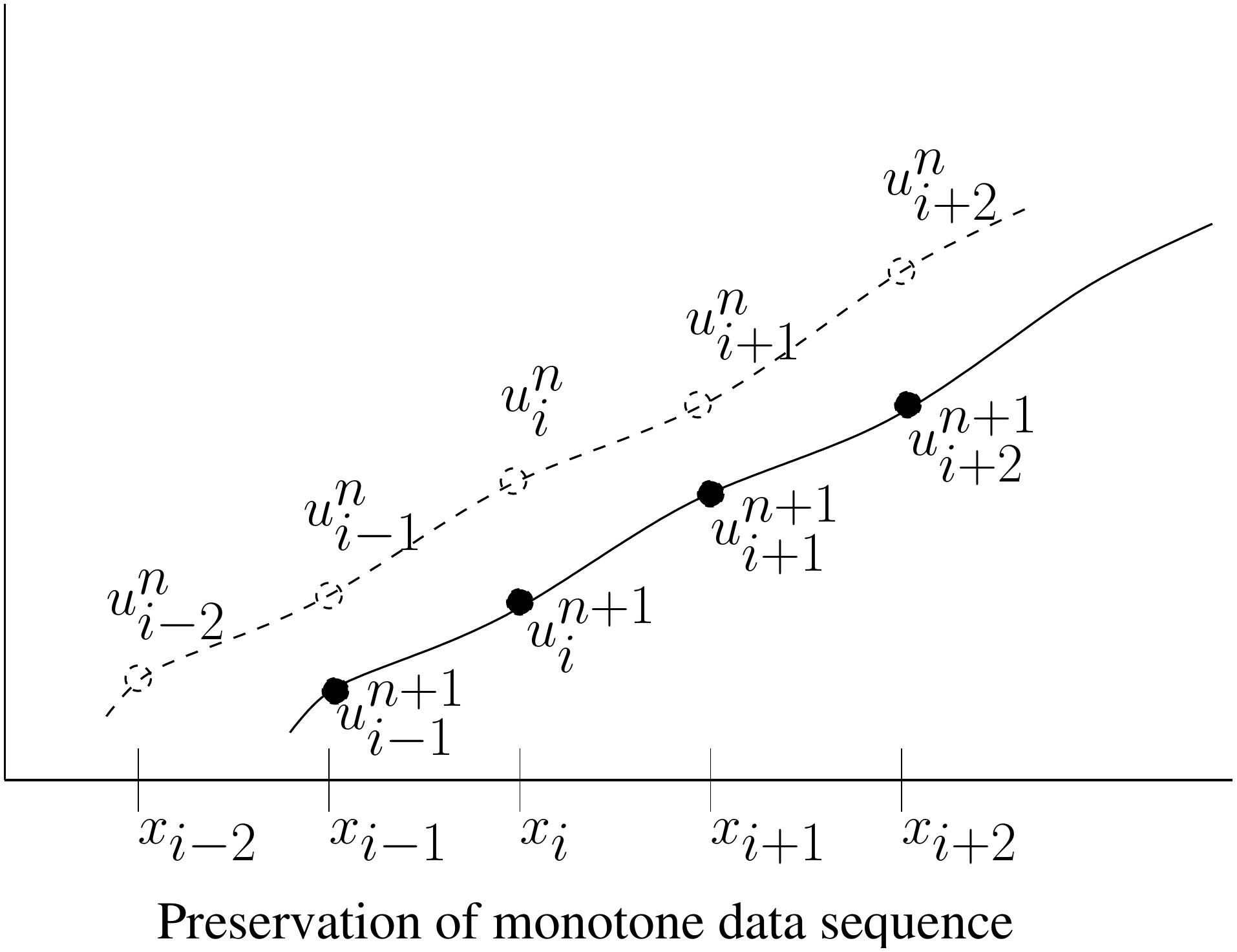} & \includegraphics[%
      scale=0.3]{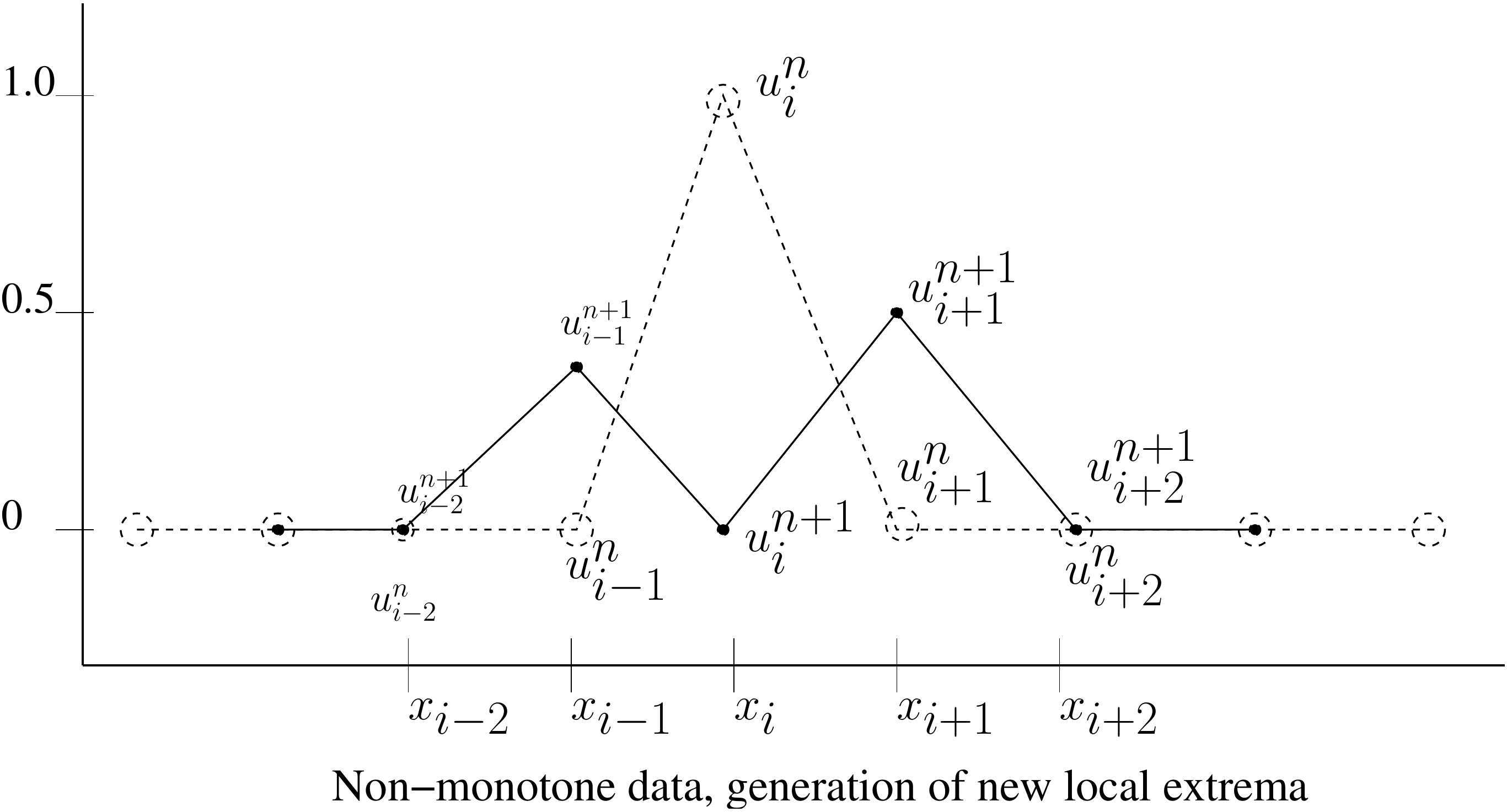}\\
      (a) & (b) 
    \end{tabular} 
    \caption{\label{F1a} Induced oscillations may occur for
      non-monotone data extrema. (a) Monotone stability rely on
      monotone data sequence. (b) For non-monotone data set TVD
      definition (\ref{tvddef}) is satisfied though updated approximation
      is oscillatory.}
  \end{figure}
  \begin{figure}[h]
    \begin{tabular}{cc}
      \includegraphics[scale=0.45]{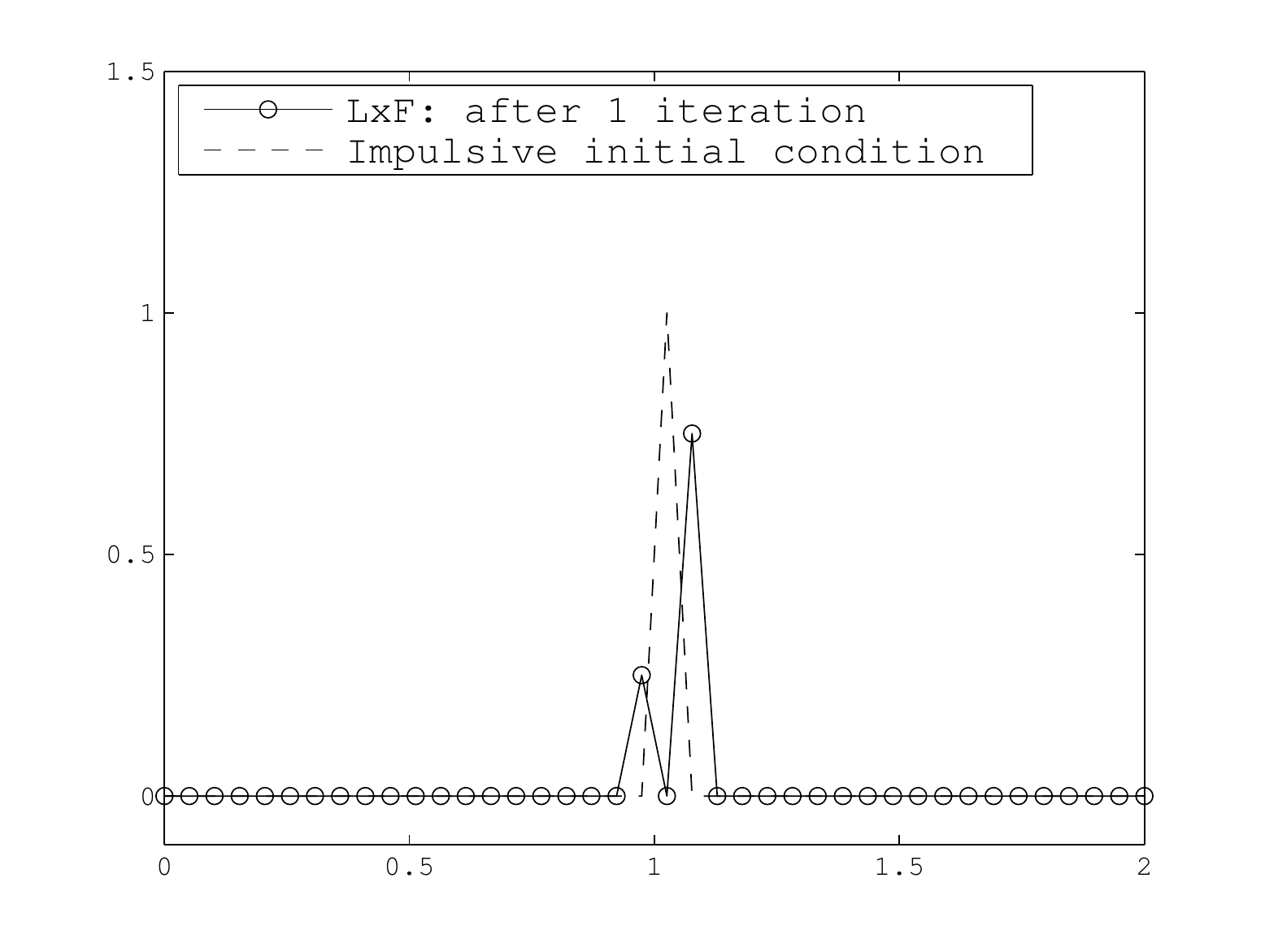} &
      \includegraphics[scale=0.45]{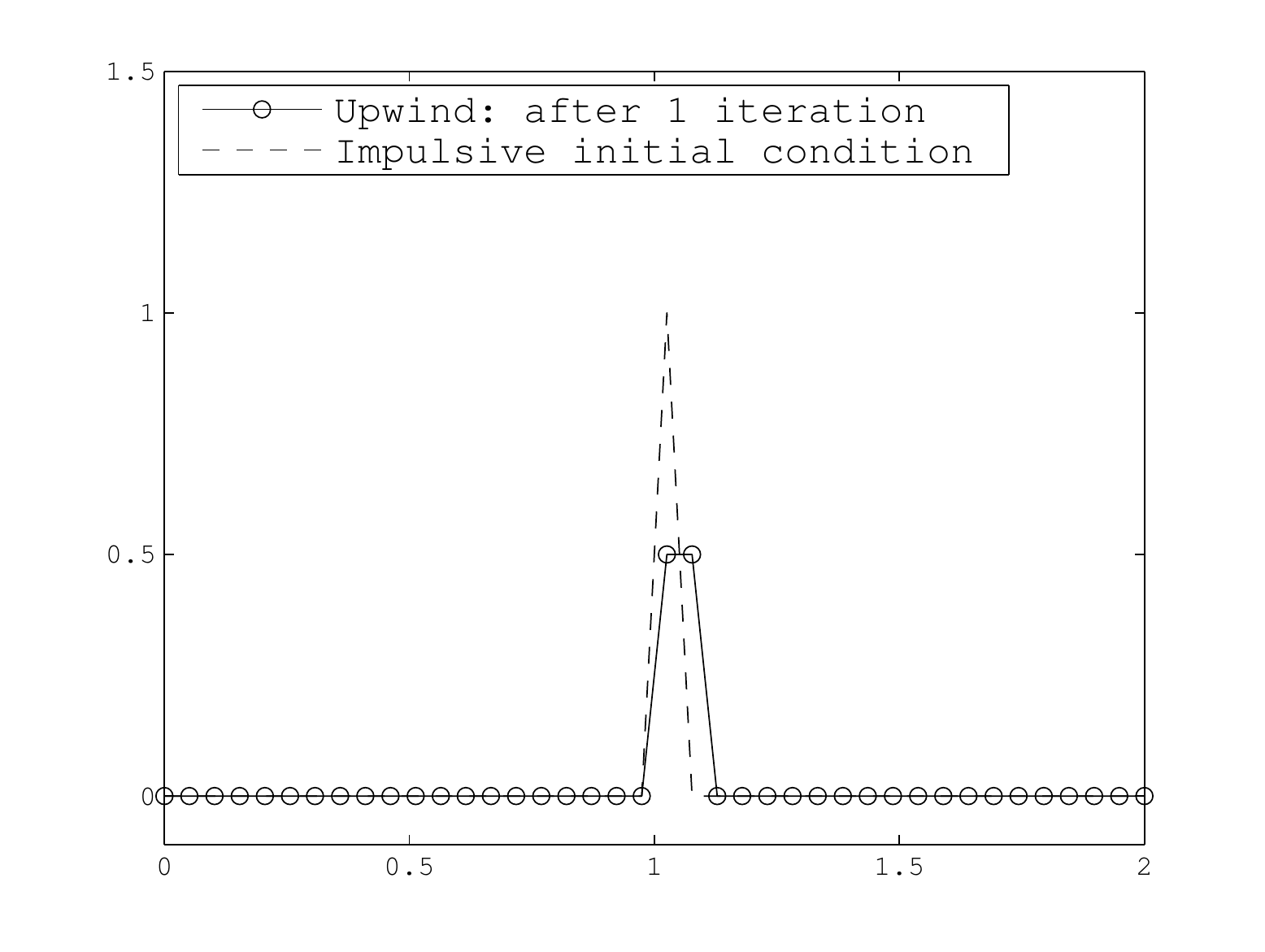}\\ (a) &
      (b) \end{tabular} \caption{\label{F1b} Numerical
      approximation of Linear transport equation (\ref{transport})
      with impulsive initial data. (a) By monotone and TVD
      Lax-Friedrichs scheme: Induced oscillations (b) First order two
      point upwind scheme: Absence of induced oscillations.}  
  \end{figure}
  It is interesting to note that two point monotone and TVD upwind
  scheme does not exhibit induced oscillations in Figure
  \ref{F1b}(b). 
  
  The main aim of this work is to construct uniformly non-oscillatory
  shock capturing monotone and TVD methods with high accuracy for
  non-sonic smooth extrema\footnote{It shows improved TVD
    approximation in the region of degenerate accuracy reported in
    \cite{chakra,tadmor1988}.}. In order to achieve it, a {\bf
    non-conservative} formulation is done using framework of local
  maximum principle (LMP) with the help of gradient ratio parameter.
  The rest of paper is organized as: For completeness, in section
  \ref{sec3}, local maximum principle (LMP) stability is defined for
  two points schemes. It is shown in Lemma \ref{lem2}, that away from
  sonic point LMP stability implies global monotone and TVD
  stability. In section \ref{sec4}, we analyze representative
  uniformly second  order schemes in non-sonic region for
  their TVD (or LMP) stability by converting them into two point
  schemes. This yields computable bounds for the stability of these
  scheme and are presented as main results in Theorems
  \ref{thm1}-\ref{thm3}. These obtained TVD bounds ensure for {\it
    second rrder TV stable approximation for smooth
    solution with non-sonic critical point of extreme nature}. In
  section \ref{sec5}, hybrid schemes are designed using the TVD
  bounds and a shock detector technique. Numerical results are
  given to support the theoretical results and claim. Conclusion and
  future work is discussed in section \ref{sec6}.
  \section{Local maximum principle (LMP) stability}\label{sec3}
  It is clear from the above discussion that the global maximum principle
  (\ref{mp}) satisfying monotone or TVD stability experience
  difficulty in the presence of non-monotone data with extrema. These
  two local phenomena i.e, induced oscillations and degeneracy in accuracy
  at non-sonic extrema by monotone and TVD schemes motivate us to look
  for their non-linear stability locally at non-sonic extrema. Consider the
  following local maximum principle (LMP)\footnote{Also known as
    upwind range condition \cite{Laney}} for scalar conservation law
  (\ref{nonlin}), \begin{subequations} \label{urc} \begin{align}
      \min_{x_{i-1}\leq x \leq x_{i}}u(x,t^{n})\leq u(x,t^{n+1})\leq
      \max_{x_{i-1}\leq x \leq x_{i}}u(x,t^{n})&\;\; \mbox{if}\;
      f^{'}(u) >0, \\ \min_{x_{i}\leq x \leq x_{i+1}}u(x,t^{n})\leq
      u(x,t^{n+1})\leq \max_{x_{i}\leq x \leq x_{i+1}}u(x,t^{n})&\;\;
      \mbox{if}\;f^{'}(u)<0.  \end{align} \end{subequations} In case
  of two point schemes, initial solution data will always be monotone
  as either $u(x_{a},t)\leq u(x_{b},t),\; x_{a}\neq x_{b}$ or vice
  verse thus the LMP condition (\ref{urc}) reduces to
  \begin{subequations} \label{urc1} \begin{align} \min(u_{i-1}^{n},
      u_{i}^{n}) \leq u_{i}^{n+1} \leq \max(u_{i-1}^{n}, u_{i}^{n})
      &\;\; \mbox{if} f^{'}(u)>0,\\ \min(u_{i}^{n}, u_{i+1}^{n}) \leq
      u_{i}^{n+1} \leq \max(u_{i}^{n}, u_{i+1}^{n}) &\;\; \mbox{if}\;
      f^{'}(u)< 0.  \end{align} \end{subequations} Thus away from
  sonic point $u_{i}^{*}$ i.e., $a_{i+\frac{1}{2}}=f^{'}(u_{i}^{*})\neq0$ define, 
  \begin{definition}\label{def2} A numerical scheme is LMP stable (\ref{urc1}) if it can be written
    as \begin{equation} \label{lmpscheme} u_{i}^{n+1} =
      \left\{ \begin{array}{ll} \mathcal{C} u_{i}^{n} + \mathcal{D}
        u_{i-1}^{n} &\;\; \mbox{if}\;\;\;\;\;\;\;\;\; 0< \lambda
        a^n_{i+\frac{1}{2}}\leq 1,\\ \mathcal{C} u_{i}^{n} +
        \mathcal{D} u_{i+1}^{n} &\;\; \mbox{if} \;\;\;\;\; -1\leq
        \lambda a^n_{i+\frac{1}{2}}< 0, \end{array}\right.
    \end{equation} where
    coefficients $\mathcal{C}$and $\mathcal{D}$ are real functions such
    that $\mathcal{C}\geq 0, \mathcal{D} \geq 0$, $\mathcal{C} +
    \mathcal{D} = 1$.
  \end{definition}
  Scheme (\ref{lmpscheme}) is essentially a convex combination of two
  point values of $u(:,t^{n})$ thus ensures that the updated solution
  value of $u(x_{i}, t^{n+1})$ will remain be bounded by both point values
  without introducing of new local maxima-minima.  Note that two point
  first order upwind scheme is a natural example of LMP stable scheme
  with coefficients
  \begin{subequations} \begin{align} \mathcal{C} = 1-\lambda
      a_{i-\frac{1}{2}}, \mathcal{D} = \lambda a_{i-\frac{1}{2}} &\;\;
      \mbox{if}\;\;\;\;\; 0< \lambda f^{'}(u_{i}^{n})\leq 1\\
      \mathcal{C} = 1-\lambda a_{i+\frac{1}{2}}, \mathcal{D} = \lambda
      a_{i+\frac{1}{2}} &\;\; \mbox{if}\;\;\;\;\; -1\leq \lambda
      f^{'}(u_{i}^{n})< 0.  \end{align}
  \end{subequations}
  This justifies the non-occurrence of local oscillation by first order
  upwind scheme in Figure \ref{F1b}(b).  From definition \ref{def2}
  it follows,
  \begin{lemma}\label{lem2} Local maximum principle stable
    scheme (\ref{lmpscheme}) is global total variation diminishing stable.
  \end{lemma} 
  \begin{proof} Using relation $\mathcal{C} =1-\mathcal{D}$, rewrite
    (\ref{lmpscheme}) in the {\bf non-conservative}\footnote{Note that for problems with constant $f'(u)$ e.g. linear transport
  equation (\ref{transport}), the form (\ref{lmptvd}) is
  conservative. Also one can obtained a conservative approximation
  from (\ref{lmptvd}) by defining $\mathcal{D}$ suitably such as
  $\displaystyle \mathcal{D}=
  \frac{\Delta_{-}f(u^{n}_{i})}{\Delta_{-}u_{i}^{n}},\; 0<\lambda
  f'(u_{i}^{n})\leq 1$ as in \cite{rkdejde1}. In the persent work the
  coefficient $\mathcal{D}$ comes out in such a form which results
  in to a non-conservative approximation.}  half incremental form as,
    \begin{equation} \label{lmptvd} u_{i}^{n+1} =
      \left\{ \begin{array}{ll} u_{i}^{n} - \mathcal{D}
          (u_{i}^{n}-u_{i-1}^{n}) &\;\; \mbox{if}\;\;\;\;\;\;\;\;\; 0<
          \lambda a^n_{i+\frac{1}{2}}\leq 1,\\ u_{i}^{n}+\mathcal{D}
          (u_{i+1}^{n} - u_{i}^{n}) &\;\; \mbox{if} \;\;\;\;\; -1\leq
          \lambda a^n_{i+\frac{1}{2}}< 0, \end{array}\right.       
    \end{equation} 
    where from Definition \ref{def2}, $0\leq \mathcal{C}\leq 1$ and
    $0\leq \mathcal{D}\leq 1$. On appropriately choosing one of the
    coefficients $\alpha$ or $\beta$ zero in I-form (\ref{iform}), the
    Lemma \ref{lem1} shows approximation by half I-form (\ref{lmptvd})
    is global TVD. 
  \end{proof}

  \begin{remark} The LMP stability is defined only in non-sonic region
    therefore by Lemma \ref{lem2}, LMP stability implies TVD stability
    in non-sonic region i.e., away from sonic point $f^{'}(u^{*}) \neq 0$. 
    In this setting, from next section onward until stated the term LMP/TVD
    stability implies global TVD stability (\ref{tvddef}) {\bf away
      from sonic point}.
\end{remark}
\section{Bounds on high order TVD accuracy}\label{sec4}
In this section using definition \ref{def2}, it is shown that second
order total variation diminishing approximation is possible for the
solution with non-sonic extreme critical points. It follows from the
LMP stability bounds given for the representative second order
accurate Lax-Wendroff (LxW), Beam-Warming (BW) and Fromm schemes
respectively for scalar problem (\ref{nonlin}). Let the stencil
$[x_{i-s},x_{i+r}]$ of $r+s+1$ point scheme locally does not contain
sonic point $u^{*}(x,t)$ i.e., $f^{'}(u^*)\neq0$ and characteristics
speed at local cell interfaces is non-zero. Note that the case of
degenerate characteristic speed i.e., $a_{i+\frac{1}{2}}=0$ or/and
$a_{i-\frac{1}{2}}=0$ are not interesting as these schemes do not
necessarily preserve their uniform order of accuracy.  We also
consider the wave speed slpit $a^{+}_{i+\frac{1}{2}} +
a^{-}_{i+\frac{1}{2}}= a_{i+\frac{1}{2}}$ such that
$a^{+}_{i+\frac{1}{2}}\geq0$ and $a^{-}_{i+\frac{1}{2}}\leq 0$. Note
that for $a_{i+\frac{1}{2}}>0 \Rightarrow
a_{i+\frac{1}{2}}=a^{+}_{i+\frac{1}{2}}>0, a^{-}_{i+\frac{1}{2}}=0$
whereas $a_{i+\frac{1}{2}}<0 \Rightarrow
a_{i+\frac{1}{2}}=a^{-}_{i+\frac{1}{2}}<0, a^{+}_{i+\frac{1}{2}}=0,$.
After dropping the superscript for time level $n$, following function
definitions and notations are used in the rest of the presentation.
Define the smoothness parameter as
\begin{equation}
  \displaystyle r^{\pm}_{i} = r^{\pm}(F_{i})= \frac{\left(1\mp \lambda a^{\pm}_{i\mp\frac{1}{2}}\right)\Delta_{\mp}F^{\pm}_{i}}{\left(1 \mp
      \lambda  a^{\pm}_{i\pm\frac{1}{2}}\right)\Delta_{\pm}F^{\pm}_{i}},\label{gradratio}
\end{equation}
where the flux split $F^{+}_{i}+F^{-}_{i}=F_{i}$ is consistent with wave split and given by
\begin{equation}
F^{\pm}_{i+1} -F^{\pm}_{i}= a^{\pm}_{i+\frac{1}{2}}(u_{i+1}-u_{i}). \label{fluxsplit}
\end{equation} 
Here the superscript $\pm$ sign of $r_{i}$ denotes the
positive/negative sign of wave speed. Also define the signum function
\begin{equation}
  \sigma(x) = sgn(x) =\left\{\begin{array}{ll} +1 & \;if\; x\geq0,\\-1 & \;if\; x< 0.\end{array}\right.
\end{equation} 
In order to analyze the local non-linear
stability of considered schemes we choose practically viable CFL like
condition
\begin{equation}
0<\lambda \max_{u}|f^{'}(u)|<1 \label{cflNo}
\end{equation}
Note that the choice $\displaystyle \lambda \max_{u}|f^{'}(u)|=0\Rightarrow f^{'}(u)$ 
corresponds to the case of degenerate characteristic speed or steady
state case.  
\subsection{Centered Lax-Wendroff scheme}
\begin{theorem}\label{thm1}
  Away from sonic point and under CFL condition (\ref{cflNo}), the
  second order accurate Lax-Wendroff scheme for scalar conservation
  law (\ref{nonlin}) is TVD in the solution region where

\[\displaystyle
r^{\pm}_{i}\in \left(-\infty,\, \kappa_{1}\left(\lambda
    a^{\pm}_{i\mp\frac{1}{2}}\right)\right)\cup
\left(\gamma_{1}\left(\lambda
    a^{\pm}_{i\mp\frac{1}{2}}\right),\,\infty\right)\]
where numbers $\kappa_{1} < 0$, $\gamma_{1}> 0$ depends on CFL number
for linear stability given by
\begin{equation}
  \kappa_{1}(x) = -\frac{1-
    x\sigma(x)}{1+x\sigma(x))}, \label{kappa1}
\end{equation}
\begin{equation}
  \gamma_{1}(x) = \frac{x \sigma(x)}{2+ x\sigma(x)}.
  \label{gamma1}
\end{equation}
\end{theorem}
\begin{proof}
  Consider the numerical flux function of Lax-Wendroff (LxW) scheme
  \begin{equation}
    F^{n,LxW}_{i+ \frac{1}{2}}= \frac{1}{2}\left(F_{i+1} + F_{i}\right) -
    \frac{\lambda\,a^2_{i+ \frac{1}{2}}}{2} \Delta_{+} {u}_{i}. \label{flxw}
  \end{equation}
  In order to ensure non sonic region, let the characteristics speed is
  locally non-zero. Since LxW uses three point centred stencil $[x_{i-s},x_{i+r}],
  r=s=1$, it suffice to assume $a_{i+\frac{1}{2}}\times a_{i-\frac{1}{2}}>0.$
\begin{itemize}
\item[] {\bf Case $f^{'}({u})> 0$:} Let $ a_{i \pm \frac{1}{2}}> 0$,
  then the conservative approximation using (\ref{flxw}) can be
  written as
  \begin{equation} {u}^{n+1}_{i}= {u}_{i} - \left[\frac{\lambda
        a_{i+\frac{1}{2}}}{2}\left(1-\lambda\,
        a_{i+\frac{1}{2}}\right) \Delta_{+}{u}_{i}+ \frac{\lambda\,
        a_{i-\frac{1}{2}}}{2}\left(1+\lambda\,
        a_{i-\frac{1}{2}}\right) \Delta_{-}{u}_{i} \right],
  \end{equation}
  which can be written in the following non-conservative half Incremental form (\ref{lmptvd}),
  \begin{equation} {u}^{n+1}_{i}= {u}_{i} - \left[\frac{\lambda
        a_{i+\frac{1}{2}}}{2}\left(1-\lambda\,
        a_{i+\frac{1}{2}}\right) \frac{\Delta_{+}{u}_{i}}
      {\Delta_{-}{u}_{i}} +\frac{\lambda\,
        a_{i-\frac{1}{2}}}{2}\left(1+\lambda\,
        a_{i-\frac{1}{2}}\right) \right]
    \Delta_{-}{u}_{i}. \label{Iflxw1}
  \end{equation}
  From Lemma \ref{lem2}, half I-from (\ref{Iflxw1}) will be TVD if,
  \begin{equation} 0\leq \left[\lambda
      a_{i+\frac{1}{2}}\left(1-\lambda\, a_{i+\frac{1}{2}}\right)
      \frac{\Delta_{+}{u}_{i}} {\Delta_{-}{u}_{i}} +\lambda\,
      a_{i-\frac{1}{2}}\left(1+\lambda\, a_{i-\frac{1}{2}}\right)
    \right] \leq 2,
  \end{equation}
  which reduces to,
  \begin{equation}
    - \lambda\, a_{i-\frac{1}{2}}\left(1+\lambda\,
      a_{i-\frac{1}{2}}\right) \leq \lambda\, a_{i+\frac{1}{2}}\left(1-\lambda\,
      a_{i+\frac{1}{2}}\right) \frac{\Delta_{+}{u}_{i}}
    {\Delta_{-}{u}_{i}} \leq 2- \lambda\, a_{i-\frac{1}{2}}\left(1+\lambda\,
      a_{i-\frac{1}{2}}\right)\label{inq1}
  \end{equation}
  Note that $\lambda a_{i-\frac{1}{2}}(1-\lambda\,
  a_{i-\frac{1}{2}})>0$ under discrete CFL condition,
  \begin{equation}
    0<\lambda\, \max_{i}{a_{i+\frac{1}{2}}} < 1, \label{cfl1}
  \end{equation} 
  Hence inequality (\ref{inq1}) can be written as
  \begin{equation}
    -\frac{(1+\lambda\, a_{i-\frac{1}{2}})}{(1-\lambda\, a_{i-\frac{1}{2}})}
    \leq
    \frac{a_{i+\frac{1}{2}}(1-\lambda\,
      a_{i+\frac{1}{2}})}{a_{i-\frac{1}{2}}(1-\lambda\,
      a_{i-\frac{1}{2}})} \frac{\Delta_{+}{u}_{i}}
    {\Delta_{-}{u}_{i}} \leq 
    \frac{2-\lambda  a_{i-\frac{1}{2}}(1+\lambda\, a_{i-\frac{1}{2}})}
    {\lambda\, a_{i-\frac{1}{2}}(1-\lambda\, a_{i-\frac{1}{2}})}
  \end{equation}
  or
  \begin{equation}
    -\frac{(1+\lambda\, a_{i-\frac{1}{2}})}{(1-\lambda\, a_{i-\frac{1}{2}})}
    \leq
    \frac{a_{i+\frac{1}{2}}(1-\lambda\,
      a_{i+\frac{1}{2}})}{a_{i-\frac{1}{2}}(1-\lambda\,
      a_{i-\frac{1}{2}})} \frac{\Delta_{+}{u}_{i}}
    {\Delta_{-}{u}_{i}} \leq 
    \frac{2+\lambda\, a_{i-\frac{1}{2}}}{\lambda  a_{i-\frac{1}{2}}},
    \label{inq2b}
  \end{equation}
  Using definition (\ref{speed}) and flux wave split (\ref{fluxsplit}), Inequality (\ref{inq2b}) becomes
  \begin{equation}
    -\frac{(1+\lambda\, a^{+}_{i-\frac{1}{2}})}{(1-\lambda\, a^{+}_{i-\frac{1}{2}})}
    \leq
    \frac{(1-\lambda\,
      a^{+}_{i+\frac{1}{2}})}{(1-\lambda\,
      a^{+}_{i-\frac{1}{2}})} \frac{\Delta^{+}_{+}{F^{+}}_{i}}
    {\Delta^{+}_{-}{F^{+}}_{i}} \leq \frac{2+\lambda\, a^{+}_{i-\frac{1}{2}}}{\lambda  a^{+}_{i-\frac{1}{2}}}.    \label{inq2c}
  \end{equation}
  Inequality (\ref{inq2c}) on inversion yields,
  \begin{equation}
    r^{+}_{i}< \kappa^{+}_{1}(\lambda\, a^{+}_{i-\frac{1}{2}})
    \; \mbox{OR}\; r^{+}_{i} >  
    \gamma^{+}_{1}(\lambda a^{+}_{i-\frac{1}{2}}), \label{inq2d}
  \end{equation}
  where $$\displaystyle r^{+}_{i}
  =\frac{(1-\lambda\, a^{+}_{i-\frac{1}{2}})
    \Delta^{+}_{-}{F}_{i}}{(1-\lambda\,
    a^{+}_{i+\frac{1}{2}})\Delta^{+}_{+}{F}_{i}},\; \kappa^{+}_{1}(\lambda
  a^{+}_{i-\frac{1}{2}})= -\frac{(1-\lambda\,
    a^{+}_{i-\frac{1}{2}})}{(1+\lambda\, a^{+}_{i-\frac{1}{2}})}\;
  \mbox{and}\;\displaystyle \gamma^{+}_{1}(\lambda
  a^{+}_{i-\frac{1}{2}})=\frac{\lambda\, a^{+}_{i-\frac{1}{2}}}{2+\lambda a^{+}_{i-\frac{1}{2}}}.$$\\
  \item[] {\bf Case ${f^{'}(u)<0}$:} Let $a_{i\pm\frac{1}{2}}<0$, then the non-conservative I-form can be written
  as
  \begin{equation} {u}^{n}_{i} = {u}_{i} - \left[\frac{\lambda\,
        a_{i+\frac{1}{2}}}{2}\left(1 -\lambda\,
        a_{i+\frac{1}{2}}\right) + \frac{\lambda\,
        a_{i-\frac{1}{2}}}{2}\left(1+\lambda\,
        a_{i-\frac{1}{2}}\right)\frac{\Delta_{-}{u}_{i}}
      {\Delta_{+}{u}_{i}}\right] \Delta_{+}{u}_{i}.\label{Iflxw2}
  \end{equation}
  Using Lemma \ref{lem2}, half I-from (\ref{Iflxw2}) will be TVD if,
  \begin{equation}
    0 \leq -\frac{\lambda\,
      a_{i+\frac{1}{2}}}{2}\left(1 -\lambda\, a_{i+\frac{1}{2}}\right)
    - \frac{\lambda\,
      a_{i-\frac{1}{2}}}{2}\left(1+\lambda\,
      a_{i-\frac{1}{2}}\right)\frac{\Delta_{-}{u}_{i}}
    {\Delta_{+}{u}_{i}}\leq 1
  \end{equation}
  which can be written as
  \begin{equation}
    \lambda\,
    a_{i+\frac{1}{2}}\left(1 -\lambda\, a_{i+\frac{1}{2}}\right)
    \leq - \lambda\,
    a_{i-\frac{1}{2}}\left(1+\lambda\,
      a_{i-\frac{1}{2}}\right)\frac{\Delta_{-}{u}_{i}}
    {\Delta_{+}{u}_{i}}\leq 2 + \lambda\,
    a_{i+\frac{1}{2}}\left(1 -\lambda\, a_{i+\frac{1}{2}}\right). \label{inq3}
  \end{equation}
The discrete CFL condition for $a(u)<0$ is,
  \begin{equation}
    -1 < \lambda a_{i+\frac{1}{2}}<0,\, \forall i.\label{cfl2}
  \end{equation}
  Therefore the quantity $-\lambda a_{i+\frac{1}{2}}(1+ \lambda
  a_{i+\frac{1}{2}})>0.$ Divide Inequality (\ref{inq3}) by it and
  using (\ref{speed}) yields,
  \begin{equation}
    -\frac{\left(1 -\lambda\, a_{i+\frac{1}{2}}\right)}{\left(1+ \lambda a_{i+\frac{1}{2}}\right)}
    \leq  \frac{\left(1+\lambda\,
        a_{i-\frac{1}{2}}\right)}{\left(1+ \lambda a_{i+\frac{1}{2}}\right)}\frac{\Delta_{-}F_{i}}
    {\Delta_{+}F_{i}}\leq -\frac{2 + \lambda\,
      a_{i+\frac{1}{2}}\left(1 -\lambda\,
        a_{i+\frac{1}{2}}\right)}{\lambda a_{i+\frac{1}{2}}\left(1+ \lambda
        a_{i+\frac{1}{2}}\right)}
  \end{equation}
  or using flux wave split (\ref{fluxsplit})
  \begin{equation}
    -\frac{\left(1 -\lambda\, a^{-}_{i+\frac{1}{2}}\right)}{\left(1+ \lambda a^{-}_{i+\frac{1}{2}}\right)}
    \leq  \frac{\left(1+\lambda\,
        a^{-}_{i-\frac{1}{2}}\right)}{\left(1+ \lambda a^{-}_{i+\frac{1}{2}}\right)}\frac{\Delta_{-}F^{-}_{i}}
    {\Delta_{+}F^{-}_{i}}\leq \frac{\lambda\,a^{-}_{i+\frac{1}{2}}-2}{\lambda\;a^{-}_{i+\frac{1}{2}}}
    \label{inq3a}
  \end{equation}
  Inequality (\ref{inq3a}) on inversion yields,
  \begin{equation}
    r^{-}_{i}<\kappa_{1}\left(\lambda a^{-}_{i+\frac{1}{2}}\right)\;
    \mbox{OR}\; r^{-}_{i} > \gamma_{1}\left(\lambda a^{-}_{i+\frac{1}{2}}\right). \label{inq3b}
  \end{equation} 
  where $\displaystyle r_{i}^{-}=\frac{\left(1+ \lambda
      a^{-}_{i+\frac{1}{2}}\right)\Delta_{+}F^{-}_{i}}{\left(1+\lambda\,
      a^{-}_{i-\frac{1}{2}}\right)\Delta_{-}F^{-}_{i}},\;
  \kappa_{1}\left(\lambda a^{-}_{i+\frac{1}{2}}\right)=
  \frac{-\left(1+\lambda a^{-}_{i+\frac{1}{2}}\right)}{\left(1-\lambda
      a^{-}_{i+\frac{1}{2}}\right)}$ and $\displaystyle
  \gamma_{1}\left(\lambda a^{-}_{i+\frac{1}{2}}\right)= \frac{-\lambda
    a^{-}_{i+\frac{1}{2}}}{2 - \lambda\, a^{-}_{i+\frac{1}{2}}}$.\\
  Condition (\ref{inq2d}) and (\ref{inq3b}) completes the proof.
\end{itemize}
\end{proof}
\subsubsection{LxW on Linear problem: Every extrema is non-sonic.}
In order to see the improvement in the TVD approximation at non-sonic
extrema, consider the linear transport equation
\begin{equation}
  \frac{\partial}{\partial t}u(x,t) + a\frac{\partial}{\partial u}(\,u(x,t)) =0,\; a\neq0  \label{transport}
\end{equation}
In this case the smoothness parameter (\ref{gradratio}) reduces to $
r^{\pm}_{i} = \displaystyle \frac{\Delta_{\mp} u_{i}}{\Delta_{\pm}
  u_{i}}$. Note that at point of extrema the measure of smoothness is
negative i.e, for transport equation (\ref{transport}), every
$r^{\pm}_{i} <0$ implies a {\it non-sonic extreme critical
  point}. Following result follows from Theorem \ref{thm1}
\begin{cor}\label{cor1}
  Under the linear stability condition $0<\lambda |a|\leq 1$, the
  second order accurate Lax-Wendroff scheme for (\ref{transport}) is
  total variation diminishing where $r_{i} \in
  (\infty,\,\kappa_{1}(a\lambda))\cup(\gamma_{1}(a\lambda),\,\infty)$.
\end{cor}
\begin{figure}[!htb]
  \begin{center}
    \begin{tabular}{cc}
      \includegraphics[%
      scale=0.55]{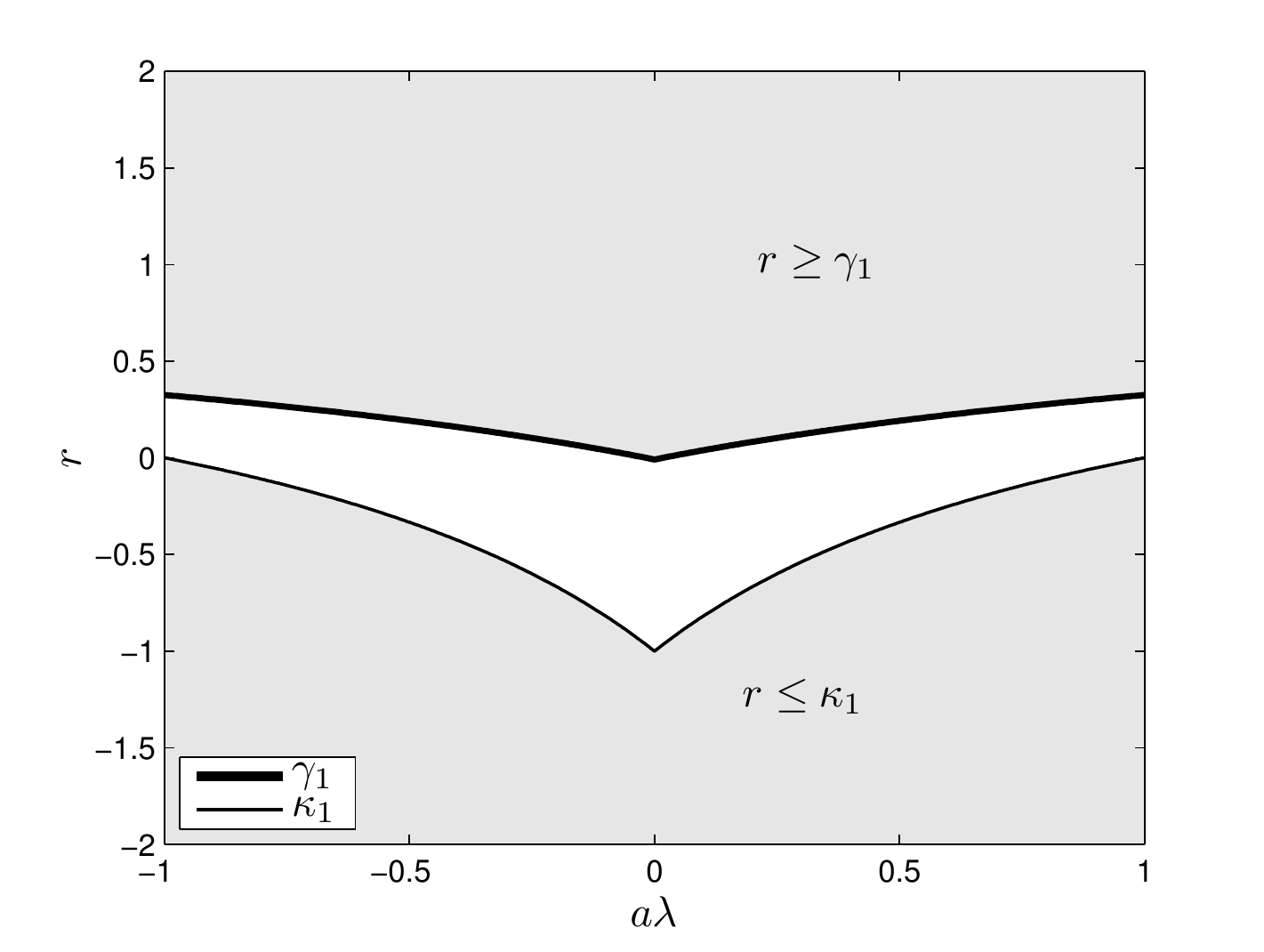}&\includegraphics[%
      scale=0.55]{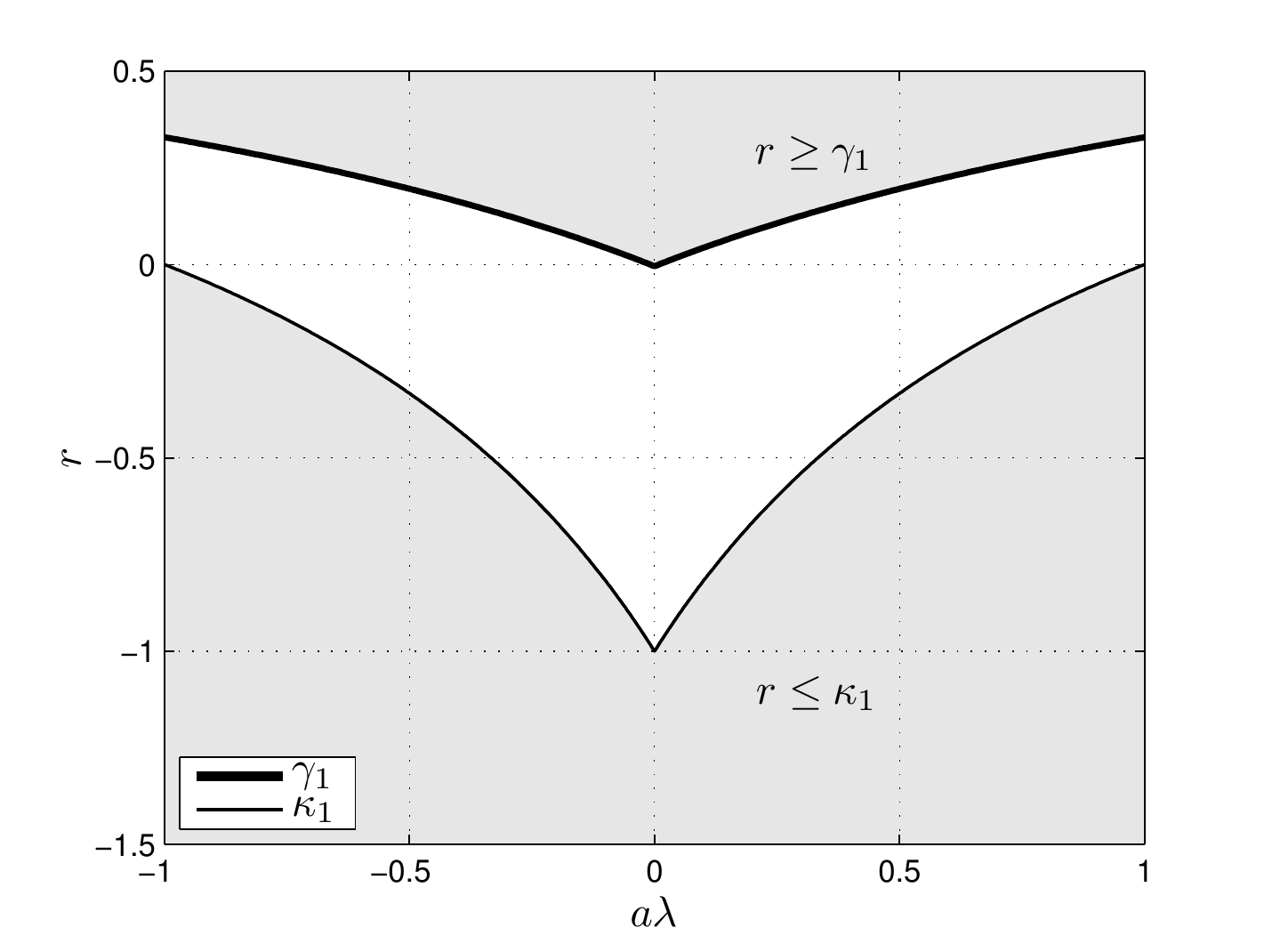}\\
      (a) & (b)
    \end{tabular}
    \caption{\label{kappagammaLxW}(a) : LMP/TVD stability region (shaded) for
      Lax-Wendroff scheme (b): Zoomed view.}
  \end{center}
\end{figure}
In Figure \ref{kappagammaLxW}, the behavior of CFL number $a\lambda$
on dependent parameters $\kappa_{1}$ and $\gamma_{1}$ is shown.  Note
that when $\lambda a \rightarrow 0^{+}$ the parameter $\gamma_{1}
\rightarrow 0^{+}$ whilst when $\lambda a \rightarrow 1^{-}$ the
parameter $\kappa_{1} \rightarrow 0^{-}$. In particular, for
$a|\lambda|=1$, definitions (\ref{kappa1}), (\ref{gamma1}) yield
non-TVD interval $[\kappa_{1}=0, \gamma_{1}=\frac{1}{3}]$.  Note that
under linear CFL condition $0<\lambda a\leq1$, $\kappa_{1}(\lambda\,a
) \in (-1,0]$ and $ \gamma_{1}(\lambda\,a) \in (0,1/3)$. The following
result give CFL independent TVD bounds for LxW scheme
\cite{riteshEJDE}.
\begin{cor}\label{cor2}
  The Lax-Wendroff scheme for (\ref{transport}) is total variation
  diminishing under the linear stability condition $0<\lambda |a|\leq1,$
  if $r_{i} =\in (\infty,\,-1)\cup(\frac{1}{3},\,\infty)$
\end{cor}
Thus it, can also be concluded from corollary \ref{cor1} and shaded
TVD region for LxW scheme in Figure \ref{kappagammaLxW} that {\it
  except for $r_{i} \in [\kappa_{1},\gamma_{1}]$ the second order
  accurate LxW scheme yields TVD approximation for all solution region
  including extreme points with $r_{i}<0$.}  More precisely,
\subsection{Upwind Beam-Warming scheme}
\begin{theorem}\label{thm2}
  Away from sonic point and under CFL condition (\ref{cflNo}), second
  order accurate Beam-Warming scheme is TVD for scalar conservation
  law (\ref{nonlin}) in the solution region where \[ r^{\pm}_{i \mp 1}
  \in \left[ \kappa_{2}\left(\lambda a^{\pm}_{i\mp\frac{1}{2}}\right),
    \gamma_{2}\left(\lambda a^{\pm}_{i\mp\frac{1}{2}}\right)\right],\]
  where parameter $\gamma_{2}$ and $\kappa_{2}$ defined as,
  \begin{equation}
    \kappa_{2}(x) = \displaystyle \frac{-(2 -x \sigma(x))}{x\sigma(x)} \label{kappa2}
  \end{equation}
  \begin{equation}
    \gamma_{2}(x) = \displaystyle \frac{3-x\sigma(x)}{1-x\sigma(x)}\label{gamma2}
  \end{equation}
\end{theorem}
\begin{proof}
\begin{itemize}
\item[]{\bf Case $f^{'}(u)>0$:} In this case, BW stencil use grid points
  $[x_{i-s}, x_{i+r}], s=2, r=0$ thus to ensure locally non-sonic
  region it suffice to let $a_{i-\frac{k}{2}} >0, k=-1,1,3$. The
  numerical flux of Beam-Warming scheme is,
  \begin{equation}
    F^{n,BW}_{i+\frac{1}{2}}= F_{i} +
    \frac{a_{i-\frac{1}{2}}}{2}\left(1- \lambda\,a_{i-\frac{1}{2}}\right)
    \Delta_{-}{u}_{i},\;\; a_{i+\frac{1}{2}}>0. \label{bwflx1}
  \end{equation}
  Resulting non-conservative half I-form can be written as
  \begin{equation} {u}^{n+1}_{i} = {u}_{i} -\left[\frac{\lambda\,
        a_{i-\frac{1}{2}}}{2}\left(3-\lambda\,
        a_{i-\frac{1}{2}}\right) - \frac{\lambda\,
        a_{i-\frac{3}{2}}}{2}\left(1- \lambda\,
        a_{i-\frac{3}{2}}\right)\frac{\Delta_{-}{u}_{i-1}}{\Delta_{-}{u}_{i}}
    \right]\Delta_{-}{u}_{i} \label{Ifbw1}
  \end{equation}
  For half I-from (\ref{Ifbw1}) to be TVD, from Lemma \ref{lem2}
  \begin{equation}
    0 \leq \frac{\lambda\,a_{i-\frac{1}{2}}}{2}\left(3 -
      \lambda\,a_{i-\frac{1}{2}}\right) - \frac{\lambda\,a_{i-\frac{3}{2}}}{2}\left(1 -
      \lambda\,a_{i-\frac{3}{2}}\right)\frac{\Delta_{-}{u}_{i-1}}
    {\Delta_{-}{u}_{i}}
    \leq 1. \label{inq5}
  \end{equation}
  Compound Inequality (\ref{inq5}) can be written as,
  \begin{equation}
    \lambda a_{i-\frac{1}{2}}(3-\lambda a_{i-\frac{1}{2}})-2 \leq 
    \lambda a_{i-\frac{3}{2}}(1-\lambda
    a_{i-\frac{3}{2}})\frac{\Delta_{-}{u}_{i-1}}{\Delta_{-}{u}_{i}}
    \leq
    \lambda
    a_{i-\frac{1}{2}}(3-\lambda a_{i-\frac{1}{2}})  
  \end{equation}

  Under CFL condition (\ref{cfl1}), i.e., $0<\lambda
  a_{i+\frac{1}{2}}< 1,\; \forall i$ quantity $\lambda
  a_{i-\frac{1}{2}}(1-\lambda a_{i-\frac{1}{2}})$ is positive, hence
  (\ref {inq5}) can be written as,
  \[
  \frac{\lambda a_{i-\frac{1}{2}}(3-\lambda
    a_{i-\frac{1}{2}})-2}{\lambda a_{i-\frac{1}{2}}(1-\lambda
    a_{i-\frac{1}{2}})} \leq \frac{a_{i-\frac{3}{2}}(1-\lambda
    a_{i-\frac{3}{2}})}{ a_{i-\frac{1}{2}}(1-\lambda
    a_{i-\frac{1}{2}})}\frac{\Delta_{-}{u}_{i-1}}{\Delta_{-}{u}_{i}}
  \leq \frac{ (3-\lambda a_{i-\frac{1}{2}})}{(1-\lambda
    a_{i-\frac{1}{2}})}
  \]
  or
  \[
  \frac{\lambda a_{i-\frac{1}{2}}-2}{\lambda a_{i-\frac{1}{2}}} \leq
  \frac{a_{i-\frac{3}{2}}(1-\lambda a_{i-\frac{3}{2}})}{
    a_{i-\frac{1}{2}}(1-\lambda
    a_{i-\frac{1}{2}})}\frac{\Delta_{-}{u}_{i-1}}{\Delta_{-}{u}_{i}}
  \leq \frac{ (3-\lambda a_{i-\frac{1}{2}})}{(1-\lambda
    a_{i-\frac{1}{2}})}
  \]
On using flux wave split (\ref{fluxsplit}) we get
  \[
  \frac{\lambda a^{+}_{i-\frac{1}{2}}-2}{\lambda
    a^{+}_{i-\frac{1}{2}}} \leq \frac{(1-\lambda
    a_{i-\frac{3}{2}})}{(1-\lambda
    a_{i-\frac{1}{2}})}\frac{\Delta_{-}F^{+}_{i-1}}{\Delta_{-}F^{+}_{i}}
  \leq \frac{ (3-\lambda a^{+}_{i-\frac{1}{2}})}{(1-\lambda
    a^{+}_{i-\frac{1}{2}})}
  \]
  Which, using (\ref{gradratio}) can be written as,
  \begin{equation}
    \kappa_{2}\left(\lambda\,a^{+}_{i-\frac{1}{2}}\right) \leq r^{+}_{i-1} \leq
    \gamma_{2}\left(\lambda\,a^{+}_{i-\frac{1}{2}}\right)  \label{inq6}
  \end{equation}
  where $\kappa_{2}$ and $\gamma_{2}$ are defined in (\ref{kappa2}) and (\ref{gamma2}).\\
  \item[] {\bf Case $f^{'}(u)<0$:} In case of negative characteristics
    speed, BW use stencil $[x_{i-s}, x_{i+r}], s=0, r=2]$ thus to
      ensure locally non-sonic region it suffice to let
      $a_{i+\frac{k}{2}} >0, k=-1,1,3$. In case of negative wave speed
      the Beam-Warming flux is given by,
  \begin{equation}
    F^{n,BW}_{i+\frac{1}{2}} = F_{i+1} -
    \frac{a_{i+\frac{3}{2}}}{2}\left(1+
      \lambda\,a_{i+\frac{3}{2}}\right)\Delta_{+}{u}_{i+1},\;\; a_{i+\frac{1}{2}}>0.\label{bwflx2}
  \end{equation}
  Resulting non-conservative half I-form is,
  \begin{equation} {u}^{n+1}_{i}= {u}_{i} +\left[
      \frac{\lambda\,a_{i+\frac{3}{2}}}{2}\left(1 +
        \lambda\,a_{i+\frac{3}{2}}\right)\frac{\Delta_{+}{u}_{i+1}}{\Delta_{+}{u}_{i}}
      -
      \frac{\lambda\,a_{i+\frac{1}{2}}}{2}\left(3 +
        \lambda\,a_{i+\frac{1}{2}}\right) \right]\Delta_{+}{u}_{i}.\label{Ifbw2}
  \end{equation}
  Condition for (\ref{Ifbw2}) to be TVD is
  \begin{equation}
    0\leq \lambda\,a_{i+\frac{3}{2}}\left(1 +
      \lambda\,a_{i+\frac{3}{2}}\right)\frac{\Delta_{+}{u}_{i+1}}{\Delta_{+}{u}_{i}} -
    \lambda\,a_{i+\frac{1}{2}}\left(3 +\lambda\,a_{i+\frac{1}{2}}\right)  \leq 2 \label{inq7}
  \end{equation}

\begin{equation}
  \lambda\,a_{i+\frac{1}{2}}\left(3 +
    \lambda\,a_{i+\frac{1}{2}}\right) \leq \lambda\,a_{i+\frac{3}{2}}\left(1 +
    \lambda\,a_{i+\frac{3}{2}}\right)\frac{\Delta_{+}{u}_{i+1}}{\Delta_{+}{u}_{i}}
  \leq 2 +\lambda\,a_{i+\frac{1}{2}}\left(3 +
    \lambda\,a_{i+\frac{1}{2}}\right) 
  \label{inq8}
\end{equation}
Note under CFL condition (\ref{cfl2}) $-1\leq
\lambda\,a_{i+\frac{1}{2}} <0,\, \forall i$, $\lambda\,
a_{i+\frac{1}{2}}\left(1+\lambda\,a_{i+\frac{1}{2}}\right)$ is
negative. Compound Inequality (\ref{inq8}) reduced to
\begin{equation}
  \frac{2 +\lambda\,a_{i+\frac{1}{2}}\left(3
      +\lambda\,a_{i+\frac{1}{2}}\right)}{\lambda\,
    a_{i+\frac{1}{2}}\left(1+\lambda\,a_{i+\frac{1}{2}}\right)}
  \leq
  \frac{a_{i+\frac{3}{2}}\left(1 +
      \lambda\,a_{i+\frac{3}{2}}\right)}{ a_{i+\frac{1}{2}}\left(1+\lambda\,a_{i+\frac{1}{2}}\right)} 
  \frac{\Delta_{+}{u}_{i+1}}{\Delta_{+}{u}_{i}}
  \leq \frac{\left(3 +
      \lambda\,a_{i+\frac{1}{2}}\right)}{\left(1+\lambda\,a_{i+\frac{1}{2}}\right)}
\end{equation}
or
\begin{equation}
  \frac{\lambda\,a_{i+\frac{1}{2}}+2}{\lambda\,a_{i+\frac{1}{2}}}  \leq
  \frac{a_{i+\frac{3}{2}}\left(1 +
      \lambda\,a_{i+\frac{3}{2}}\right)}{ a_{i+\frac{1}{2}}\left(1+\lambda\,a_{i+\frac{1}{2}}\right)} 
  \frac{\Delta_{+}{u}_{i+1}}{\Delta_{+}{u}_{i}}
  \leq \frac{\left(3 +
      \lambda\,a_{i+\frac{1}{2}}\right)}{\left(1+\lambda\,a_{i+\frac{1}{2}}\right)} 
  \label{inq9a}
\end{equation}
on using (\ref{fluxsplit})
\begin{equation}
  \frac{\lambda\,a^{-}_{i+\frac{1}{2}}+2}{\lambda\,a^{-}_{i+\frac{1}{2}}}  \leq
  \frac{\left(1 +
      \lambda\,a^{-}_{i+\frac{3}{2}}\right)}{ \left(1+\lambda\,a_{i+\frac{1}{2}}\right)} 
  \frac{\Delta_{+}F^{-}_{i+1}}{\Delta_{+}F^{-}_{i}}
  \leq \frac{\left(3 +
      \lambda\,a^{-}_{i+\frac{1}{2}}\right)}{\left(1+\lambda\,a^{-}_{i+\frac{1}{2}}\right)} 
  \label{inq9}
\end{equation}

Using (\ref{gradratio}), (\ref{kappa2}) and (\ref{gamma2}) Inequality
(\ref{inq9}) can be written as,
\begin{equation}
  \kappa_{2}\left(\lambda\, a^{-}_{i+\frac{1}{2}}\right)
  \leq r^{-}_{i+1} \leq 
  \gamma_{2}\left(\lambda\, a^{-}_{i+\frac{1}{2}}\right)
\end{equation}
\end{itemize}
\end{proof}
\subsubsection{BW on Linear problem}
\begin{cor}\label{cor3}
  Beam-Warming scheme for (\ref{transport}) is total variation
  diminishing under the linear stability condition $0<\lambda |a|
  \leq1$, if the smoothness parameter $r^{\pm}_{i\mp 1} \in
  [\kappa_{2}(\lambda a),\gamma_{2}(\lambda a)] $
\end{cor}
\begin{figure}[!\htb]
  \begin{tabular}{cc}
    \includegraphics[%
    scale=0.55]{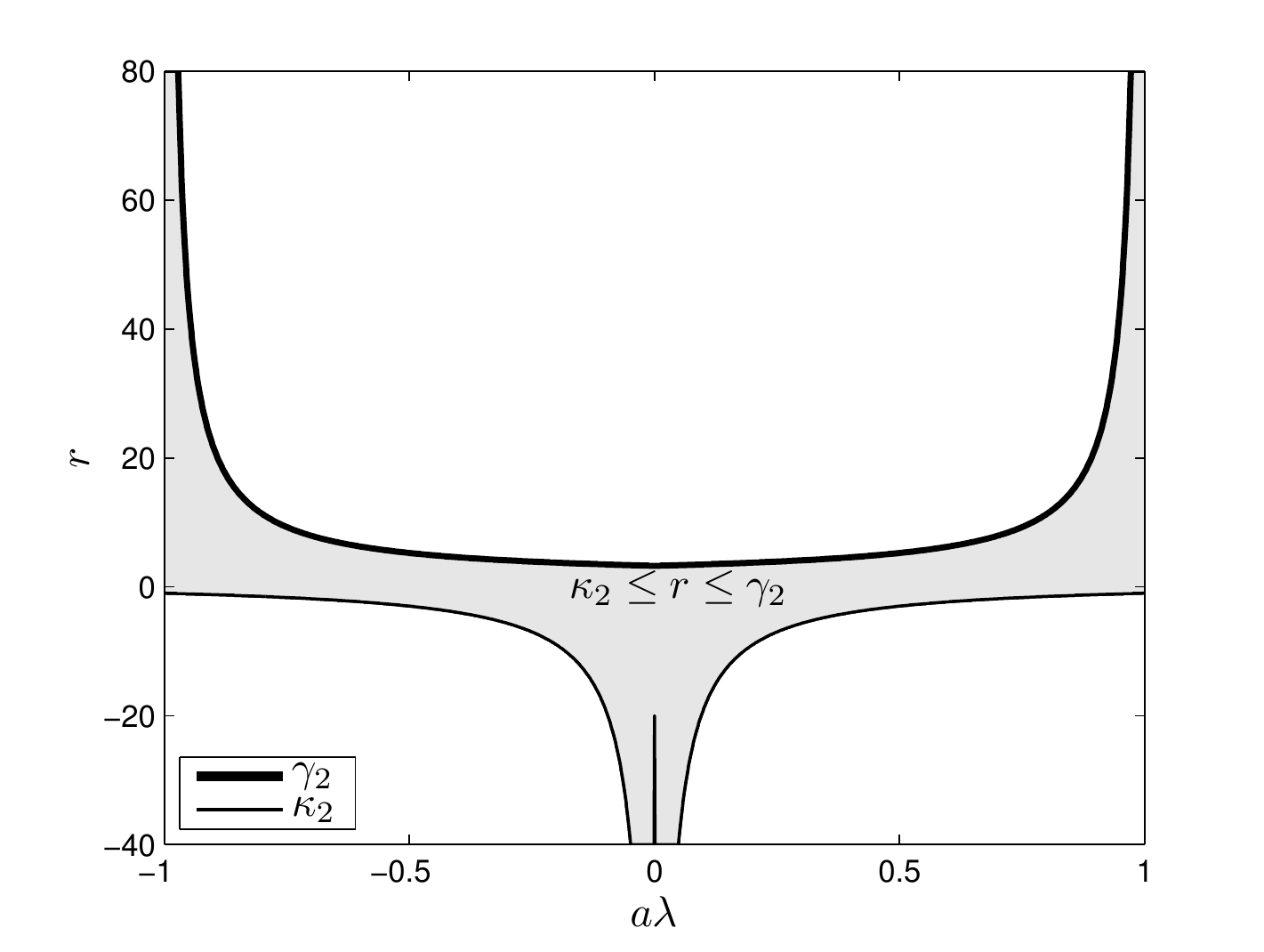}& \includegraphics[%
    scale=0.55]{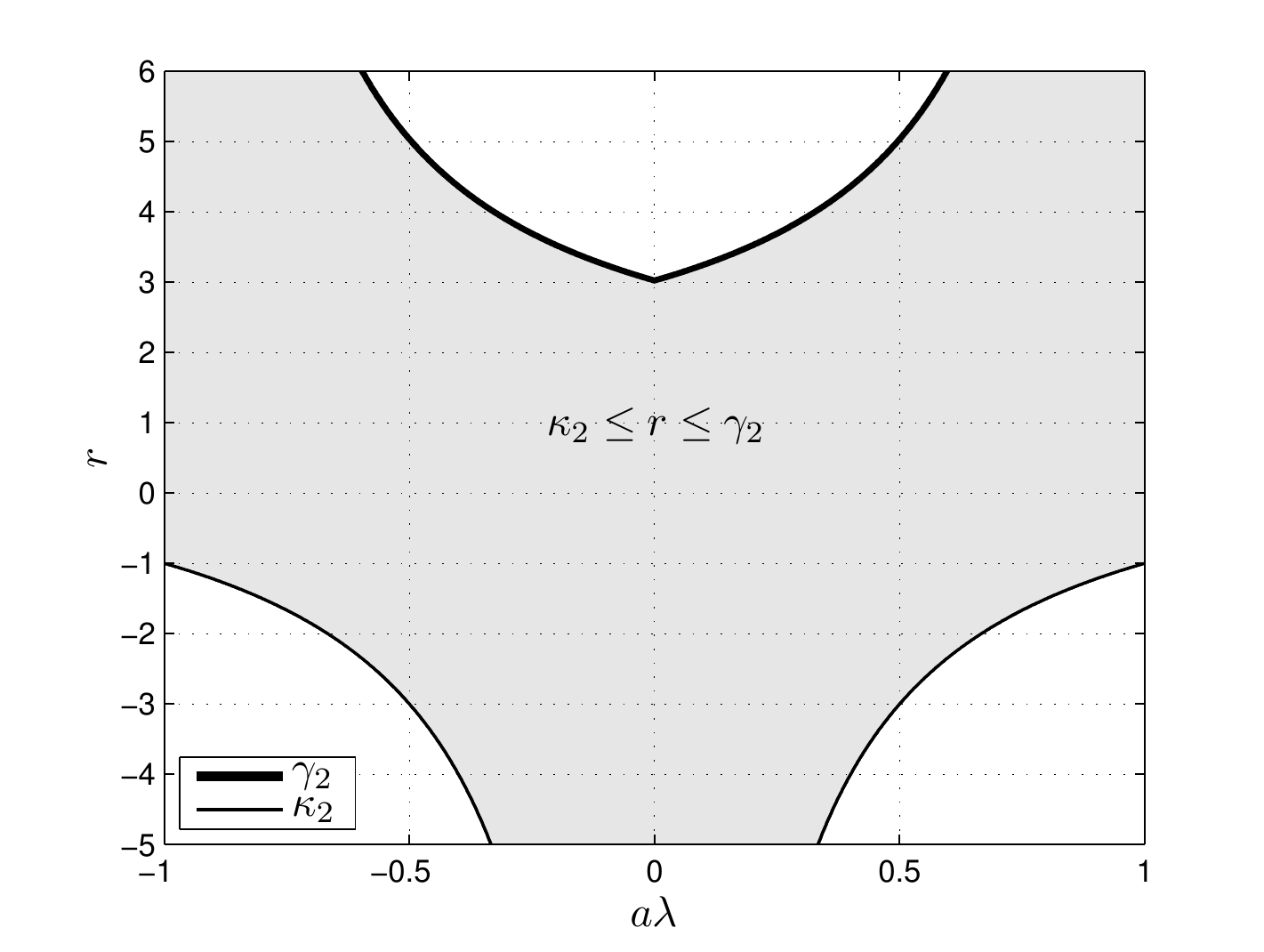}\\
    (a) & (b)
  \end{tabular}
  \caption{\label{kappagammaBW} (a) : LMP/TVD stability region (shaded) for
    Beam-Warming scheme (b): Zoomed view.}
\end{figure}
In Figure \ref{kappagammaBW}, TVD region for Beam-Warming scheme is
shown as shaded region. It can be deduced that
\begin{enumerate}
\item $\lambda a\rightarrow 0^{+}$ Beam-Warming scheme yield TVD
  approximation for $r_{i-\sigma(a)} \in (-\infty, 3]$.
\item $\lambda a\rightarrow 1^{-}$ Beam-Warming scheme yield TVD
  approximation for $r_{i-\sigma(a)} \in [-1,\infty)$.
\end{enumerate}
Note that for $\lambda |a| \in (0,1]$ parameters $\kappa_{2}\in
(-\infty, -1]$ and $\gamma_{2} \in [3,\infty)$. Following CFL number
independent weaker TVD bounds can be concluded 
\begin{cor}\label{cor4}
  The Beam-Warming scheme for (\ref{transport}) is total variation
  diminishing under the linear stability condition $0<\lambda |a|< 1,$
  if $r^{\pm}_{i \mp 1} =[-1,3]$.
\end{cor}

\subsection{Fromm's scheme}
A less ocsillatory and second order accurate scheme is obtained by using a simple
average of LxW and BW flux i.e.,
\begin{equation}
  F^{n,FROMM}_{i+\frac{1}{2}}= \frac{1}{2}\left(F^{n,LxW}_{i+\frac{1}{2}} + F^{n,BW}_{i+\frac{1}{2}}\right)\label{3rd}
\end{equation}
From Theorem \ref{thm1} and Theorem \ref{thm2} following result can be
proved,
\begin{theorem}\label{thm3}
  Away from sonic point and CFL condition (\ref{cflNo}), the Fromm's scheme corresponding
  to flux (\ref{3rd}) for scalar conservation law (\ref{nonlin}) is
  TVD in the solution region where
  \[\displaystyle 
  r^{\pm}_{i}\in \left(-\infty,\, \kappa_{1}\left(\lambda
      a^{\pm}_{i\mp\frac{1}{2}}\right)\right)\cup
  \left(\gamma_{1}\left(\lambda
      a^{\pm}_{i\mp\frac{1}{2}}\right),\,\infty\right)\]
  and
  \[ r^{\pm}_{i \mp1} \in \left[ \kappa_{2}\left(\lambda
      a^{\pm}_{i\mp\frac{1}{2}}\right),
    \gamma_{2}\left(\lambda
      a^{\pm}_{i\mp\frac{1}{2}}\right)\right].\] where
  parameter $\kappa_{1}, \gamma_{1},\kappa_{2}$ and $ \gamma_{2}$ are
  defined in (\ref{kappa1}), (\ref{gamma1}), (\ref{kappa2}) and
  (\ref{gamma2}) respectively.
\end{theorem}
\section{Hybrid high order LMP/TVD stable schemes} \label{sec5}
It follows from the Theorem \ref{thm1}, \ref{thm2} and Theorem
\ref{thm3} that it is possible to achieve second or higher order TVD
approximation for most solution region including non-sonic exterma
where $r_{i}<0$. In order to demonstrate it numerically, we construct
hybrid schemes using a monotone/TVD scheme as {\bf complementary conservative 
  scheme (CCS)}. The following hybrid schemes are the natural choice
which satisfies the LMP/TVD bounds obtained in previous section and
thus ensures a LMP/TVD approximation. The second order accurate
LMP/TVD schemes use second order LxW and BW schemes in the region of
their LMP/TVD stability using bounds on smoothness parameter in
Theorem \ref{thm1} and \ref{thm2} respectively, otherwise use a
conservative conservative scheme (CCS).
\subsection{Centered scheme: LW-CCS}\label{algo1}  
  \begin{algorithmic}[1]
    \IF{$r^{\pm}_{i}\leq \kappa_{1}$ OR  $r^{\pm}_{i}\geq \gamma_{1}$} \STATE Update $u_{i}^{n+1} \gets$
    LxW scheme 
    \ELSE \STATE Update $u_{i}^{n+1} \gets$ CCS.
    \ENDIF
  \end{algorithmic}
\subsection{Upwind Scheme: BW-CCS}\label{algo2}
  \begin{algorithmic}[1]
    \IF{($r^{\pm}_{i\mp 1}\geq \kappa_{2}$ AND
      $r^{\pm}_{i\mp 1}\leq \gamma_{2}$)} \STATE Update $u_{i}^{n+1}
    \gets$ BW scheme \ELSE \STATE Update $u_{i}^{n+1} \gets$ CCS.
    \ENDIF
  \end{algorithmic}
\subsection{Centred-Upwind scheme: FLWBW-CCS approximation}\label{algo3}
This LMP/TVD stable scheme can be obtained using Fromm scheme in
its region of LMP/TVD stability using bounds in \ref{thm3} along
with schemes \ref{algo1} and \ref{algo2} as follows
\begin{algorithmic}[1]
  \IF{($r^{\pm}_{i}\leq \kappa_{1}$ OR $r^{\pm}_{i}\geq \gamma_{1}$)
    \text{AND} ($r^{\pm}_{i\mp1}\geq \kappa_{2}$ AND
    $r^{\pm}_{i\mp1}\leq \gamma_{2}$)} \STATE Update $u_{i}^{n+1}
  \gets$ Fromm's scheme \ELSIF{$r^{\pm}_{i}\leq \kappa_{1}$ OR
    $r^{\pm}_{i}\geq \gamma_{1}$} \STATE Update $u_{i}^{n+1} \gets$
  LxW scheme \ELSIF{($r^{\pm}_{i\mp 1}\geq \kappa_{2}$ AND
    $r^{\pm}_{i\mp 1}\leq \gamma_{2}$)} \STATE Update $u_{i}^{n+1}
  \gets$ BW scheme \ELSE \STATE Update $u_{i}^{n+1} \gets$ CCS.
  \ENDIF
\end{algorithmic}
Note that the scheme \ref{algo1}-\ref{algo3} are non-conservative as
they are based on TVD conditions on the smoothness parameter dedced
from the non-conservative form of studied schemes\footnote{except for
  equations having constant characteristic speed e.g. linear transport
  problem.}. Therefore they capture the steady shock accurately but may
produce moving shock at wrong location see results in Figure
\ref{Fig6a}(a). Note that incorrect shock location by scheme
\ref{algo1} is legging whereas by scheme \ref{algo2} it is leading to
the exact shock location in Figure \ref{Fig6a}(a). It is interesting
to see the scheme \ref{algo3} cancels the leading and legging errors
and gives exact shock location in Figure \ref{Fig6b}(a).  This
phenomena of yielding wrong moving shock location by non-conservative
schemes along with the shock correction criteria is well explained in
\cite{Hou1994}. The idea for shock correction is to apply locally a
shock capturing conservative scheme in the vicinity of discontinuity
using a shock detector. It is therefore, to capture the moving shock
correctly, the following hybrid approach can be used,
\subsection{Shock Correction: SC-LW-CCS, SC-BW-CCS, SC-FLWBW-CCS hybrid schemes}\label{algo4}
  \begin{algorithmic}[1]
    \IF{Shock region} \STATE Update $u_{i}^{n+1} \gets$ use CCS,
    \ELSE \STATE Update $u_{i}^{n+1} \gets$
    with either of algorithm \ref{algo1}-\ref{algo3}.  \ENDIF
  \end{algorithmic}
\subsection{Extension to system of hyperbolic conservation
  laws} \label{algo4sys} Consider the hyperbolic systems of
conservation law in one dimensions,
\begin{equation}
  \frac{\partial}{\partial
    t}\textbf{u}+\frac{\partial}{\partial x}\textbf{F(u)}=0, \label{1Dsys}
\end{equation}
where $\textbf{u}$ is vector of conserved quantities $u^{j},\,
i=1,2,\dots l$ and $\textbf{F}$ is the vector flux function. The above
proposed schemes for scalar case are extended to non-linear systems
(\ref{1Dsys}) in the natural manner using flux vector splitting and
average flux Jacobian matrix $A=\textbf{F}^{'}\textbf{(u)}$ of the
flux function. In particular for the numerical results presented in
next section, the Steger-Warming flux vector splitting is used for 1D
and 2D systems. The average Jacobian matrix is computed as follows,
$$A^{n}_{i+\frac{1}{2}}=A\left(\frac{\textbf{u}^n_{i+1}+\textbf{u}^n_{i}}{2}\right).$$
In order to compute the TVD bounds, the $n-$characteristic speeds associated
with system \ref{1Dsys}) 
\begin{equation}
a^{j}_{i+\frac{1}{2}}=\left\{\begin{array}{cc}
\frac{F^{j}_{i+1}-F^{j}_{i}}{u^{j}_{i+1}-u^{j}_{i}} &
     {u^{j}_{i+1}-u^{j}_{i}}\neq 0
     \\ \sigma(A_{i+\frac{1}{2}}) & else,\end{array}\right.j=1,2\dots l, \label{sysspeed}
\end{equation}
where $\sigma_{i+\frac{1}{2}}$ is the spectrum of eigen values of
$A_{i+\frac{1}{2}}$.  In above computation (\ref{sysspeed}) the
nonphysical discrete wave speed caused by numerical overflow in case
of $u^{j}_{i+1}\approx u^{j}_{i}$ are corrected using following way
which is similar to the wave speed correction technique proposed in
\cite{jairaghu}.  It is done as,
\begin{eqnarray}
a^{j}_{i+\frac{1}{2}}= \sigma_{max} \;\mbox{if}\; |a^{j}_{i+\frac{1}{2}}|\geq \sigma_{max},\\
a^{j}_{i+\frac{1}{2}}= \sigma_{min} \;\mbox{if}\; |a^{j}_{i+\frac{1}{2}}|\leq \sigma_{min},  
\end{eqnarray}
where $\sigma_{max}$ and $\sigma_{min}$ refer to the local maxima and
minima of the magnitudes of characteristic speeds associated with
system (\ref{1Dsys}). For example, the one dimensional Euler equations
has the eigenvalues $u,\;u\pm c$ where $u$ and $c$ denotes fluid
velocity and the speed of sound respectively. In this case, we define
\begin{eqnarray}
\,\sigma_{\max}&=& \max{(\max{(|u|, |u-c|, |u+c|)_{i}} , \max{(|u|, |u-c|, |u+c|)_{i+1}) }},\\
\sigma_{\min}&=& \max{(\min{(|u|,|u-c|, |u+c|)_{i}} , \min{(|u|, |u-c|, |u+c|)_{i+1} )}}.
\end{eqnarray}

\subsection{Shock Sensor}
In order to locate the presence of discontinuities, a shock detector
proposed in \cite{Liushock} with some modification is used. A brief
detail on the shock switch is given below for the sake of completeness
of the discussion on numerical implementation.
\begin{itemize}
\item[Step 1:] Check multigrid ratio check
$$MR(i,h)=\frac{T_{C}(i,h)}{T_{F}(i,h) +\epsilon}$$ where
  $T_{C}(i,2h)$ and $T_{F}(i,h)$ are the $(4^{th}, 5^{th}\;
  \mbox{and}\; 6^{th})$ order truncation error sum on a coarse (with
  N/2 grid points) and fine grid (with N points) respectively. The
  derivatives in this step are calculated by by sixth order compact
  scheme proposed in \cite{Lele1992}. The small parameter
  $0<\epsilon<<1$ is used to avoid division by zero.
\item[Step 2:] Calculate the local ratio check at
  the grid point $x_i$ which has multigrid ratio $MR(i,h)\leq 4$,
$$LR(i) = \left|\frac{\left(u^{'}_{R}\right)^2-\left(u^{'}_{L}\right)^2}{\left(u^{'}_{R}\right)^2 +
    \left(u^{'}_{R}\right)^2 +\epsilon}\right| $$ where
  $u^{'}_{L}=3u_{i}-4u_{i-1}+u_{i-2}$ and
  $u^{'}_{R}=3u_{i}-4u_{i+1}+u_{i+2}$ left and right slope
  respectively at grid point $x_i$.
\item[Step 3:] Use a cutoff value $\delta \in(0,1]$ to create a shock switch (SS) on the result of step 2. i.e.,
$$ SS(i)=\left\{\begin{array}{ll} 0, & \mbox{if}\;LR(i)<\delta\; \mbox{i.e., data is locally smooth around grid point}\; x_{i},\\
    1,& \mbox{if} \;LR(i)\geq \delta\; \mbox{i.e., data is discontinuous around grid point}\; x_{i}.
\end{array}\right.$$
\end{itemize}
Note that the above shock detector has parameters $\epsilon$ and
$\delta$ which governs the sensitivity of shock switch. It is observed
in numerical computations that for larger value of parameters e.g.,
$\epsilon =1\times 10^{-2},\delta=0.8$ above shock switch is less
sensitive for mild shock or sharp turns and detects only strong shocks
whereas $\epsilon =1\times 10^{-8},\delta=0.2$ detect corners and mild
shock along with the strong shock. Also in case of non-linear systems,
it is observed that small oscillations may arise in the vicinity of
shock depending on the choice of shock parameters
$\epsilon,\delta$. It is therefore, to make it robust and less prone
to parameters $\epsilon,\delta$ a slight modification is done as
follows,
\begin{itemize}
\item[Step 4:] Treat neighbouring grid point $x_{i\pm 1}$ in
  discontinuity region if $SS(i)=1$ i.e., use $SS(i\pm1)=1$ if
  $SS(i)=1$.
\end{itemize}

\section{Numerical Results} \label{sec6} 
In this section numerical results are presented for various benchmark
scalar and system test problems in both one and two
dimensions. Different smooth as well discontinuous initial conditions
are taken to show the performance of schemes in section \ref{sec5} in terms of
accuracy and discontinuity capturing respectively. Numerical results
show that the proposed hybrid scheme, due to improved accuracy at
extrema and steep gradient region, nicely approximates the smooth
region of solution with crisp resolution for rarefaction, contact and
shock discontinuities.  Moreover it produces the total variation
diminishing numerical approximation.

\subsection{\bf Linear transport equation: every extrema is non-sonic}
Consider the linear transport equation
\begin{equation}
u_{t} + u_{x}=0,\; u(x,0)= u_{0}(x) \label{transport1}
\end{equation}
with periodic boundary condition. The exact solution of
(\ref{transport1}) equation convects with out changing the initial
shape of $u_{0}(x)$ and is given by $u(x,t) = u_{0}(x-t)$. Note that
in general number of extrema are finite and clipping error due to
degenerate accuracy at smooth extrema by existing high order monotone
and TVD method is visible only after a long period of time in form of
approximation of smooth extrema with corners or flatten profile see
Figure \ref{Fig1a}(a). Also due to degenerate accuracy at extrema and
steep gradient region their erratic convergence rate can be seen only
after a long time see Table \ref{LWflmT3linprb3}. It is therefore,
probably the transport equation (\ref{transport1}) the only test which
can be used to check the large time performance of any method. Since
the problem (\ref{transport1}) is linear thus discontinuities present
in the solution does not represents shock or rarefaction therefore
scheme \ref{algo3} can be directly applied for this test problem with
out shock switch. The numerical computation for problem
(\ref{transport1}) is done by using the first order upwind scheme as
complementary conservative scheme (CCS) in hybrid scheme \ref{algo3}
and results are shown by legend {\it method}.
\subsubsection{\bf Accuracy check: smooth initial condition}
Consider \ref{transport1}) along with the the following three different
initial conditions which comprises of smooth extrema, monotone region
with mild as well sharp turn. 
\begin{enumerate}
\item[i] Smooth extrema 
  \begin{equation}
    u(x,0) = \sin(\pi x), \; x\in
    [-1,\;1].\label{Lin-IC1}
  \end{equation} 
  The initial profile consists smooth extreme points at $x=\pm
  \frac{1}{2}$ which preservingly convects to the right direction.
\item[ii] Smooth extrema with monotone data region
  \begin{equation}
    u(x,0) = \sin^4(\pi x), \; x\in [0,\;1].\label{Lin-IC2}
  \end{equation} 
  This initial condition is taken from \cite{zhang2010} and has a
  smooth extrema at $x=0.5$ with monotone solution regions with mild
  turn towards the bottom.
\item[iii] Smooth extrema with steep gradient region  
  \begin{equation}\label{Lin-IC3}
    u(x,0) = \left\{\begin{array}{ll}
        e^{\frac{-1}{1-x^2}}\; & x\in [-1:1]\\
        0\; & else
      \end{array}\right.
  \end{equation}
  This initial condition has a smooth extrema at $x=0$. Compared to
  initial condition (\ref{Lin-IC2}) it has high gradient monotone
  region with sharp turns towards the bottom where $r\rightarrow 0+$
  or $r>>1$ respectively. This is a good test to see the degenerate
  convergence rate for any limiter based TVD scheme.
\end{enumerate}
 For smooth initial data numerical solution plots are given at large
 time level{ $\bf T_{s}$}\footnote{We consider large time $T_{s}$ as
   the time level when corners are visible in the approximation of
   smooth extrema by high order TVD methods such as in
   \cite{harten1983,Dubey2013,sweby1984,Yee1987,davis1987}.}.  The
 convergence rate of scheme \ref{algo4} is given at various time
 instance with varying $CFL$ number to show the robust and higher than
 second order convergence rate of the method \ref{algo3}. In Figure
 \ref{Fig1a}(a), numerical solution obtained by method corresponding
 to IC (\ref{Lin-IC1}) is compared with high order TVD Lax-Wendroff
 flux limited method (LxWFLM) \cite{Rider} with compressive Superbee
 limiter \cite{roe1985some}. In Figure \ref{Fig1b} and Figure
 \ref{Fig1c} approximate solution is given for transport problem
 corresponding to initial conditions (\ref{Lin-IC2}) and
 (\ref{Lin-IC3}) respectively.  The total variation of the computed
 solution obtained by method is compared with uniformly second  order LxW, BW and Fromm schemes respectively for all three IC's in Figure
 \ref{Fig1a}(b), \ref{Fig1b}(b) and \ref{Fig1c}(c).

From the numerical results in it is evident that problem of flattening
of smooth round shaped solution profile is removed due to improved
approximation of extreme points. Moreover it can be observed and
Figure \ref{Fig1b}(a) that this improvement more visible as $\lambda a
\rightarrow 1$ which support the improved TVD region for extrema of
LxW as discussed in the Corollary \ref{cor1}. Total variation plots show
that total variation for designed scheme \ref{algo3} is decreasing
whilst uniformly second order LxW and BW schemes do not produce TVD
solution though it remain total variation bounded (TVB).

In Table \ref{Tab1linprb1}, discrete maximum $L^{\infty}$ error
convergence rate is given for scheme \ref{algo3} as $L^{\infty}$ error
is the best indicator for checking the performance of any scheme in
terms of clipping error due to drop in accuracy at smooth extrema. In
Table \ref{Tab2linprb2} and Table \ref{Tab3linprb3}, error convergence
rates are given in terms of $L^{1}$ and $L^{\infty}$ error for method
with different choice of $CFL$ and time for linear test corresponding
to initial conditions (\ref{Lin-IC2}) and (\ref{Lin-IC3})
respectively. The numerical results show that the designed scheme
\ref{algo3} shows higher than second order convergence rate
independently of the choice of $CFL$ number or final time $T$. Also
due to improved approximation of extrema and steep gradient region,
the used method yields smooth approximation with out clipping error
which support the Corollary \ref{cor1} and \ref{cor2}.
\begin{figure}[!htb]
  \begin{center}
    \begin{tabular}{cc}
      \includegraphics[%
      scale=0.5]{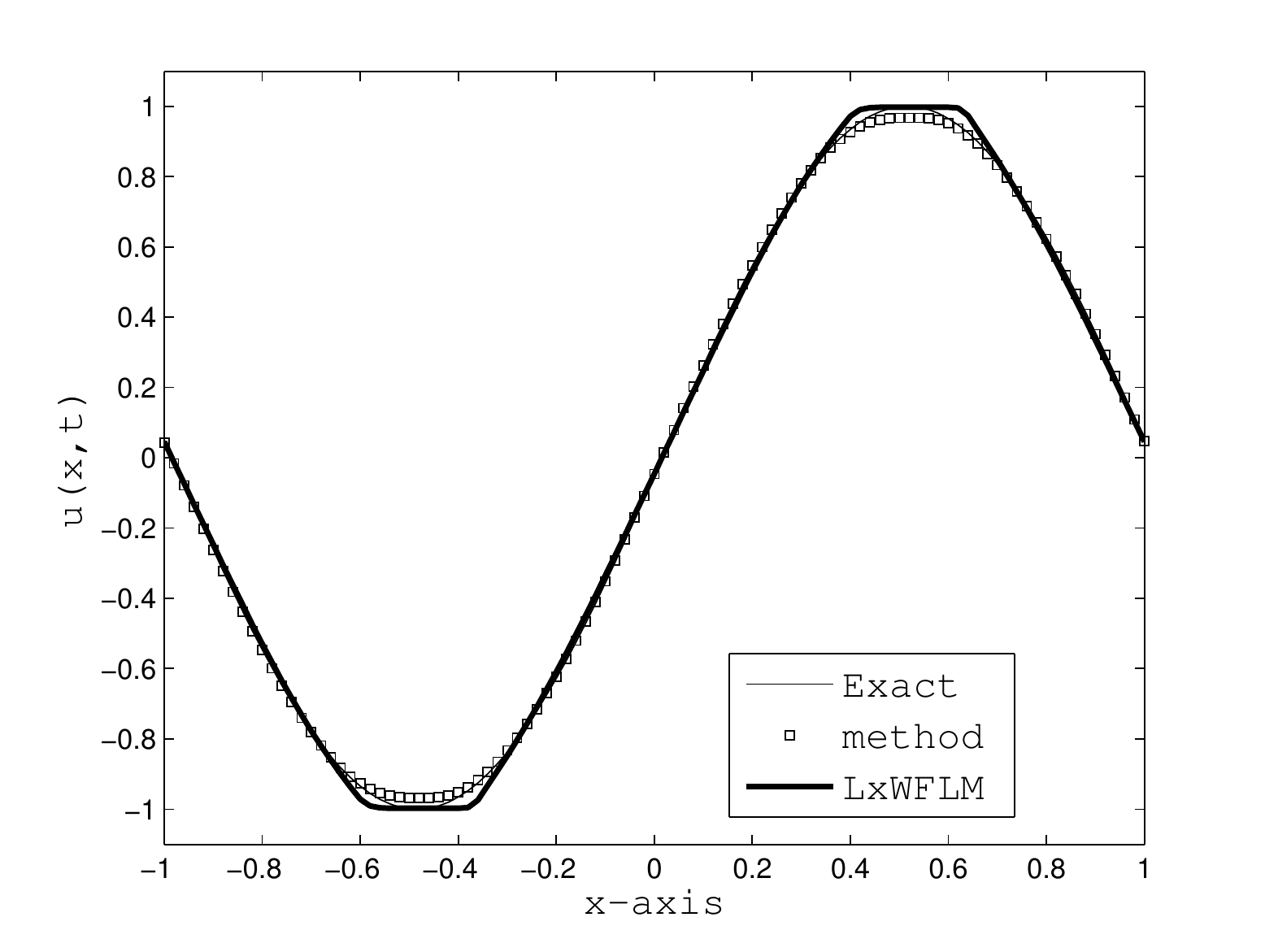} &\includegraphics[%
      scale=0.5]{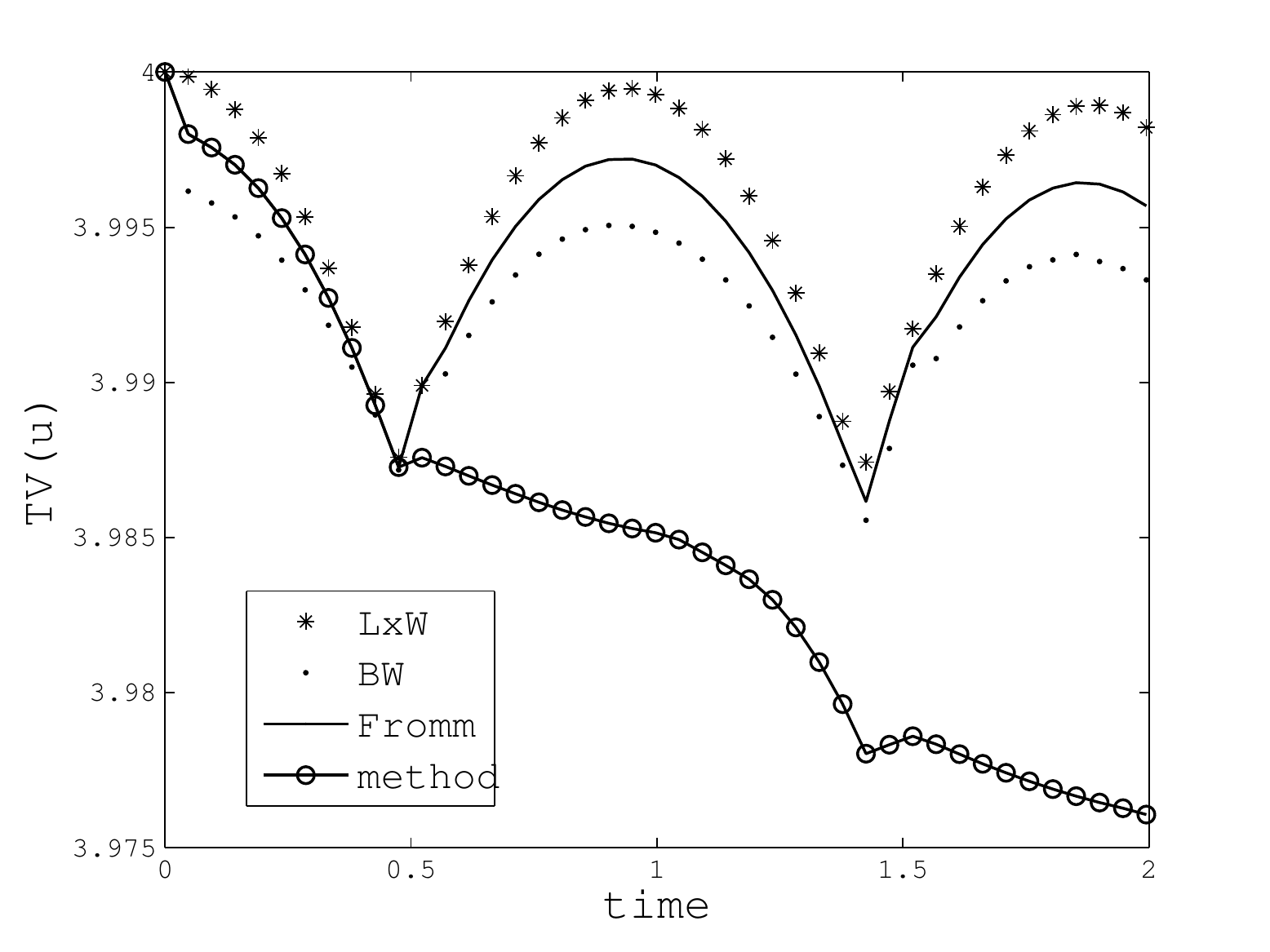}\\
      (a) at $T=30$. & (b) at $T=2$ with $CFL=0.95$.\\
    \end{tabular}
  \end{center}
  \caption{\label{Fig1a} Solution of (\ref{transport1}) with IC
    (\ref{Lin-IC1}) using $N=80$: (a) Flatten approximation for smooth
    extrema by LxWFLM whereas proposed scheme \ref{algo4} preserve
    solution with smooth extrema with out introducing corners. (b)
    Comparison of total variation.}
\end{figure}

\begin{table}[!htb]
\begin{tabular}{|c|c|}
\hline
T=2 & T=30\\
\hline
\begin{tabular}{c|c}
CFL=0.5 & CFL=0.95\\
\hline
\begin{tabular}{ccc}
N & $L^{\infty}$ error & Rate\\
\hline
20 & 9.1927e-03 & \dots \\
40 & 1.9456e-03 & 2.240 \\
80 & 3.7656e-04 & 2.369 \\
160 & 7.0744e-05 & 2.412 \\
320 & 1.2599e-05 & 2.489 \\
640 & 3.2485e-06 & 1.955\\
\end{tabular}
&
\begin{tabular}{cc}
$L^{\infty}$ error & Rate\\
\hline
1.4026e-03 & \dots \\
2.6640e-04 & 2.396 \\
5.6365e-05 & 2.241 \\
1.1332e-05 & 2.314 \\
2.6879e-06 & 2.076 \\
3.1057e-07 & 3.113 \\
\end{tabular}
\end{tabular}
&
\begin{tabular}{c|c}
CFL=0.5 & CFL=0.95\\
\hline
\begin{tabular}{cc}
$L^{\infty}$ error & Rate\\
\hline
6.3628e-02 & \dots \\
1.0892e-02 & 2.546 \\
1.8543e-03 & 2.554 \\
3.3400e-04 & 2.473 \\
4.1812e-05 & 2.998 \\
3.5142e-06 & 3.573\\
\end{tabular}
&
\begin{tabular}{cc}
$L^{\infty}$ error & Rate\\
\hline
1.1716e-02 & \dots\\ 
2.2095e-03 & 2.407 \\
4.4726e-04 & 2.305 \\
8.6187e-05 & 2.376 \\
1.4613e-05 & 2.560 \\
4.5598e-07 & 5.002 \\
\end{tabular}
\end{tabular}\\
\hline
\end{tabular}
\caption{\label{Tab1linprb1} {\it Consistent higher than second order $L^{\infty}$ convergence rate with the
    mesh refinement corresponding to initial condition (\ref{Lin-IC1})}.}
\end{table}

\begin{figure}[!htb]
  \begin{center}
    \begin{tabular}{cc}
      \includegraphics[%
      scale=0.5]{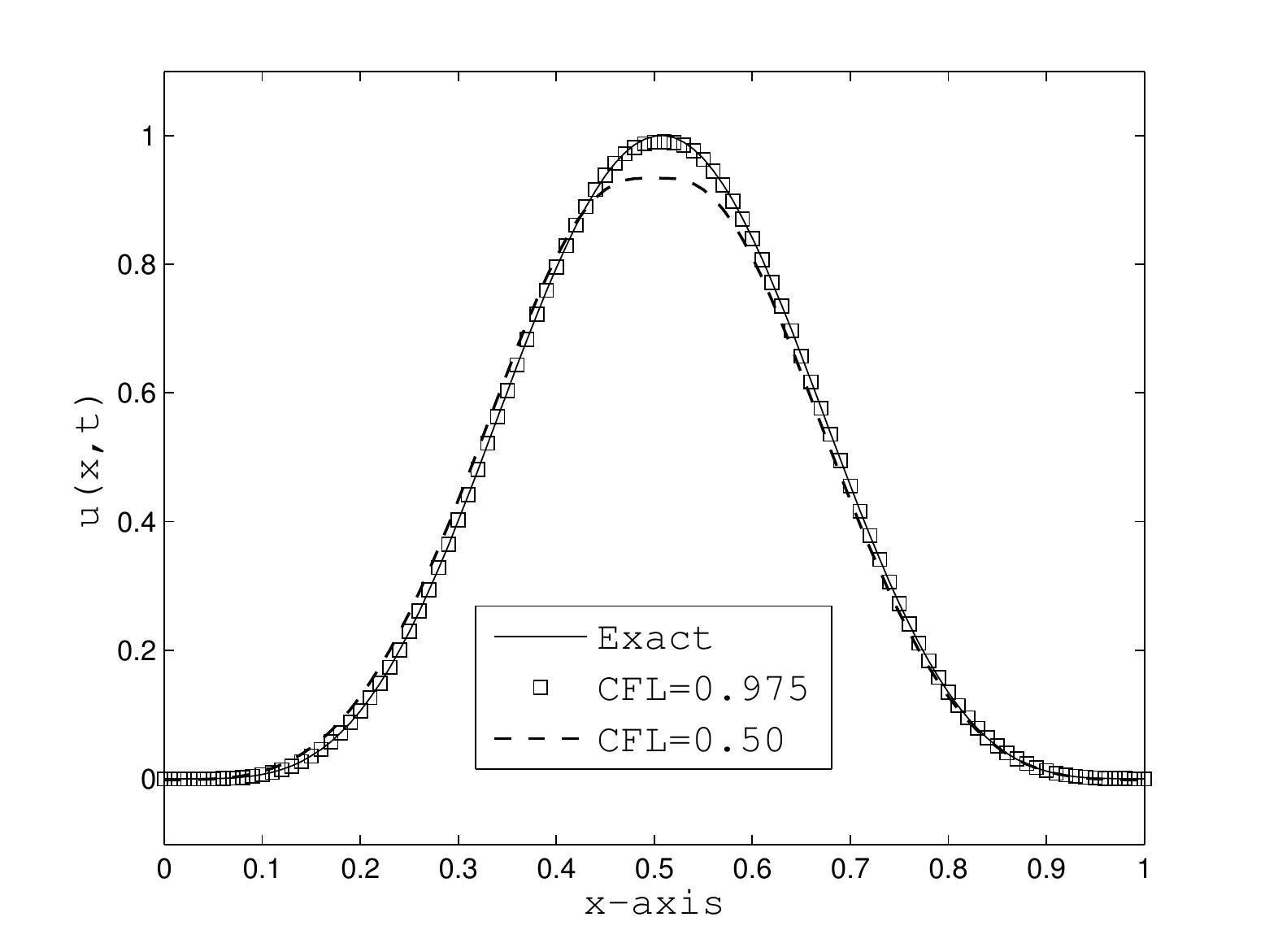} &\includegraphics[%
      scale=0.5]{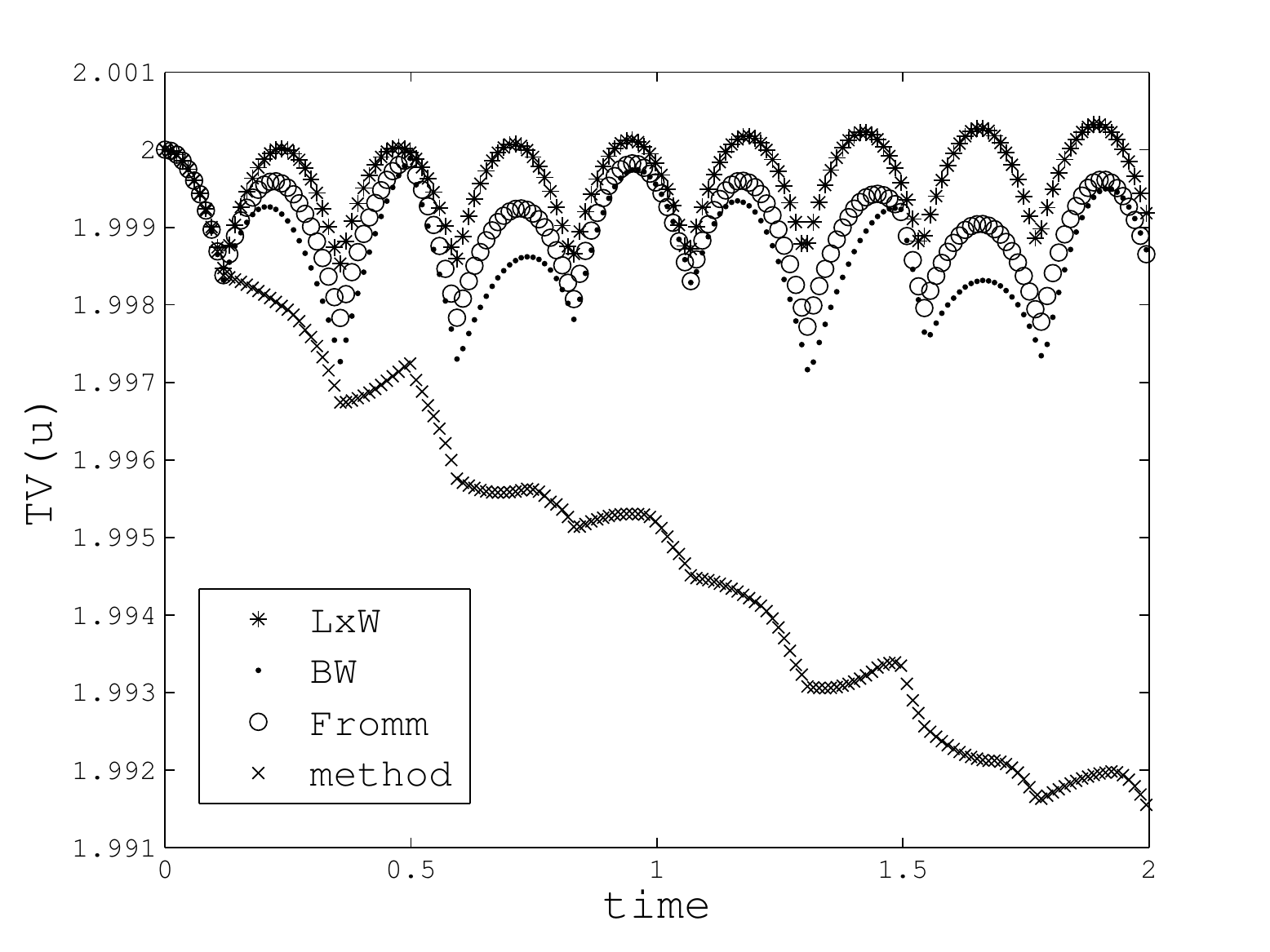}\\
      (a) & (b)\\
    \end{tabular}
  \end{center}
  \caption{\label{Fig1b} Solution for linear equation
    (\ref{transport}) with IC (\ref{Lin-IC2}) (a) at time $T=20,\;
    N=100$, Smooth approximation of extrema with reduced clipping
    error: (b) TV plot up to $T=2$ using $N=80, CFL=0.95$.}
\end{figure}

\begin{table}[!htb]
  \begin{tabular}{|c|c|}
    \hline
    CFL=0.5 & CFL=0.95\\
    \hline
    \begin{tabular}{ccc|cc}
      \hline
      N& $L^{1}$ Error & Rate & $L^{\infty}$ error & Rate\\
      \hline
      40 & 7.6962e-02 & \dots & 	 6.1348e-03 & \dots \\
      80 & 1.7350e-02 & 2.149 & 	 1.1039e-03 & 2.474 \\
      160 & 3.8066e-03 & 2.188 & 	 1.9596e-04 & 2.494 \\
      320 & 7.3723e-04 & 2.368 & 	 2.4368e-05 & 3.008 \\
      640 & 1.1693e-04 & 2.656 & 	 4.3613e-07 & 5.804 \\
      1280 & 2.8601e-05 & 2.031 & 	 5.4189e-08 & 3.009 \\
      2560 & 7.0685e-06 & 2.017 & 	 6.6891e-09 & 3.018 \\
    \end{tabular}
    & 
    \begin{tabular}{cc|cc}
      \hline
      $L^{1}$ Error & Rate & $L^{\infty}$ error & Rate\\
      \hline
      1.3179e-02 & \dots & 	 1.4315e-03 & \dots \\
      3.4723e-03 & 1.924 & 	 2.7228e-04 & 2.394 \\
      8.7436e-04 & 1.990 & 	 5.2159e-05 & 2.384 \\
      2.0210e-04 & 2.113 & 	 8.6669e-06 & 2.589 \\
      3.2508e-05 & 2.636 & 	 1.1141e-07 & 6.282 \\
      8.0887e-06 & 2.007 & 	 1.3877e-08 & 3.005 \\
      2.0182e-06 & 2.003 & 	 1.7315e-09 & 3.003 \\
    \end{tabular}\\
    \hline
  \end{tabular}
  \caption{\label{Tab2linprb2} {\it Order of convergence with the
      mesh refinement at $T=20$ corresponding to initial condition (\ref{Lin-IC2})}.}
\end{table}

\begin{figure}[!htb]
  \begin{center}
    \begin{tabular}{cc}
      \includegraphics[%
      scale=0.5]{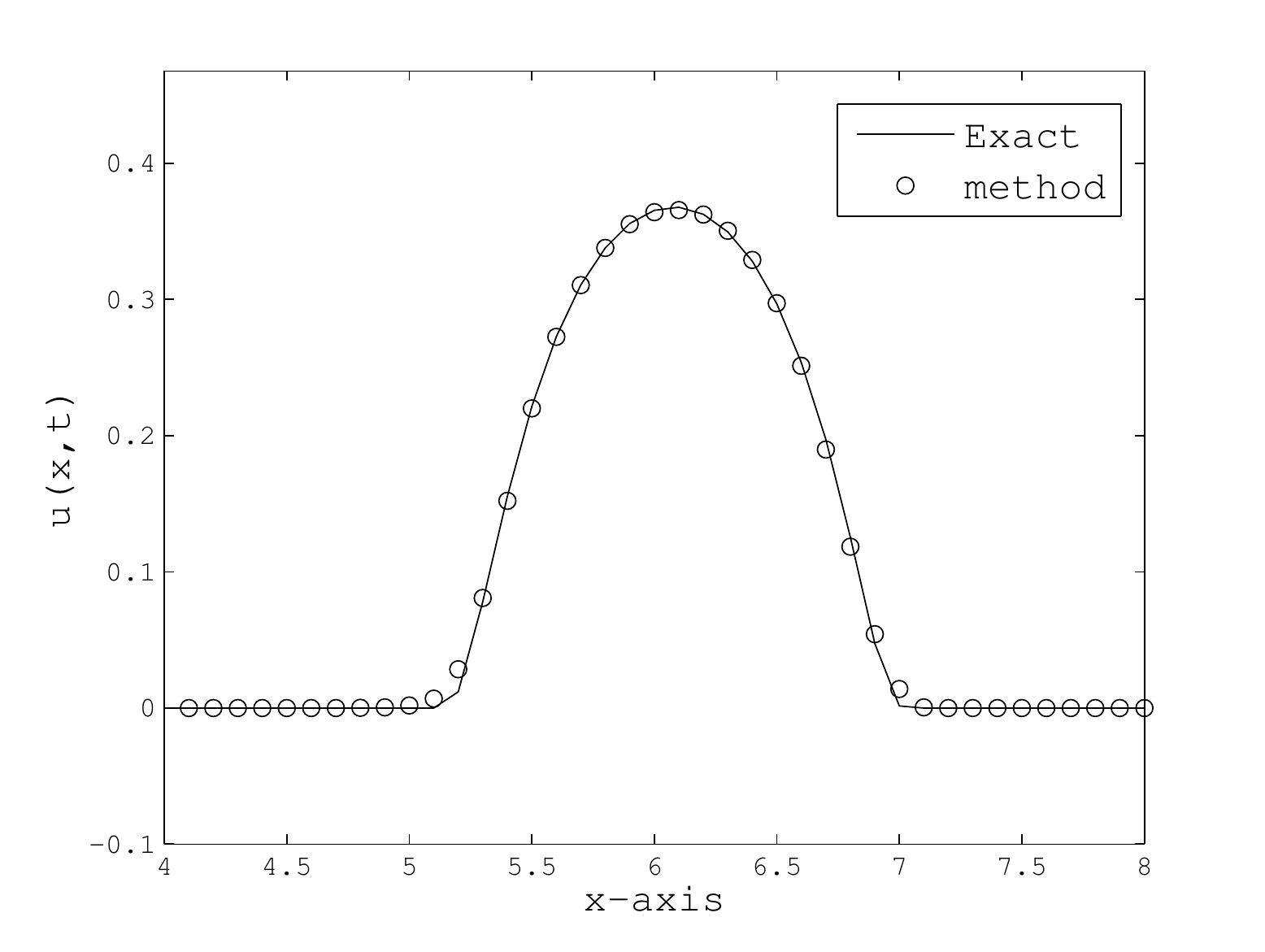} &\includegraphics[%
      scale=0.5]{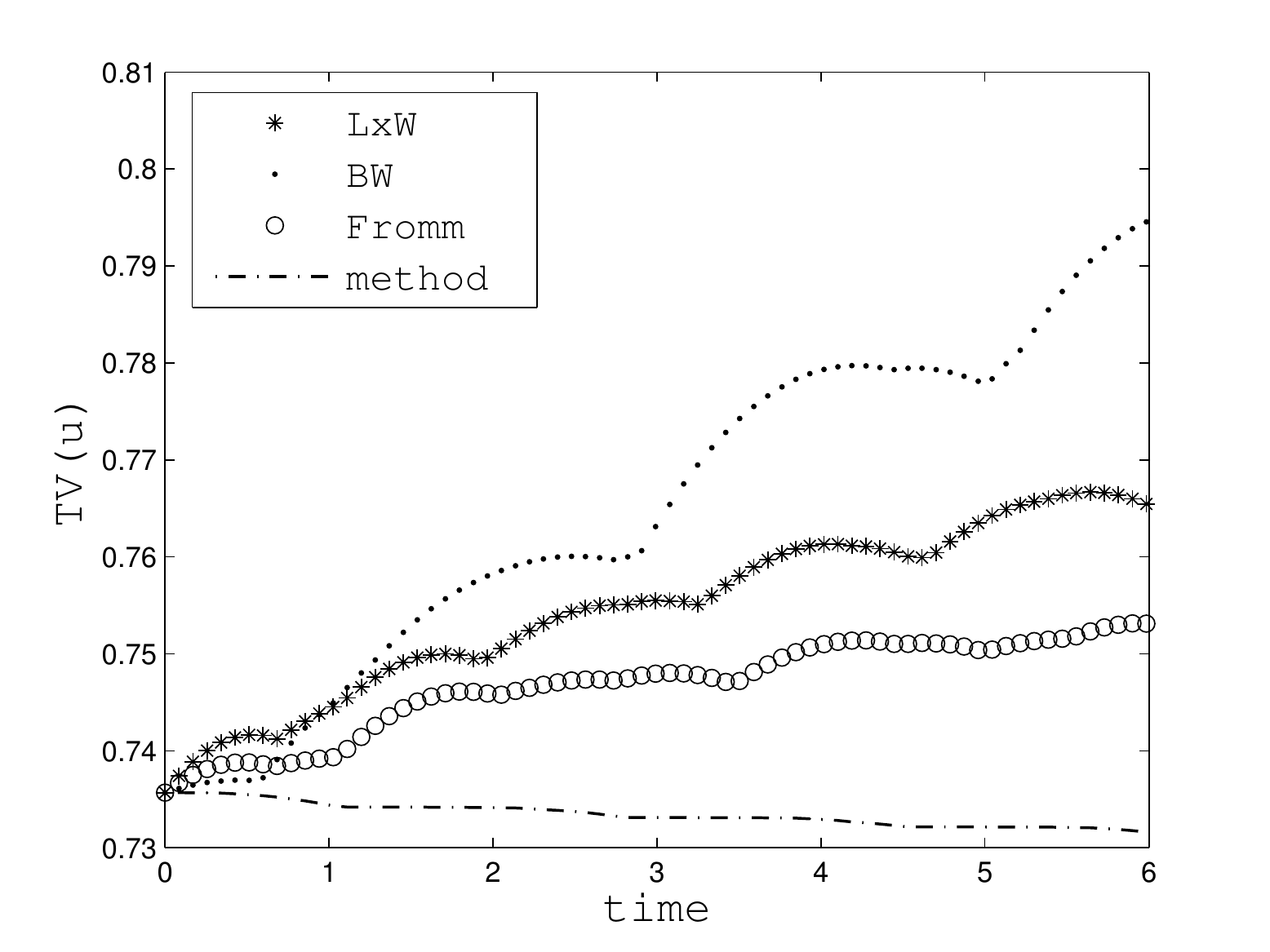}\\
      (a) & (b)\\
    \end{tabular}
  \end{center}
  \caption{\label{Fig1c} Solution for linear equation
    (\ref{transport}) with IC (\ref{Lin-IC3}) at time $T=6$ and
    $N=100,\;CFL=0.95$. (a) Smooth approximation of extrema with no
    clipping error by Scheme \ref{algo1}. (b) Comparison of total
    variation.}
\end{figure}

\begin{table}[!htb]
  \begin{tabular}{|c|c|}
    \hline
    I-Order Upwind & Second order TVD scheme\\
    \hline
    \begin{tabular}{c|cc|cc}
      \hline
      N& $L^{1}$ Error & Rate & $L^{\infty}$ error & Rate\\
      \hline
      80 & 2.5248e-01 & 0.583 & 	 1.6091e-02 & 1.482 \\
      160 & 1.6210e-01 & 0.639 & 	 5.8392e-03 & 1.462 \\
      320 & 9.7673e-02 & 0.731 & 	 2.2646e-03 & 1.367 \\
      640 & 5.6346e-02 & 0.794 & 	 8.2206e-04 & 1.462 \\
      1280 & 3.1355e-02 & 0.846 & 	 2.8328e-04 & 1.537 \\
      2560 & 1.6889e-02 & 0.893 & 	 9.2241e-05 & 1.619 \\
    \end{tabular}
    & 
    \begin{tabular}{cc|cc}
      \hline
      $L^{1}$ Error & Rate & $L^{\infty}$ error & Rate\\
      \hline
      2.2319e-02 & 1.314 & 	 2.4725e-03 & 2.294 \\
      1.6468e-02 & 0.439 & 	 1.3464e-03 & 0.877 \\
      5.2704e-03 & 1.644 & 	 2.0531e-04 & 2.713 \\
      2.3125e-03 & 1.188 & 	 8.7812e-05 & 1.225 \\
      1.0757e-03 & 1.104 & 	 4.1735e-05 & 1.073 \\
      4.1654e-04 & 1.369 & 	 1.0950e-05 & 1.930\\
    \end{tabular}\\
    \hline
  \end{tabular}
  \caption{\label{LWflmT3linprb3} {\it Order of convergence using I
      order upwind and LxW flux limited TVD method with Superbee
      limiter at $T=6, CFL=0.5$ corresponding to initial condition
      (\ref{Lin-IC3})}.}
\end{table}

\begin{table}[!htb]
  \begin{tabular}{|c|c|}
    \hline
    CFL=0.5 & CFL=0.95\\
    \hline
    \begin{tabular}{c|cc|cc}
      \hline
      N& $L^{1}$ Error & Rate & $L^{\infty}$ error & Rate\\
      \hline
      80 & 4.3378e-02 & 1.654 & 	 4.9870e-03 & 2.139\\ 
      160 & 1.2992e-02 & 1.739 & 	 1.3187e-03 & 1.919 \\
      320 & 3.2156e-03 & 2.014 & 	 2.6333e-04 & 2.324 \\
      640 & 6.1867e-04 & 2.378 & 	 3.8218e-05 & 2.785 \\
      1280 & 1.1353e-04 & 2.446 & 	 4.1315e-06 & 3.210 \\
      2560 & 1.7115e-05 & 2.730 & 	 3.6114e-07 & 3.516 \\
    \end{tabular}
    & 
    \begin{tabular}{cc|cc}
      \hline
      $L^{1}$ Error & Rate & $L^{\infty}$ error & Rate\\
      \hline
      1.1681e-02 & 1.646 & 	 2.2055e-03 & 2.239 \\
      3.2801e-03 & 1.832 & 	 5.0305e-04 & 2.132 \\
      8.1011e-04 & 2.018 & 	 8.1325e-05 & 2.629 \\
      1.7801e-04 & 2.186 & 	 1.3173e-05 & 2.626 \\
      4.5829e-05 & 1.958 & 	 1.8927e-06 & 2.799 \\
      1.2376e-05 & 1.889 & 	 2.5213e-07 & 2.908 \\
    \end{tabular}\\
    \hline
  \end{tabular}
  \caption{\label{Tab3linprb3} {\it Consistent higher than second order of convergence with
      the mesh refinement at $T=6$ corresponding to initial condition (\ref{Lin-IC3})}}
\end{table}

\subsubsection{\bf Discontinuous initial condition}
In this test is taken from \cite{despres} originally used by Harten in
\cite{Harten1989}. Initial solution is complex in nature which
contains parts of smooth solution, mix discontinuities,
discontinuities of derivative in the interval $[-1,1]$. In Figure
\ref{Fig4}(a), numerical results is given by proposed method and in
Figure \ref{Fig4}(b) the total variation plot of the computed solution
by {\it method} is given.
\begin{equation}
  \label{icprb3}
  u_{0}(x) = \left\{\begin{array}{lcl}
      2x + 2 -sin(3\pi(x-0.5))/6 & \mbox{if} &-1\leq x\leq 0.5, \\
      (0.5-x)sin(1.5\pi(x-0.5)^2) & \mbox{if} &-0.5<x<1/6,\\
      |sin(2\pi(x-0.5))| & \mbox{if} & 1/6< x<= 5/6,\\
      2x - 2 -sin(3\pi(x-0.5))/6 & \mbox{if} & 5/6< x <=1.\\
    \end{array}\right. \\ 
\end{equation}
It can be seen in Figure \ref{Fig4} that the proposed scheme \ref{algo3} yields TVD
solution with crisp capturing of the discontinuities with out clipping
error at smooth extrema. The error convergence rate is not shown for
this discontinuous test problem as discontinuities can only be
approximated with at most first order of accuracy \cite{cwshu2012}
\begin{figure}[httb!]
  \begin{center}
    \begin{tabular}{cc}
      \includegraphics[scale=0.5]{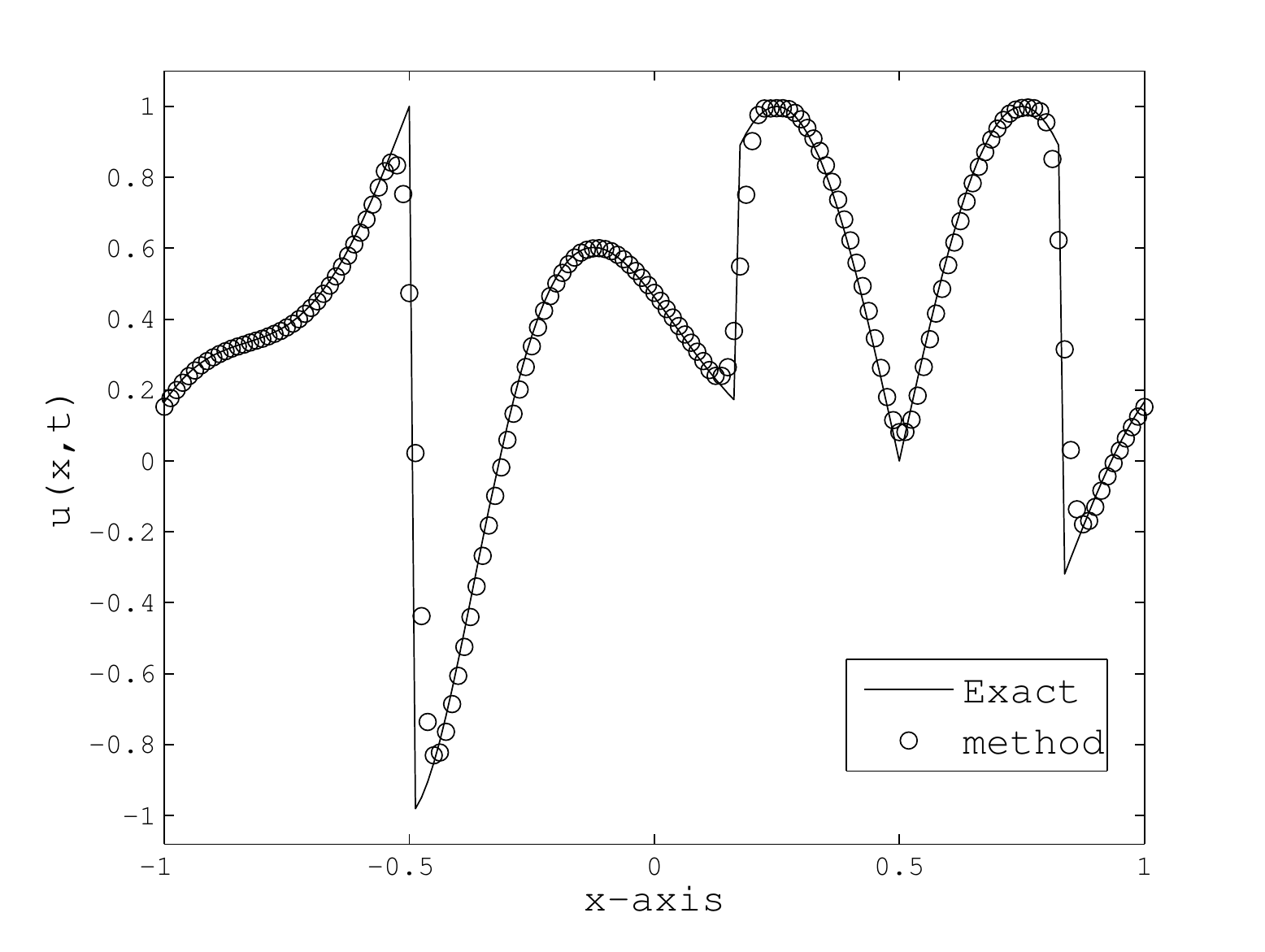} & \includegraphics[%
    scale=0.5]{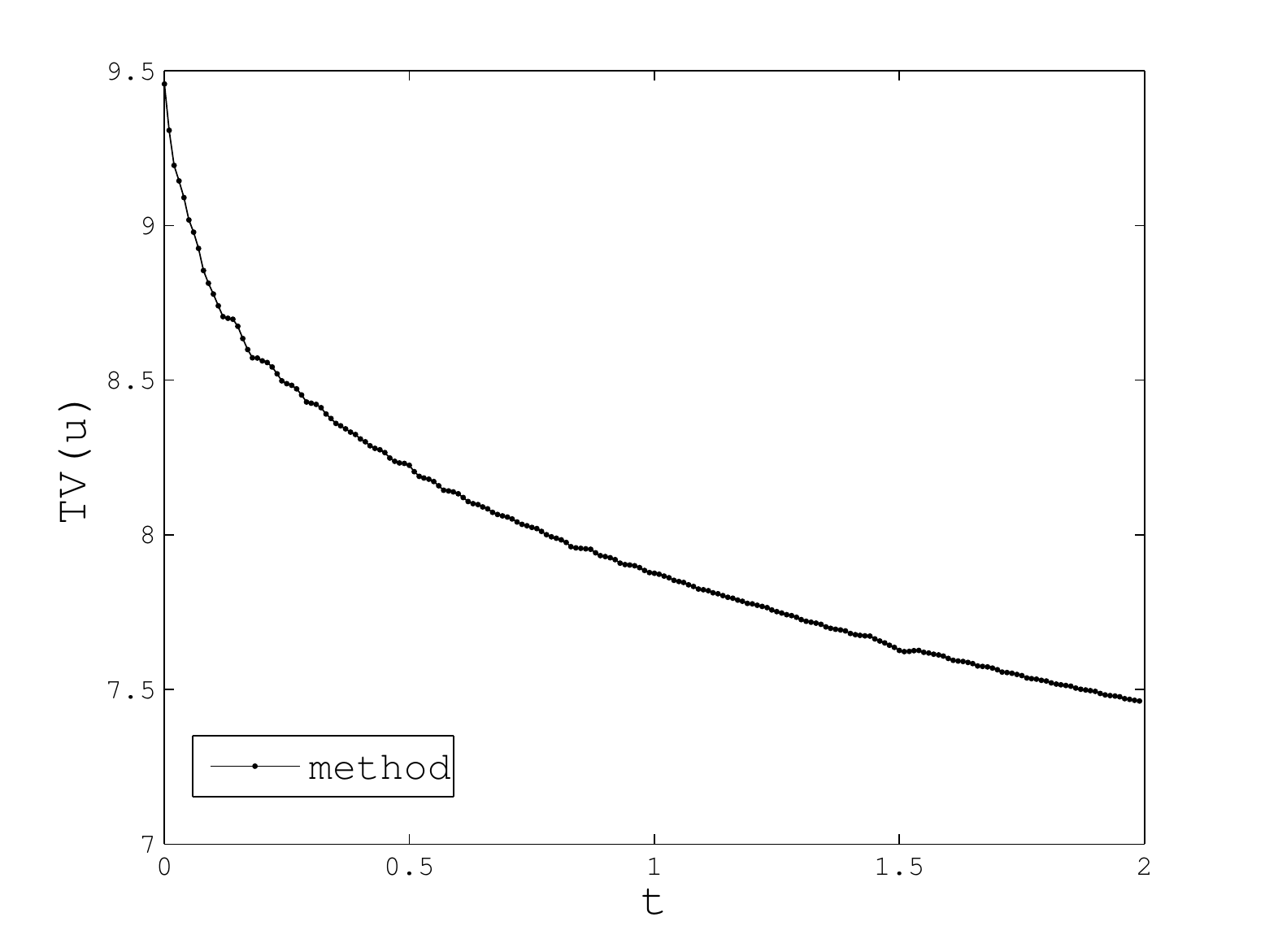}\\
      (a) & (b)
    \end{tabular}
    \caption{\label{Fig4} (a) High resolution oscillation free solution
      for test problem \ref{icprb3} by scheme \ref{algo4} for data
      $CFL=0.95,\, N=160,\, T=2.0 $. (b) Total variation decreasing plot of method}
  \end{center}
\end{figure}
\subsection{\bf Non-linear case: Burgers equation}
Consider the Burgers equation
\begin{equation}
  \displaystyle u_t + \left(\frac{u^2}{2}\right)_x=0, -a\leq x
  \leq b \label{burgers}
\end{equation}
with initial condition $u_{0}(x)$ and periodic boundary conditions. It
is the non-linear nature of the equation (\ref{burgers}) that even for
smooth initial condition, the solution of (\ref{burgers}) eventually develops
discontinuities like rarefaction and shocks after breaking time $T_{b}$
given by
\begin{equation}
  T_{b}=\frac{-1}{\min_{x}(u^{'}_0(x))} \label{Tbtime}
\end{equation}
Also the unique sonic point for Burgers equation (\ref{burgers}) is
$u^{*}=0$. It is therefore, Burgers equation is a good test to check
the performance of any scheme for smooth as well discontinuous
solutions profile at pre and post-shock time $T_{b}$ respectively. In
the numerical computation FORCE scheme is used as CCS in schemes
\ref{algo1} to \ref{algo3} and their respective hybrid shock corrected
analog in scheme \ref{algo4}.
\subsubsection{\bf Shock correction moving shock} 
We first consider the following discontinuous initial condition to show
the non-conservative nature of schemes \ref{algo1}, \ref{algo2} and
\ref{algo3} as they yield solution with wrong location of moving shock. 
\begin{equation}
    u(x,0) = \left\{\begin{array}{cc} 1, & |x|\leq \frac{1}{3},\\ 0, &
    else. \end{array}\right.\label{nonlin-IC2}
\end{equation}
The solution corresponding to IC (\ref{nonlin-IC2}), develops a
rarefaction wave and a {\it moving} shock which corresponds to initial
discontinuities at $x=-1/3$ and $x=1/3$ respectively. In Figure
\ref{Fig6a}(a) results obtained by second order Lax-Wendroff
(LW-FORCE) and Beam-Warming (BW-FORCE) schemes \ref{algo1} and
\ref{algo2} respectively are given. It is clear that shock location
given by LW-FORCE is legging behind whereas by BW-FORCE it is leading
ahead of exact shock location. Results by shock corrected schemes
SC-LW-FORCE and SC-BW-FORCE as described in \ref{algo4} are given in
\ref{Fig6a}(b) which show correct location of shock is recovered with
out loosing crisp resolution of left rarefaction. Numerical results in
Figure \ref{Fig6b}(a), obtained by scheme \ref{algo3}
(FLWBW-FORCE) shows that shock is crisply captured at right location
with high resolution for bottom and top of left rarefaction. In Figure
\ref{Fig6b}(b), results by SC-FLWBW-FORCE are given which show little
dissipative resolution for shock which is due to approximation by
dissipative FORCE scheme in the vicinity of shock.
\begin{figure}[httb!]
  \begin{tabular}{c|c}
\includegraphics[scale=0.4]{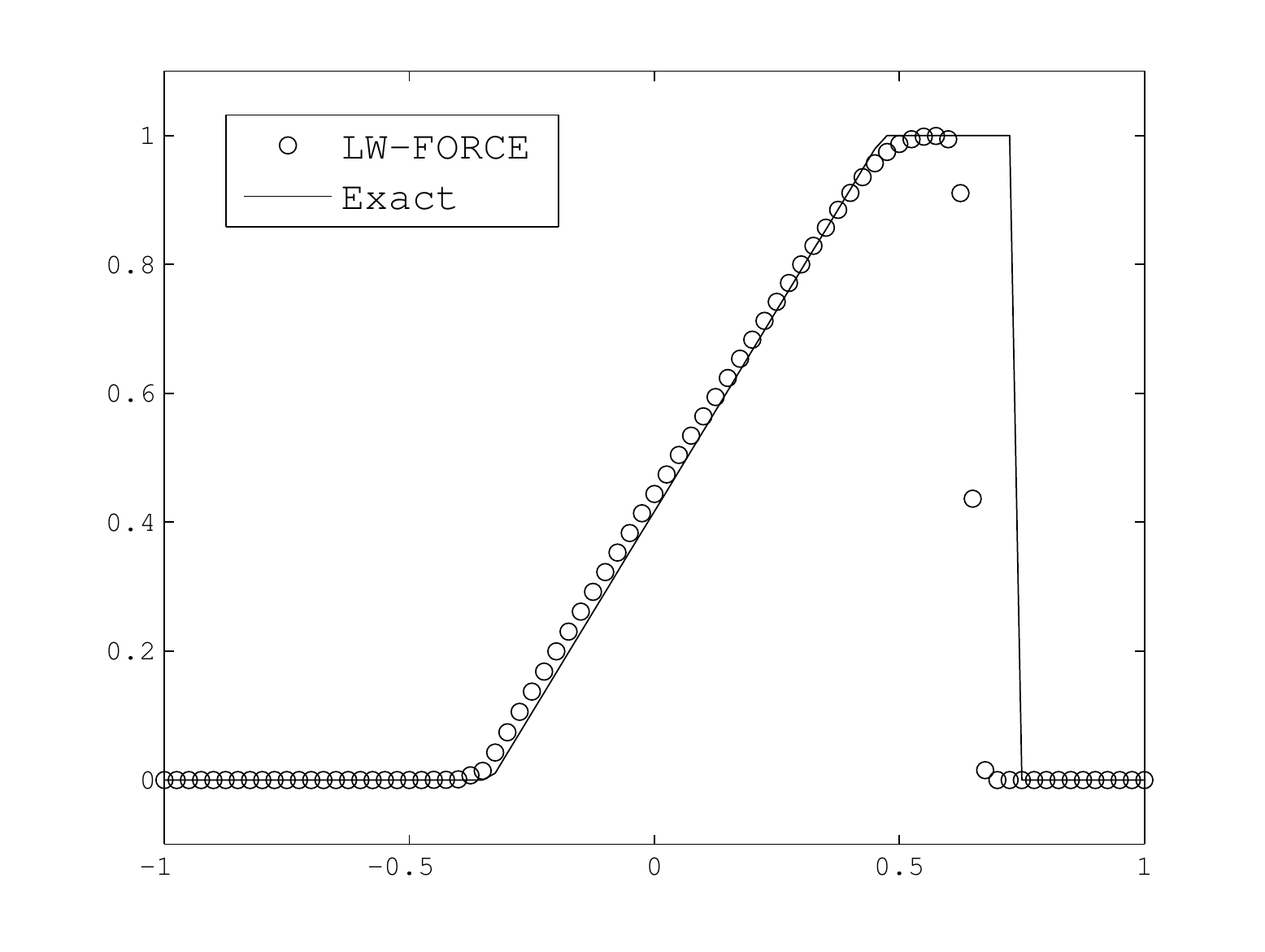} & \includegraphics[scale=0.4]{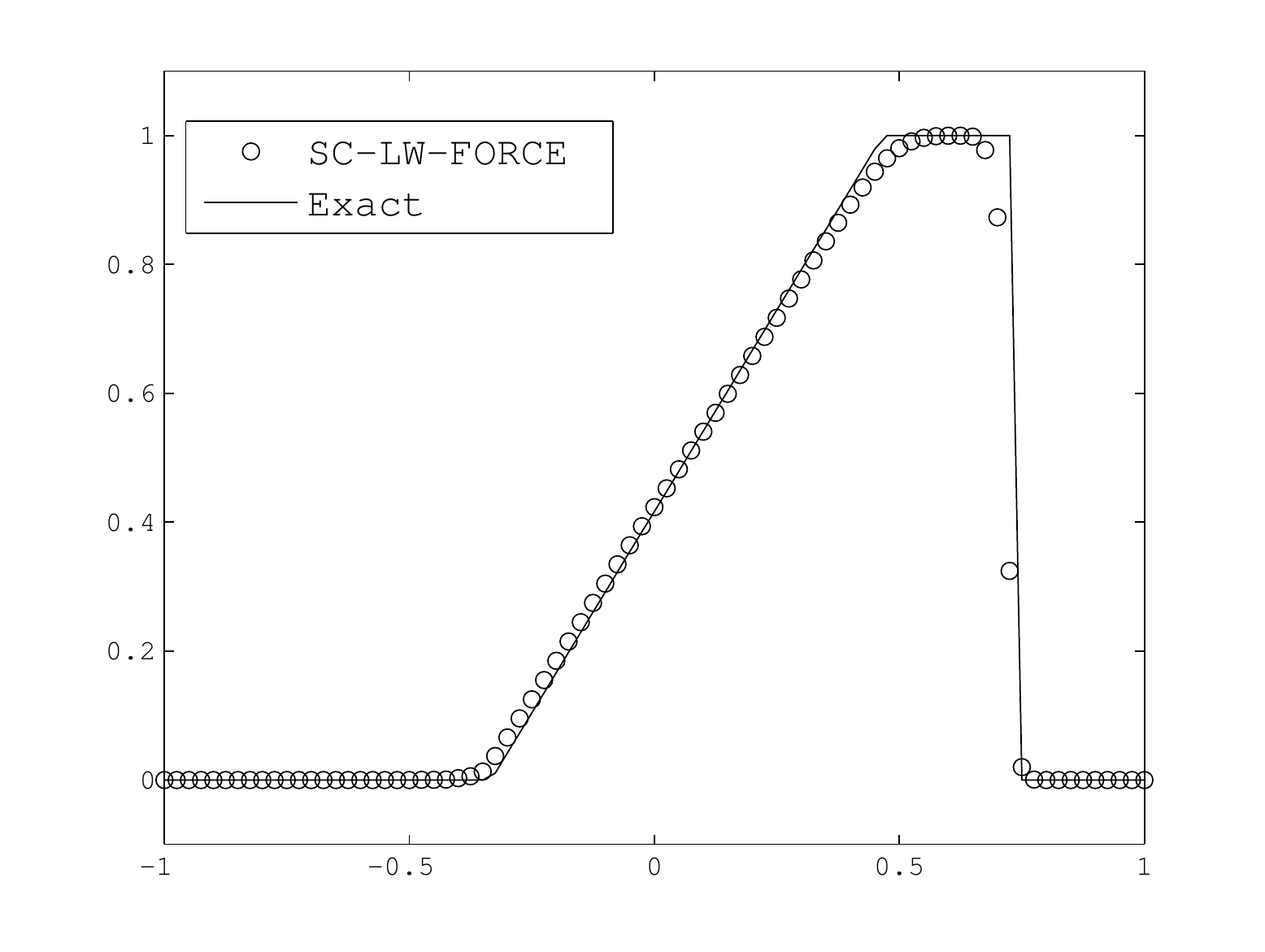}\\ 
\includegraphics[scale=0.4]{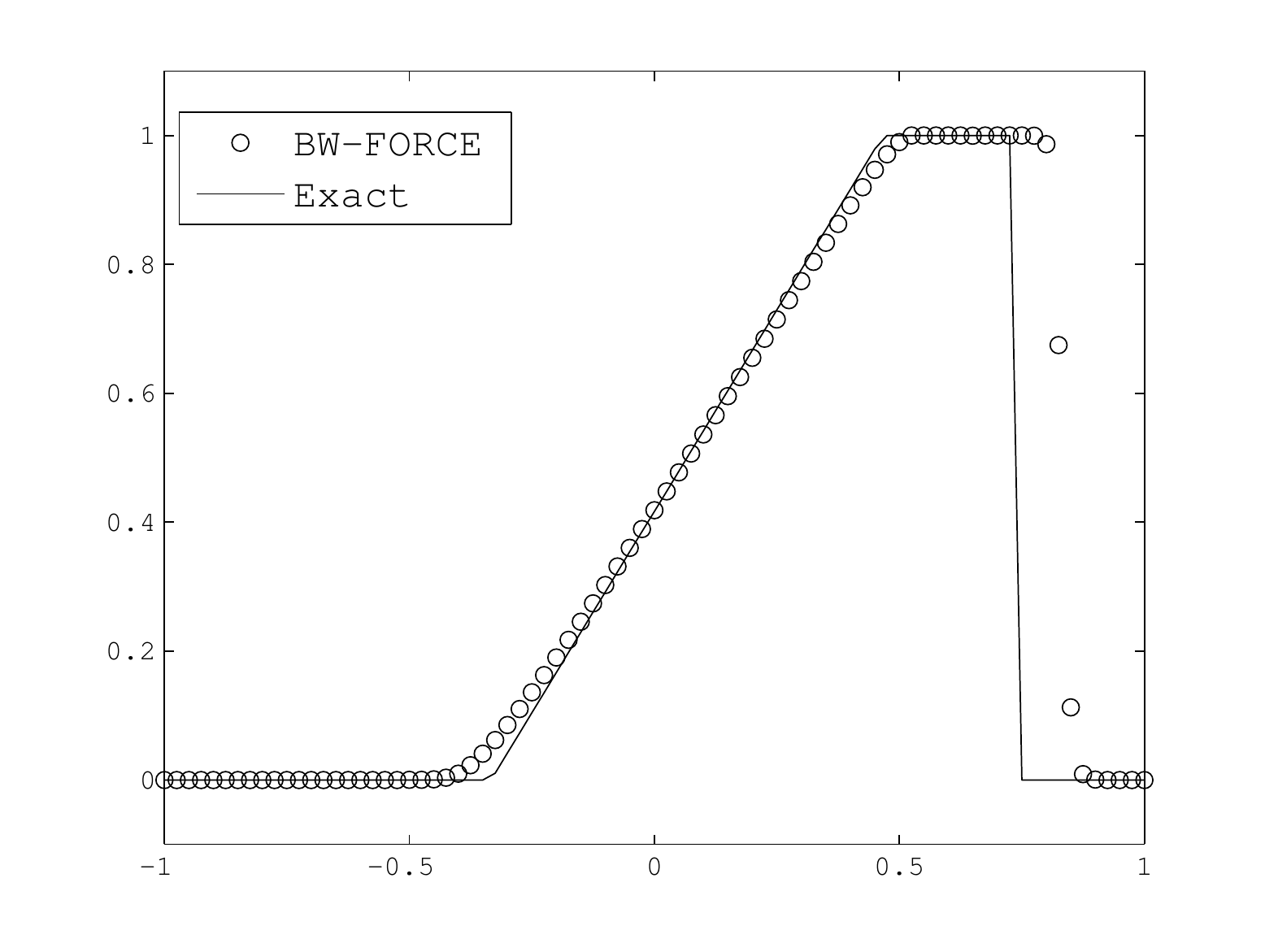} & \includegraphics[scale=0.4]{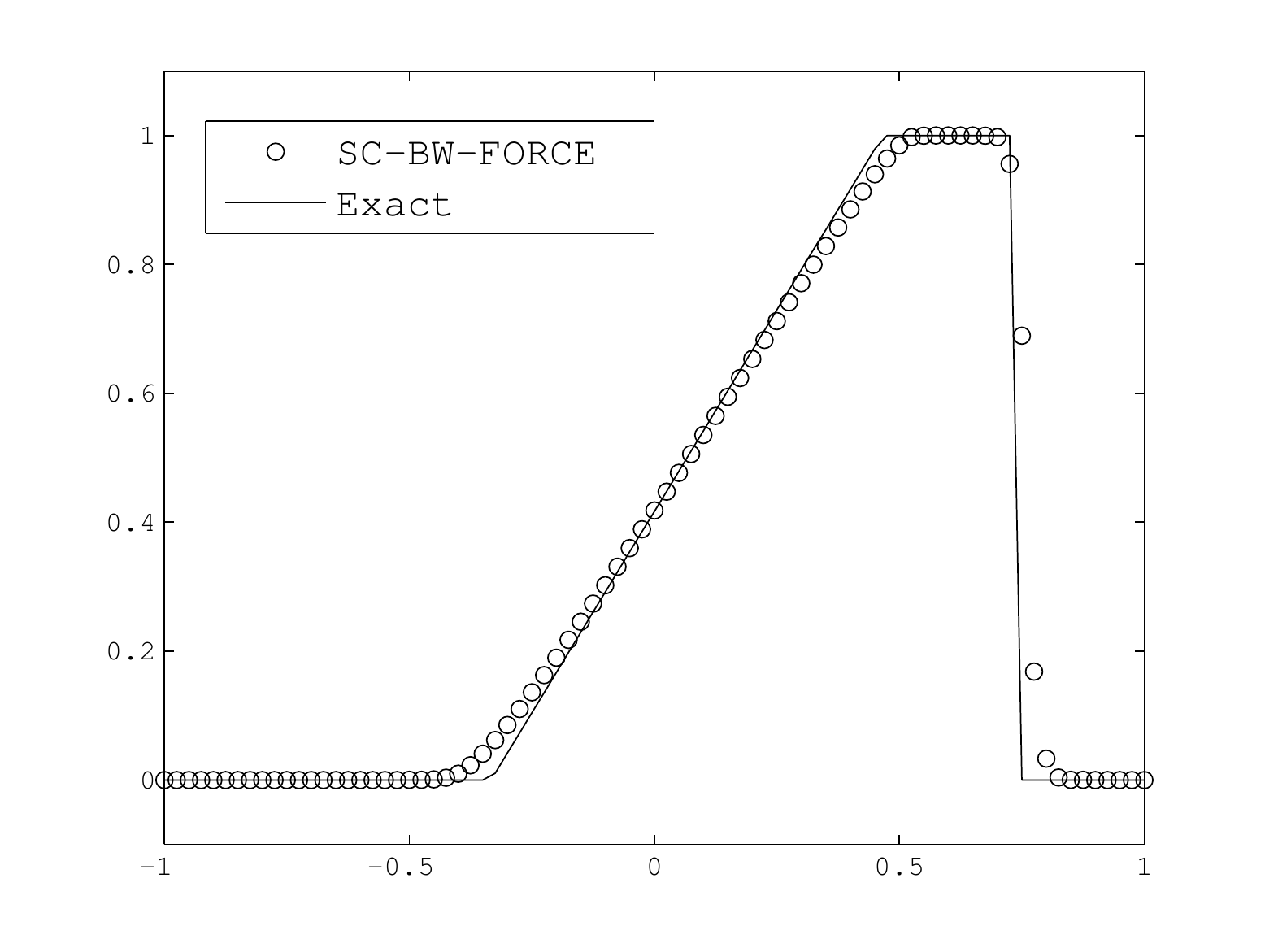} \\
(a) & (b)\\
\end{tabular}
    \caption{\label{Fig6a} Solution at $T =0.8$ using $CFL=0.8,\;
      N=80$, (a) Wrong location of moving shock using Scheme
      \ref{algo1} and \ref{algo2} (b) Shock correction using shock
      switch.}
\end{figure}

\begin{figure}[httb!]
  \begin{tabular}{c|c}
 \includegraphics[scale=0.4]{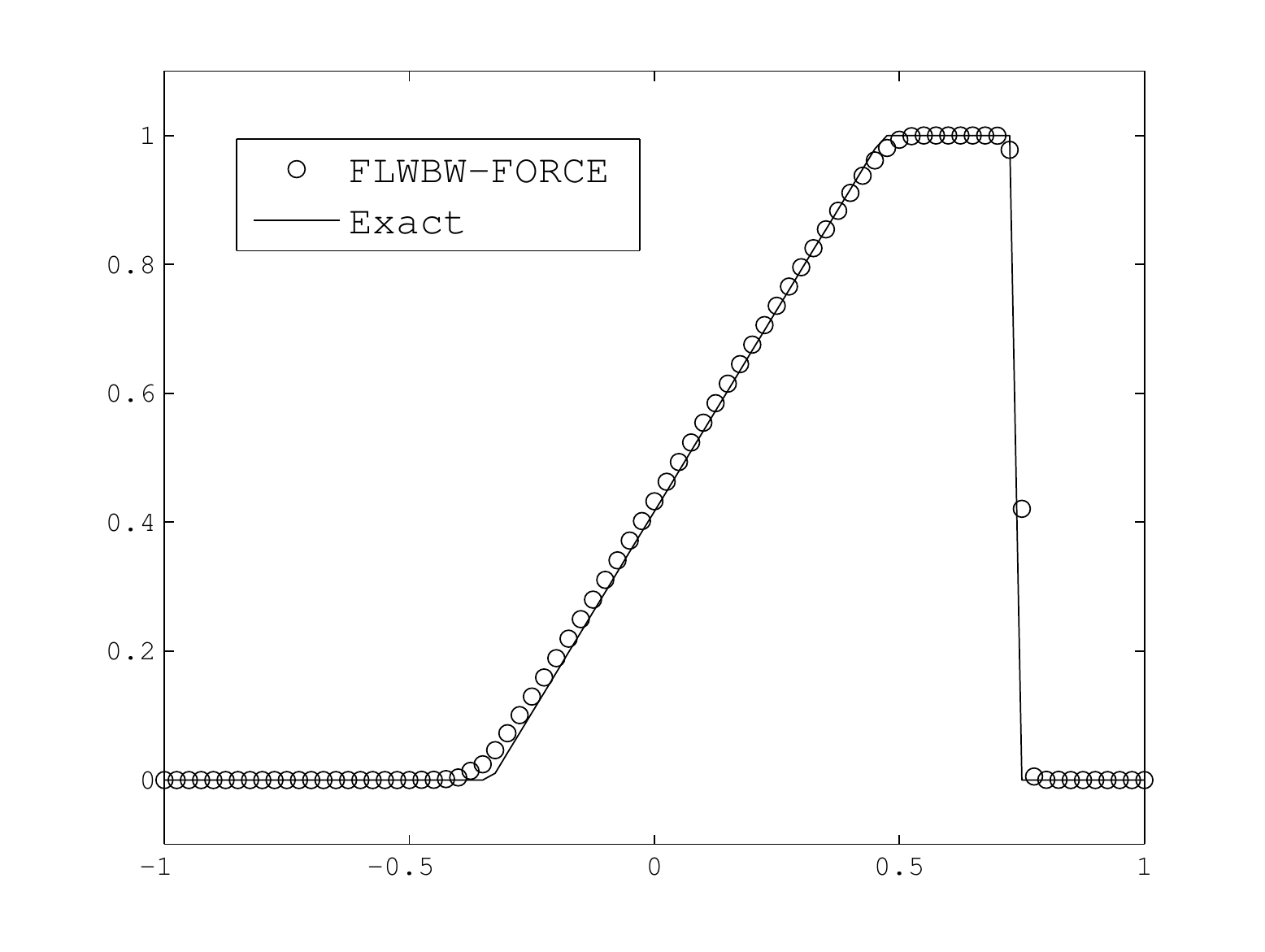}
 &\includegraphics[scale=0.4]{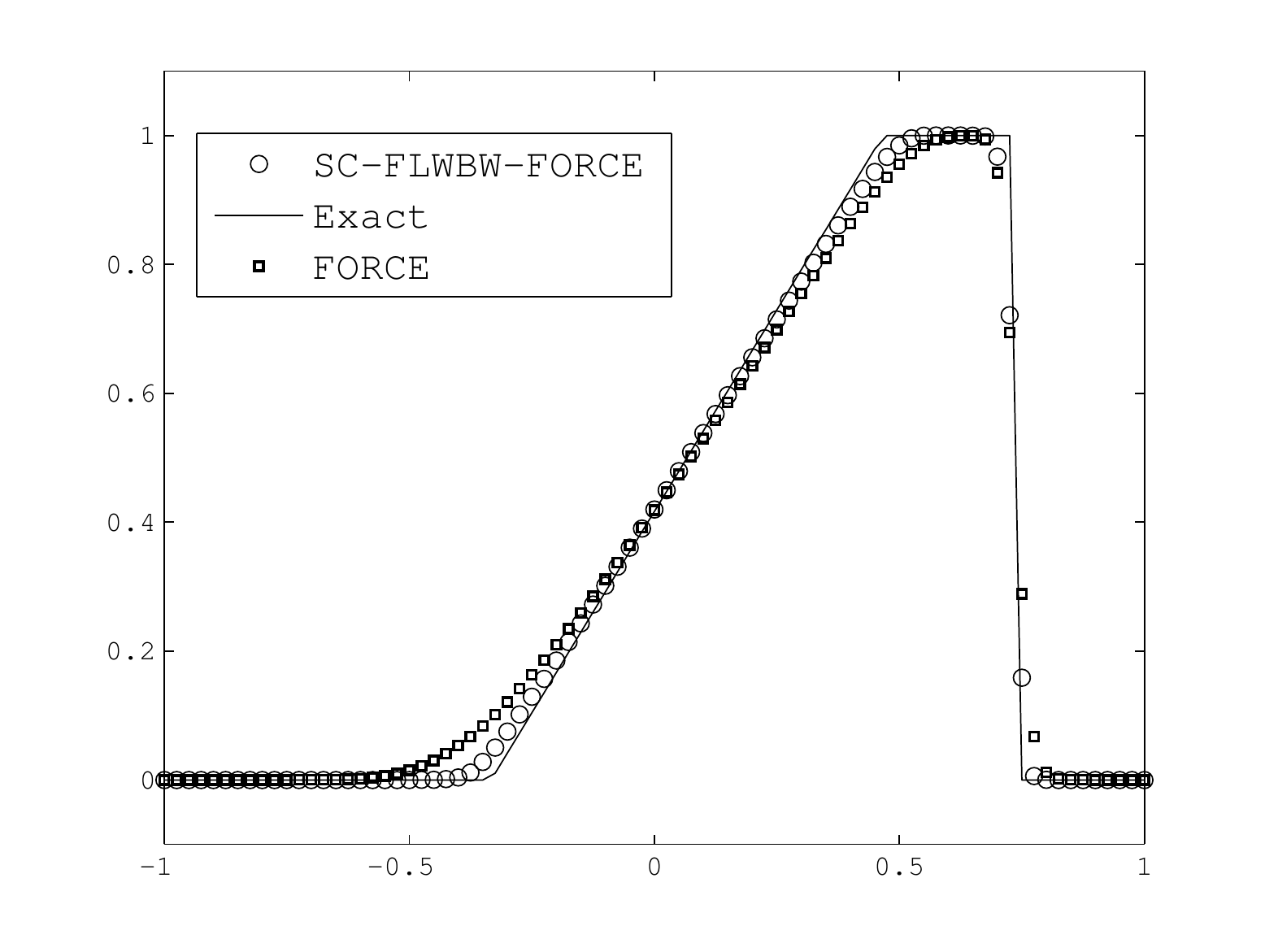}\\
 (a) & (b)\\
\end{tabular}
\caption{\label{Fig6b} Solution at $T =0.8$ using $CFL=0.8,\; N=80$,
  (a) Scheme \ref{algo3} give correct shock location high
  resolution for left rarefaction and crisp capturing for the moving
  shock, (b) Solution by Scheme \ref{algo4}.}
\end{figure}
\subsubsection{\bf Accuracy Check: Smooth Initial conditions}
Consider three different smooth initial
conditions (IC) along with corresponding breaking time $T_{b}$
  \begin{equation}
    u(x,0) = 0.1+ sin^{4}(\pi x),\, x\in[0,\,1],\; T_{b} =
    0.27803225 \label{nonlin-IC2a}
  \end{equation}
  \begin{equation}
    u(x,0) = 0.1+\exp^{-x^{4}},\, x\in [-2:3], \; T_{b} =
    0.65669683. \label{nonlin-IC2b}
  \end{equation} 
\begin{equation}
  u(x,0) = \frac{1}{4}(1+ sin(\pi x)),\, x\in [-1:1], \;
  T_{b}=\frac{4}{\pi}.  \label{nonlin-IC2c}
\end{equation}
IC (\ref{nonlin-IC2a}) and (\ref{nonlin-IC2b}) does not contain any
sonic point whereas IC (\ref{nonlin-IC2c}) has a sonic point at
$x=-0.5$ since $u(-0.5, 0)=0$.  The solution corresponding to IC
(\ref{nonlin-IC2a}) and (\ref{nonlin-IC2b}) develop a {\it moving}
shock followed by a rarefaction fan whereas the moving shock
corresponding to IC (\ref{nonlin-IC2c}) is separated by two
rarefaction fans. In Figure \ref{Fig1burgerIC2a}, \ref{Fig2burgerIC2b}
and \ref{burgerShuosherFig} the pre and post-shock solution of Burgers
equation obtained by the shock corrected hybrid Scheme \ref{algo4}
({\it SC-FLWBW-FORCE}) corresponding to IC (\ref{nonlin-IC2a}),
(\ref{nonlin-IC2b}) and (\ref{nonlin-IC2c}) respectively are
given. The total variation plots are also given for different choices
of CFL number $\lambda \max_{u}|f^{'}(u)|$.  In Table
\ref{Tab1burgerIC2a} to Table \ref{burgerShuosherTab}, $L^{1}$ and $L^{\infty}$
errors are shown at pre-shock time $\displaystyle T_{b}$ using
different CFL number. Results show that the hybrid scheme nicely
approximated pre-shock solution with out clipping error and does not
introduce induced oscillations near shock in the post-shock
solution. Moreover purposed method yields a total variation
diminishing solution and shows a consistent convergence rate between
second and third order in both the norms.
  \begin{figure}[!htb]
    \begin{tabular}{cc}
      \includegraphics[%
      scale=0.5]{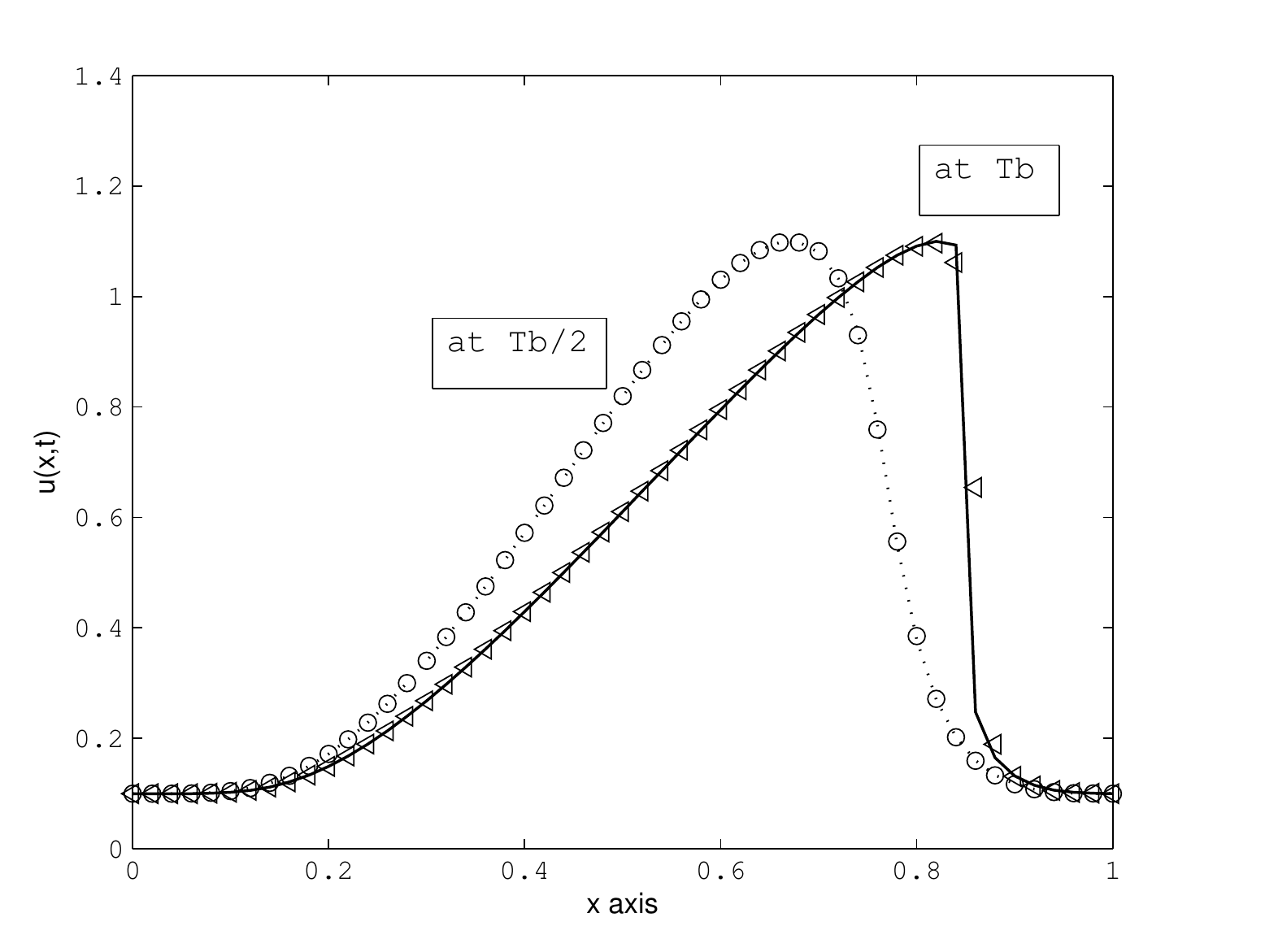}&\includegraphics[%
      scale=0.5]{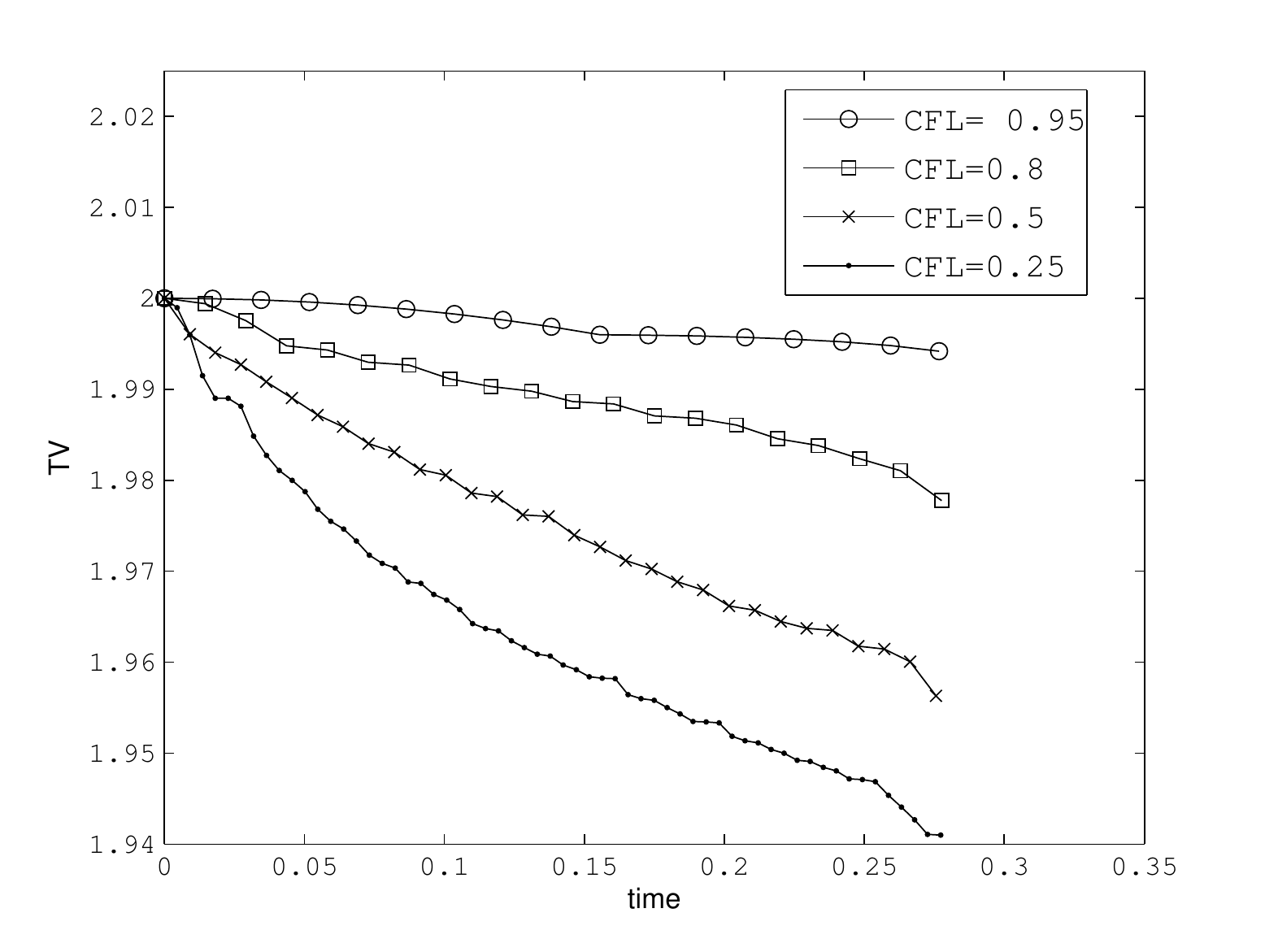}\\
      (a) & (b)\\
    \end{tabular}
    \caption{\label{Fig1burgerIC2a} Burgers equation solution using {\it SC-FLWBW-FORCE}
      corresponding to IC (\ref{nonlin-IC2a}): (a) No clipping error for
      smooth extrema as well near shock zone $CFL=0.95, N=50$, (b) Effect of CFL on total variation of computed solution}
  \end{figure}

    \begin{table}[!htb]
      \begin{tabular}{|c|c|c|}
        \hline N & CFL=0.6 & CFL=0.9\\ \hline
        \begin{tabular}{c}
          \\
          10\\
          20\\
          40\\
          80\\
          160\\
          320\\
          640\\
        \end{tabular}
        &
        \begin{tabular}{cccc}
          $L^{1}$ error & Rate &$L^{\infty}$ error & Rate\\
          \hline
          3.1074e-02 & \dots  &	 1.2720e-02 & \dots\\
          1.3874e-02 & 1.1634 &	 4.6798e-03 & 1.4426\\
          2.0530e-03 & 2.7566 &	 3.9157e-04 & 3.5791e\\
          4.0690e-04 & 2.3350 &	 6.9600e-05 & 2.4921e\\
          8.6630e-05 & 2.2317 &	 1.2510e-05 & 2.4760e\\
          1.4700e-05 & 2.5591 &	 1.2000e-06 & 3.3820e\\
          2.4800e-06 & 2.5674 &	 5.0000e-08 & 4.5850e\\
        \end{tabular}
        &
        \begin{tabular}{cccc}
          $L^{1}$ Error & Rate &$L^{\infty}$ error & Rate\\
          \hline
         2.7767e-02 & \dots  & 	 1.0627e-02 & \dots\\
         5.8287e-03 & 2.2521 & 	 2.5638e-03 & 2.0514\\
         1.0398e-03 & 2.4868 & 	 2.2157e-04 & 3.5324\\
         2.3176e-04 & 2.1657 & 	 3.1460e-05 & 2.8162\\
         5.2120e-05 & 2.1527 & 	 3.9000e-06 & 3.0120\\
         1.0380e-05 & 2.3280 & 	 4.1000e-07 & 3.2498\\
         2.4200e-06 & 2.1007 & 	 5.0000e-08 & 3.0356\\
        \end{tabular}\\
        \hline
      \end{tabular}
      \caption{\label{Tab1burgerIC2a} Convergence rate corresponding to IC (\ref{nonlin-IC2a}) at time $T=T_{b}/2$}
    \end{table}

  \begin{table}[!htb]
    \begin{tabular}{|c|c|}
\hline
        CFL=0.45 & CFL=0.95\\
        \hline
        \begin{tabular}{ccccc}
          N& $L^{1} error$ & Rate& $L^{\infty} error$ & Rate\\
          \hline
          10 & 3.1120e-01 & \dots & 	 1.0335e-01 & \dots\\
          20 &  6.4051e-02 & 2.2805 & 	 2.7574e-02 & 1.9061\\
          40 &  8.7373e-03 & 2.8740 & 	 2.9299e-03 & 3.2344\\
          80 &  1.9428e-03 & 2.1690 & 	 4.6987e-04 & 2.6405\\
          160 &  4.1491e-04 & 2.2273 & 	 6.1350e-05 & 2.9371\\
          320 &  9.0750e-05 & 2.1928 & 	 6.7500e-06 & 3.1841\\
          640 &  2.0870e-05 & 2.1205 & 	 7.3000e-07 & 3.2089\\
        \end{tabular}
      &
        \begin{tabular}{cccc}
          $L^{1} error$ & Rate & $L^{\infty} error$ & Rate\\
          \hline
          2.5190e-01 & \dots & 	 9.7632e-02 &\dots\\
          4.7036e-02 & 2.4210 & 	 2.2831e-02 & 2.0964\\
          1.2098e-02 & 1.9590 & 	 5.2950e-03 & 2.1083\\
          2.3084e-03 & 2.3899 & 	 5.5543e-04 & 3.2530\\
          4.2048e-04 & 2.4568 & 	 5.5800e-05 & 3.3153\\
          8.6830e-05 & 2.2758 & 	 5.3900e-06 & 3.3719\\
          1.9470e-05 & 2.1569 & 	 5.6000e-07 & 3.2668\\
        \end{tabular}\\
\hline
\end{tabular}
\caption{\label{Tab1burgerIC2b} Third order $L^{\infty}$ convergence rate corresponding to IC (\ref{nonlin-IC2b}) at pre-Shock time Tb/2, $Tb=0.65669683$}
  \end{table}

  \begin{figure}[!htb]
    \begin{center}
      \begin{tabular}{cc}
      \includegraphics[%
      scale=0.45]{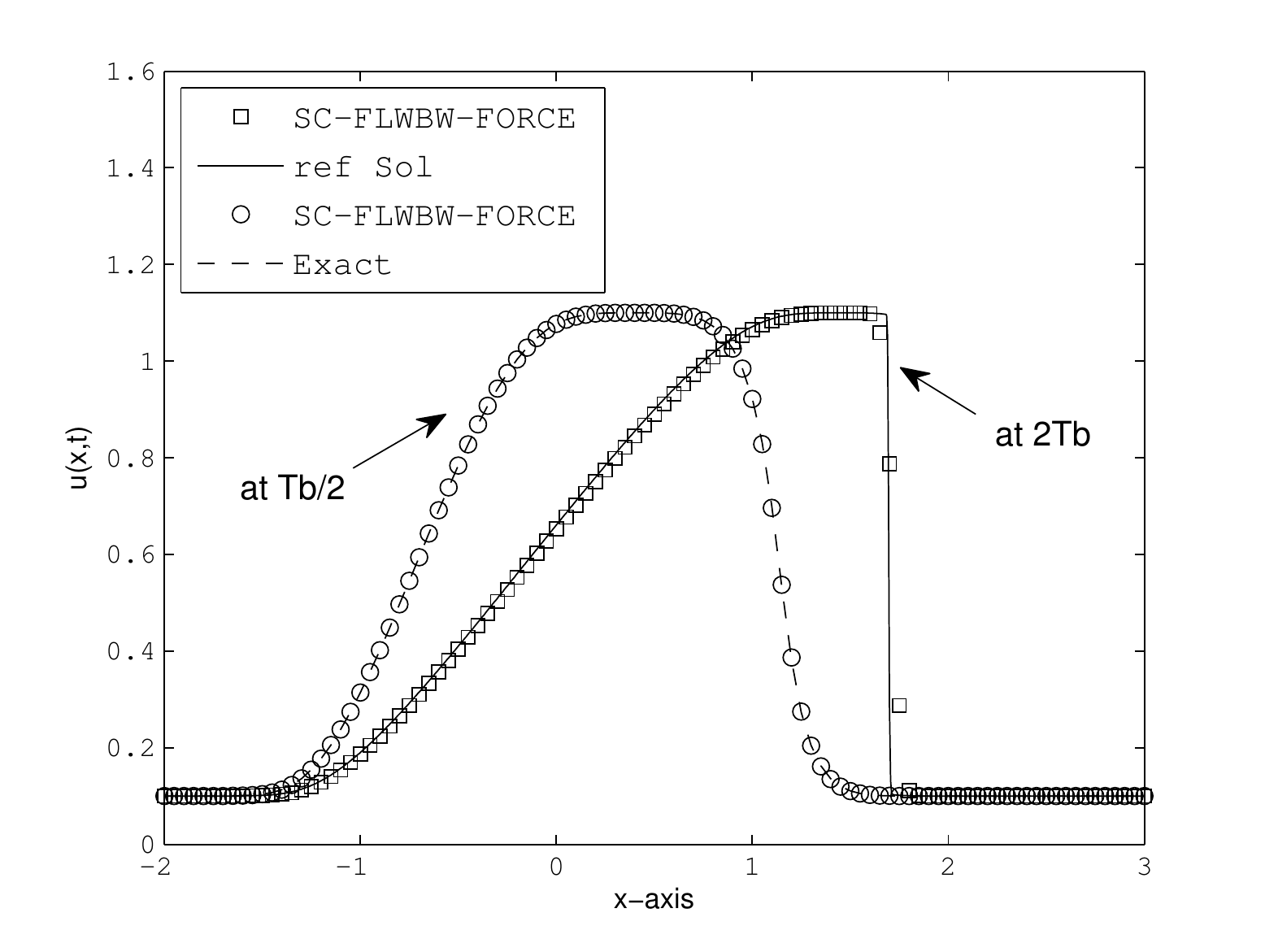}&\includegraphics[%
      scale=0.45]{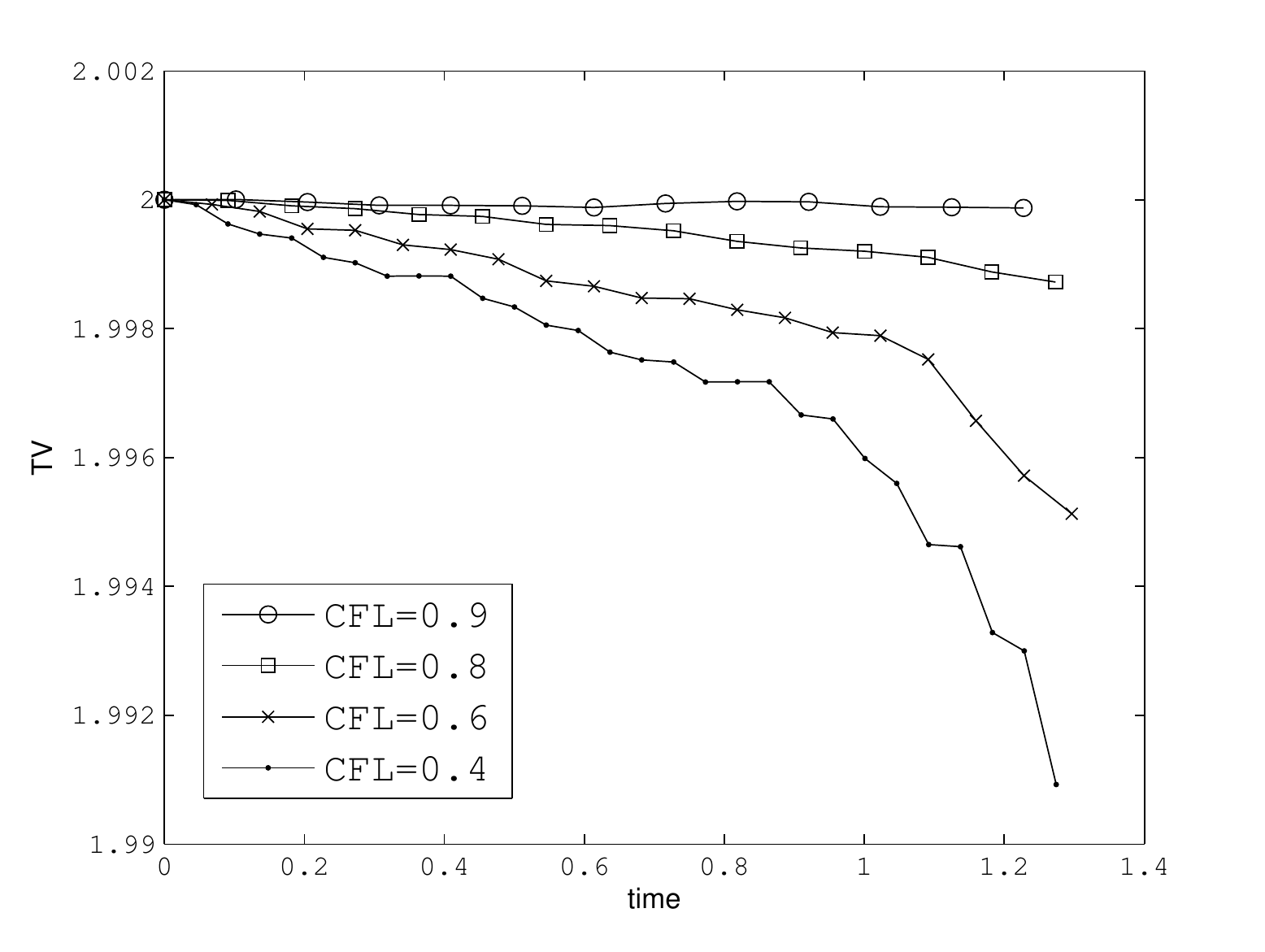}\\
      (a) & (b)\\
      \end{tabular}
    \end{center}
    \caption{\label{Fig2burgerIC2b} Solution corresponding to IC
      (\ref{nonlin-IC2b}): (a) No clipping error for smooth extrema as
      well near shock zone $CFL=0.8, N=80$, (b) Effect of CFL on total variation diminishing plot of computed solution.}
  \end{figure}

\begin{figure}[!htb]
\begin{tabular}{cc}
\includegraphics[scale=0.45]{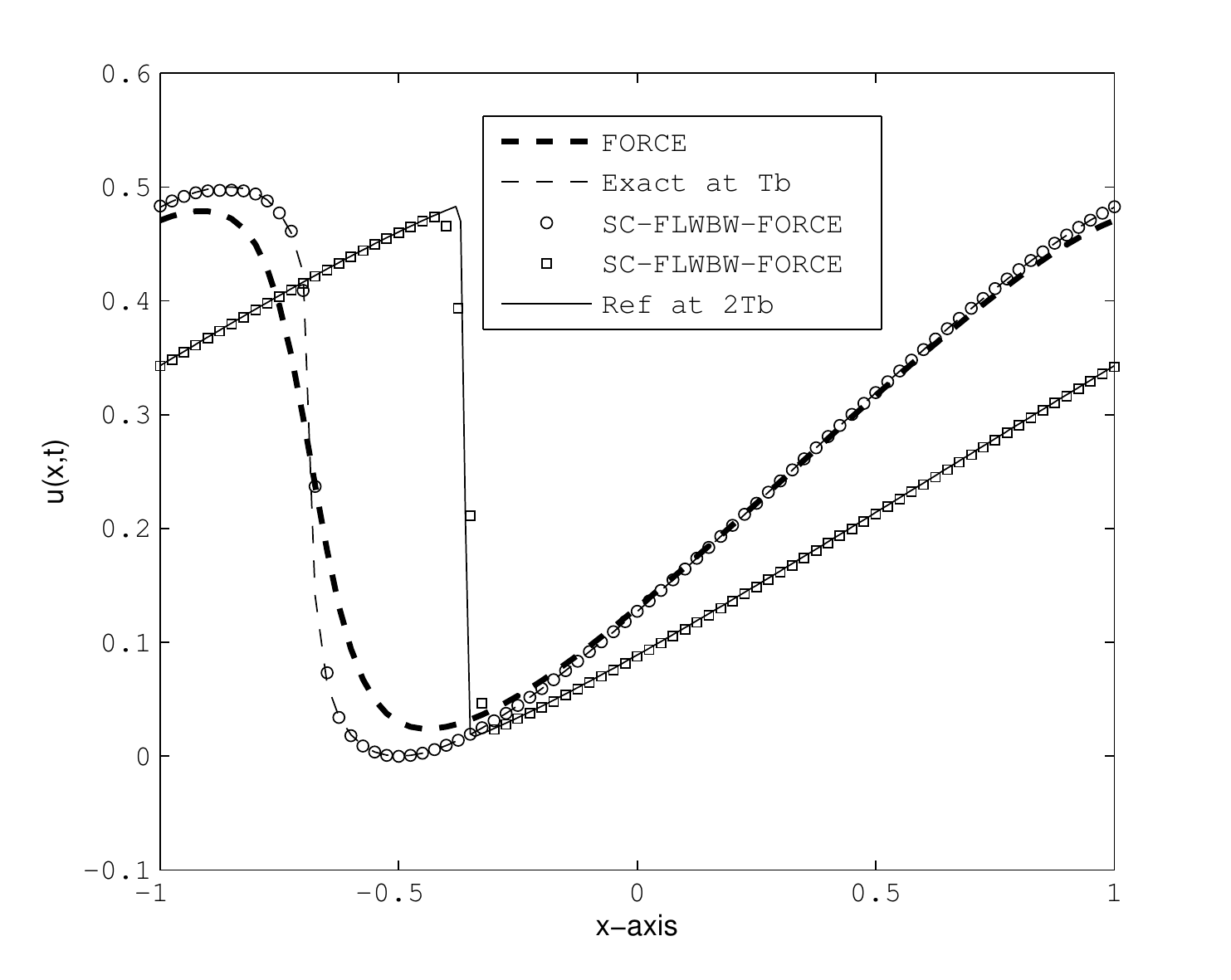}& \includegraphics[scale=0.45]{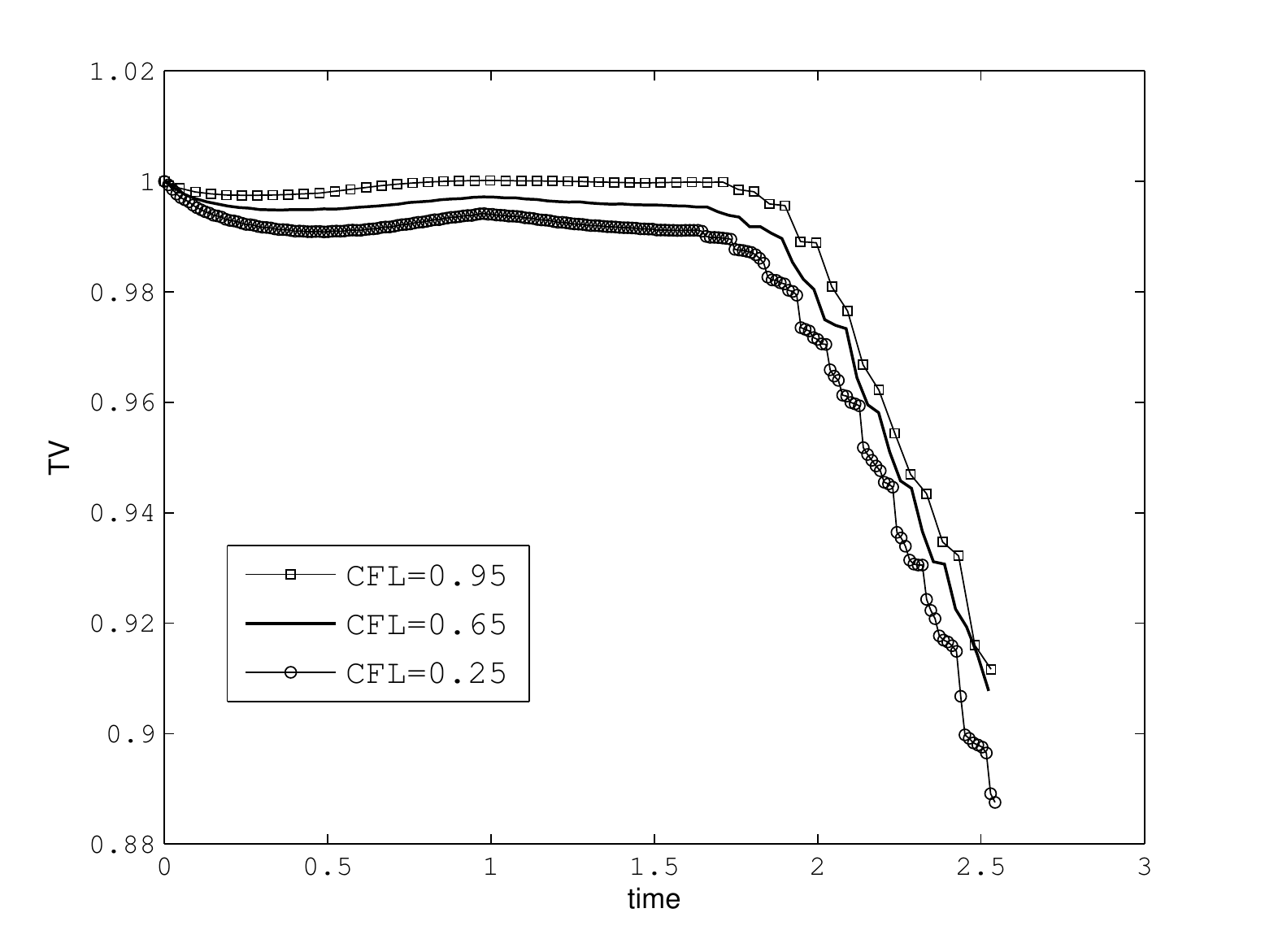}\\
(a) & (b)\\
\end{tabular}
\caption{\label{burgerShuosherFig} (a) Comparison of computed solution
  of corresponding to IC (\ref{nonlin-IC2c}) at  $Tb=1/4\pi$ $CFL=0.6, N=80$ (b) Total variation diminishing plot up to
  $t=2Tb$}
\end{figure}
\begin{table}[!htb]
\begin{tabular}{|ccccc|cccc|}
\hline
& & CFL=0.45 & &  & & CFL=0.9 & & \\
\hline
N& $L^{1}$ error& Rate & $L^{\infty}$ error& Rate & $L^{1}$ error& Rate & $L^{\infty}$ error& Rate\\
\hline
10  & 3.9495e-02 & 0.0000 & 	 1.0467e-02 & 0.0000 & 1.9593e-02 & 0.0000 & 	 4.7957e-03 & 0.0000\\
20  & 7.2108e-03 & 2.4534 & 	 1.8537e-03 & 2.4974 & 4.6493e-03 & 2.0753 & 	 1.1326e-03 & 2.0821\\
40  & 2.1212e-03 & 1.7653 & 	 3.7326e-04 & 2.3122 & 1.0845e-03 & 2.1000 & 	 2.4774e-04 & 2.1927\\
80  & 4.9995e-04 & 2.0850 & 	 7.0710e-05 & 2.4002 & 2.7441e-04 & 1.9826 & 	 4.2200e-05 & 2.5535\\
160 & 9.2620e-05 & 2.4324 & 	 1.7250e-05 & 2.0353 & 6.6870e-05 & 2.0369 & 	 8.3800e-06 & 2.3322\\
320 & 2.2230e-05 & 2.0588 & 	 4.1900e-06 & 2.0416 & 1.6580e-05 & 2.0119 & 	 1.7400e-06 & 2.2679 \\
640 & 5.5900e-06 & 1.9916 & 	 1.0400e-06 & 2.0104 & 4.1900e-06 & 1.9844 & 	 3.6000e-07 & 2.2730 \\
\hline
\end{tabular}
\caption{\label{burgerShuosherTab} Convergence rate for test case corresponding to IC (\ref{nonlin-IC2c})  at pre-shock time $t=2/\pi$ $CFL=0.8$}
\end{table}

\subsection{Buckley Leverett Equation}
Consider Buckley-Leverett equation which has convex-concave flux. This equation physically represents 
the flow of a mixture of oil and water through a porous medium. 
\begin{equation}
\frac{\partial u}{\partial t} +\frac{\partial f(u)}{\partial x} =0\label{buckley}
\end{equation}
The flux function is given by,
\begin{equation}
f(u) = \frac{u^2}{u^2 + \alpha(1-u)^2}.
\end{equation}
Here $\alpha$ is viscosity ratio and $u$ represents the saturation of water and 
lies between $0$ and $1.$
\subsubsection{One moving shock }
Consider equation (\ref{buckley}) with $\alpha =\frac{1}{2}$ and initial condition
\begin{equation}
u(x,\,0)= \left\{
\begin{array}{ll} 
  1, & x<0,\\ 
  0,  & x>0.
\end{array}
\right. \label{ICbuckley1}
\end{equation}
The solution involves one single moving shock followed by an rarefaction wave. 
\subsubsection{Two moving shock}
Consider equation (\ref{buckley}) with $\alpha =\frac{1}{4}$ and 
subject to initial condition
\begin{equation}
u(x,\,0)= \left\{
\begin{array}{ll} 
  1, & -0.5\leq x\leq 0,\\ 
  0,  & elsewhere.
\end{array}
\right. \label{ICbuckley2}
\end{equation}
The solution involves two moving shocks, each followed by an
rarefaction wave. In numerical simulation flux limited high resolution
LxW TVD scheme \cite{Rider} is used in hybrid scheme \ref{algo4} as
CCS. The results corresponding IC (\ref{ICbuckley1}) and
(\ref{ICbuckley2}) are given in Figure \ref{Figbuckley}(a) and
\ref{Figbuckley}(b) respectively. Results show that the proposed
scheme sharply captures both the fast and slow shocks. The rarefaction
waves are also approximated with high resolution.

\begin{figure}[!htb]
\begin{tabular}{cc}
\hspace{-1cm}\includegraphics[%
scale =0.55,angle =0]{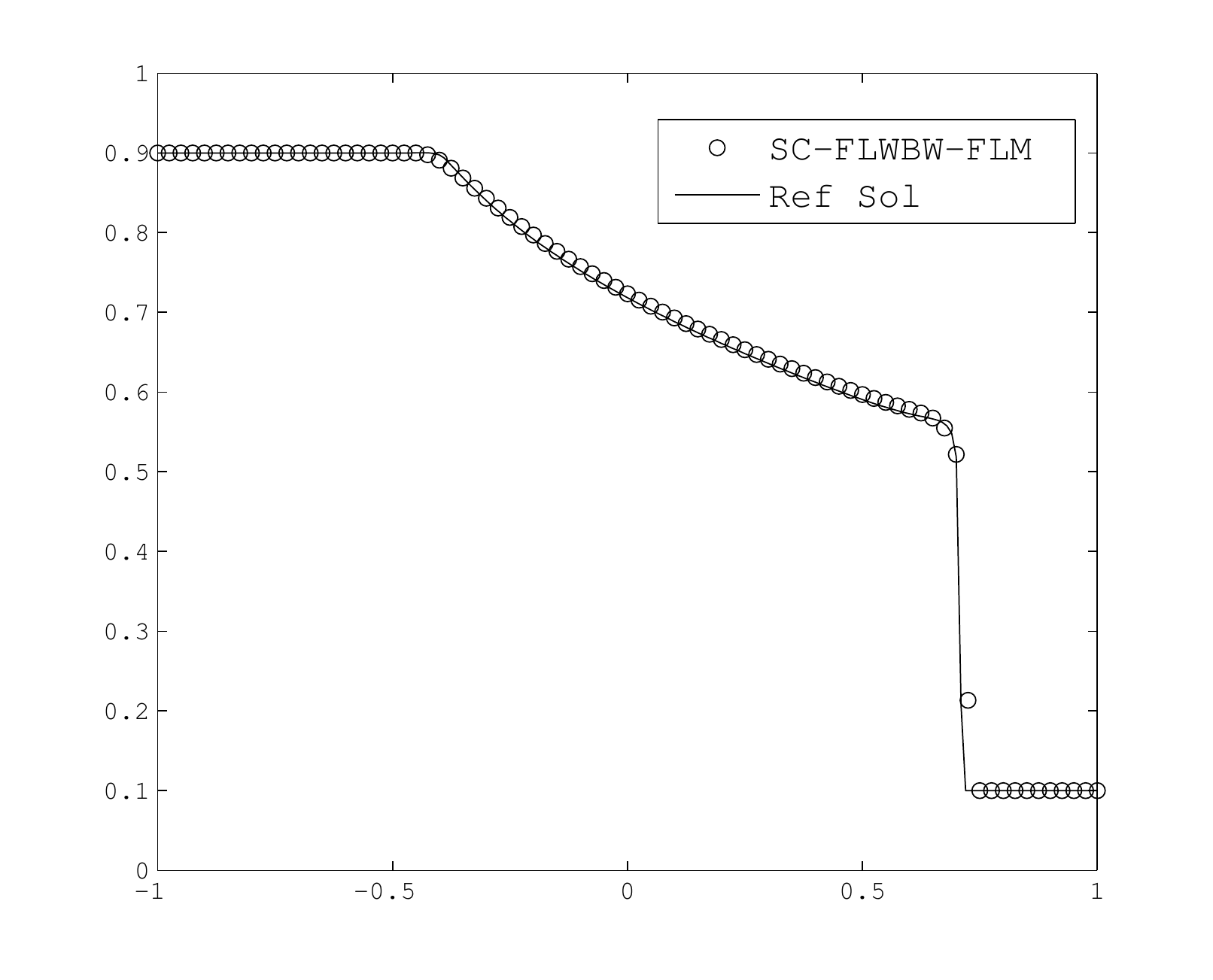}&\hspace{-1cm}\includegraphics[%
scale =0.55,angle =0]{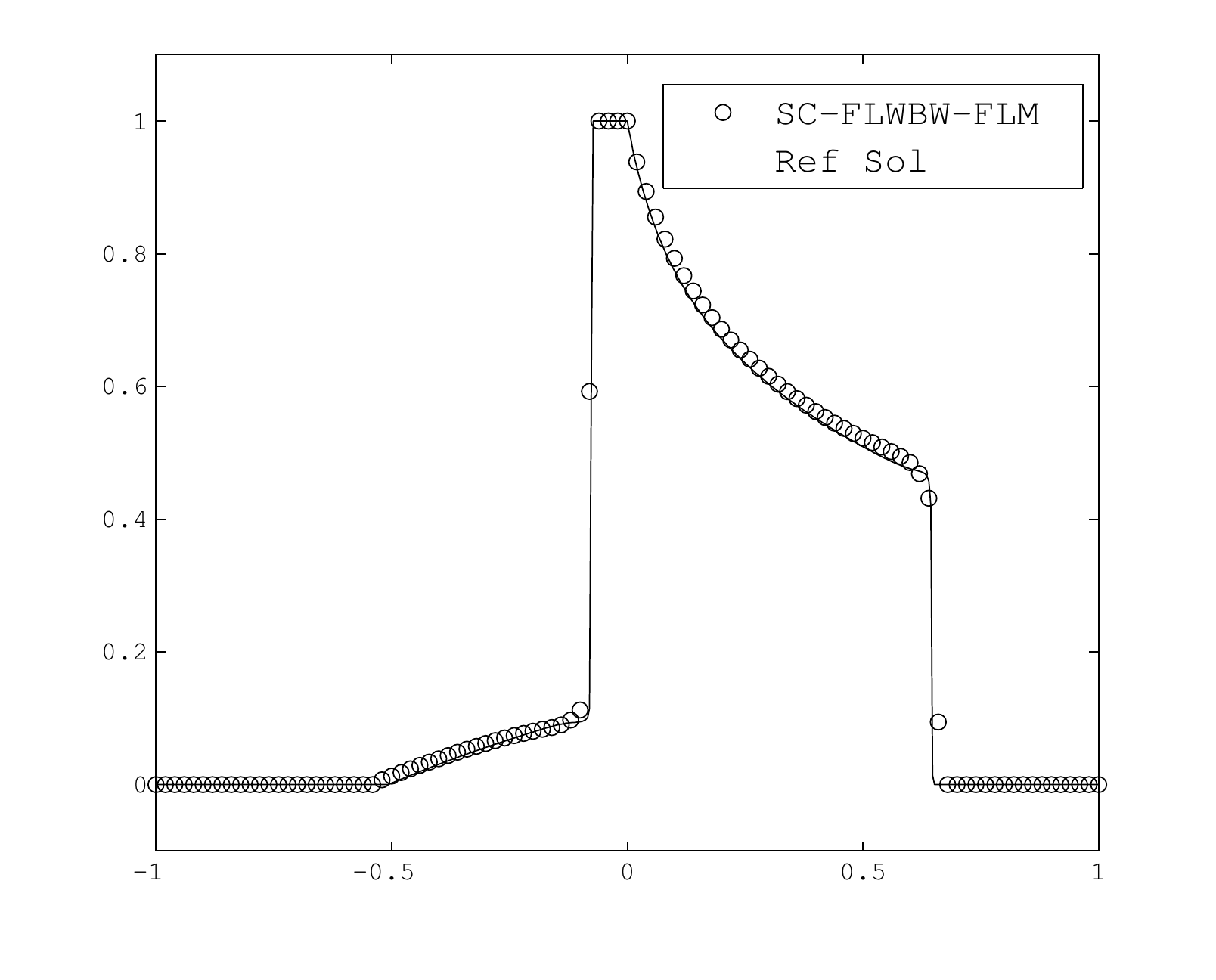}\\
(a) & (b)\\
\end{tabular}
\caption{\label{Figbuckley}{\it Numerical solution using $CFL=0.8$ (a) $N=80$ at time $t=0.75$ (b) $N=100$ at time $t=0.4$: Sharp
  resolution for rarefaction fans and slow and fast moving shocks.}}
\end{figure}
\subsection{\bf 1D Euler Equation}
The 1D Euler equations of the Gas dynamics is given by
\begin{equation}
  \frac{\partial}{\partial
    t}\textbf{u}+\frac{\partial}{\partial x}\textbf{F(u)}=0, \label{euler1d}
\end{equation}
where $\textbf{u}=\left(\begin{array}{c}\rho\\ \rho
  u\\ E\end{array}\right)$ and
  $\textbf{F(u)}=\left(\begin{array}{c}\rho u\\ \rho
    u^2+p\\ (E+p)u \end{array}\right)$ denotes vector of conservative
  variables and conservative fluxes respectively. Variables $\rho,u $
  and $p$ represents density, velocity and pressure respectively . The
  total energy $e$ is defined by,
\begin{equation}
  e=\frac{p}{\gamma-1}+\frac{\rho u^2}{2}
\end{equation}
where $\gamma$ is the ratio of specific heat coefficients. We consider
the four shock tube problems modeled by (\ref{euler1d}) to check the
robustness of proposed scheme in section \ref{algo4sys}.  These shock
tube tests check any method in capturing the contact and shock
discontinuity along with non-oscillatory high resolution approximation
for smooth extrema. In all the numerical test a simple high resolution
TVD flux limited centered (FLIC) \cite{toro2000,toro2009} scheme with
MINBEE limiter is used as CCS in \ref{algo4}.  We denote results by
this scheme by FLWBW-FLIC instead SC-FLWBW-FLIC. Numerical results are
compared with FLIC to see the improvement in capturing the solution
profile by FLWBW-FLIC. Note that, the MINBEE limiter satisfies the
universal TVD stability region given in \cite{Dubey2013} and therefore
robustly works for both positive and negative characteristics speed
associated with system (\ref{euler1d})
\subsubsection{Shu-Osher shock tube test \cite{shuosher}}
\begin{equation}
(\rho,u,p)=\begin{array}{ll} (3.857143,2.629369,10.3333)  & x<-4.0,\\
(1+0.2\sin(5x),0,1) & x\geq -4.0. \label{1dEulerIC1}
\end{array}
\end{equation}
This test depicts shock interaction with a sine wave in density. The
main challenge in this case is to capture both the complex small-scale
smooth flow and shocks. In Figure \ref{FigShuosher1DEuler} results are
presented an compared with FLIC scheme. It is evident from zoomed
figure \ref{FigShuosher1DEuler}(b) that the FLWBW-FLIC yields
oscillation free approximation for shock with higher resolution
compared to FLIC for complex oscillatory solution region about
$[0.5,2.5]$. It also capture the smooth region in around $[-3,0.5]$
with out clipping or flattening error which is due to improved
approximation of smooth extrema and steep gradient region.
\begin{figure}[!htb]
\begin{tabular}{cc}
\hspace{-1cm}\includegraphics[scale=0.55]{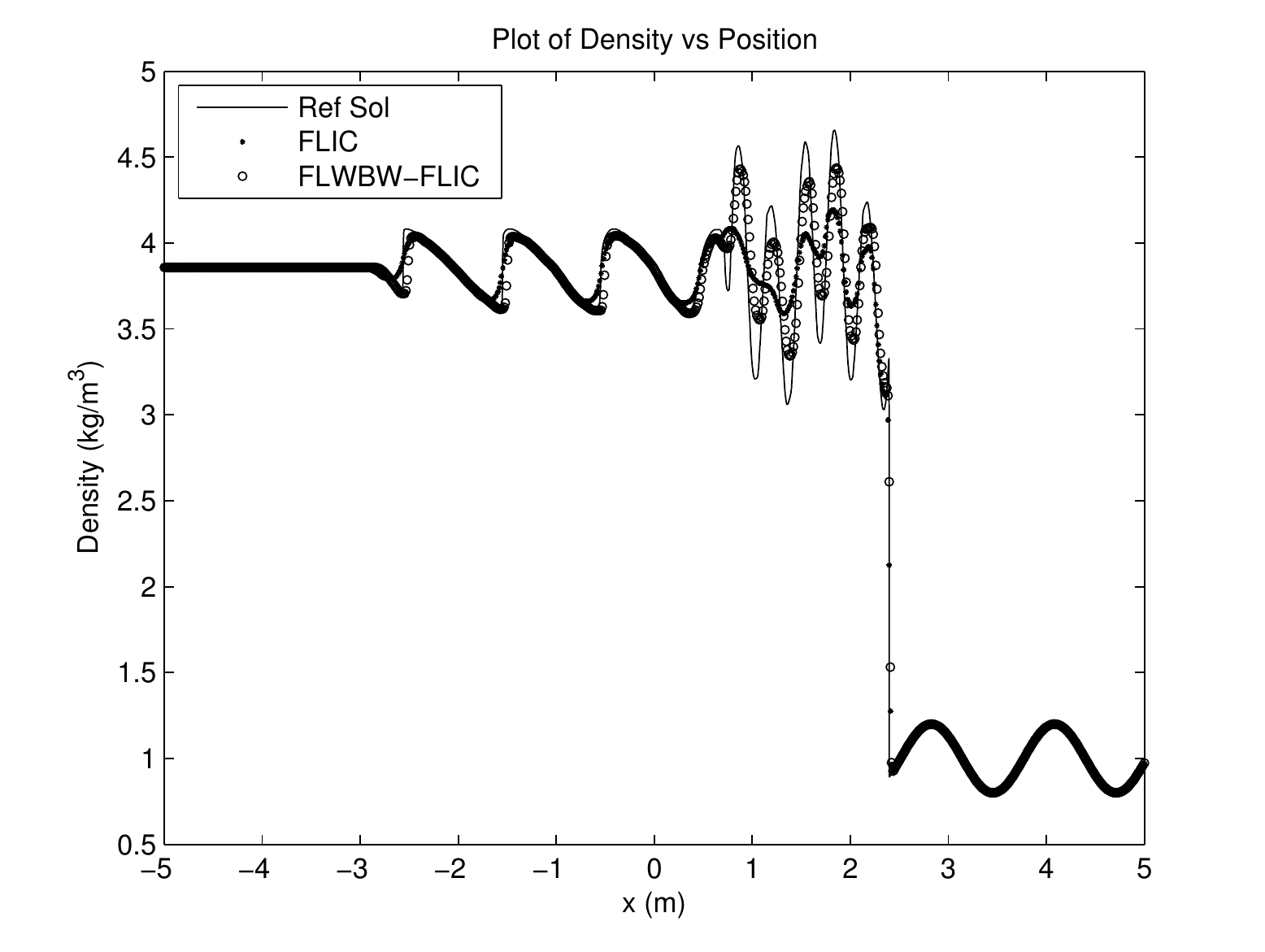}&\hspace{-1cm}\includegraphics[scale=0.55]{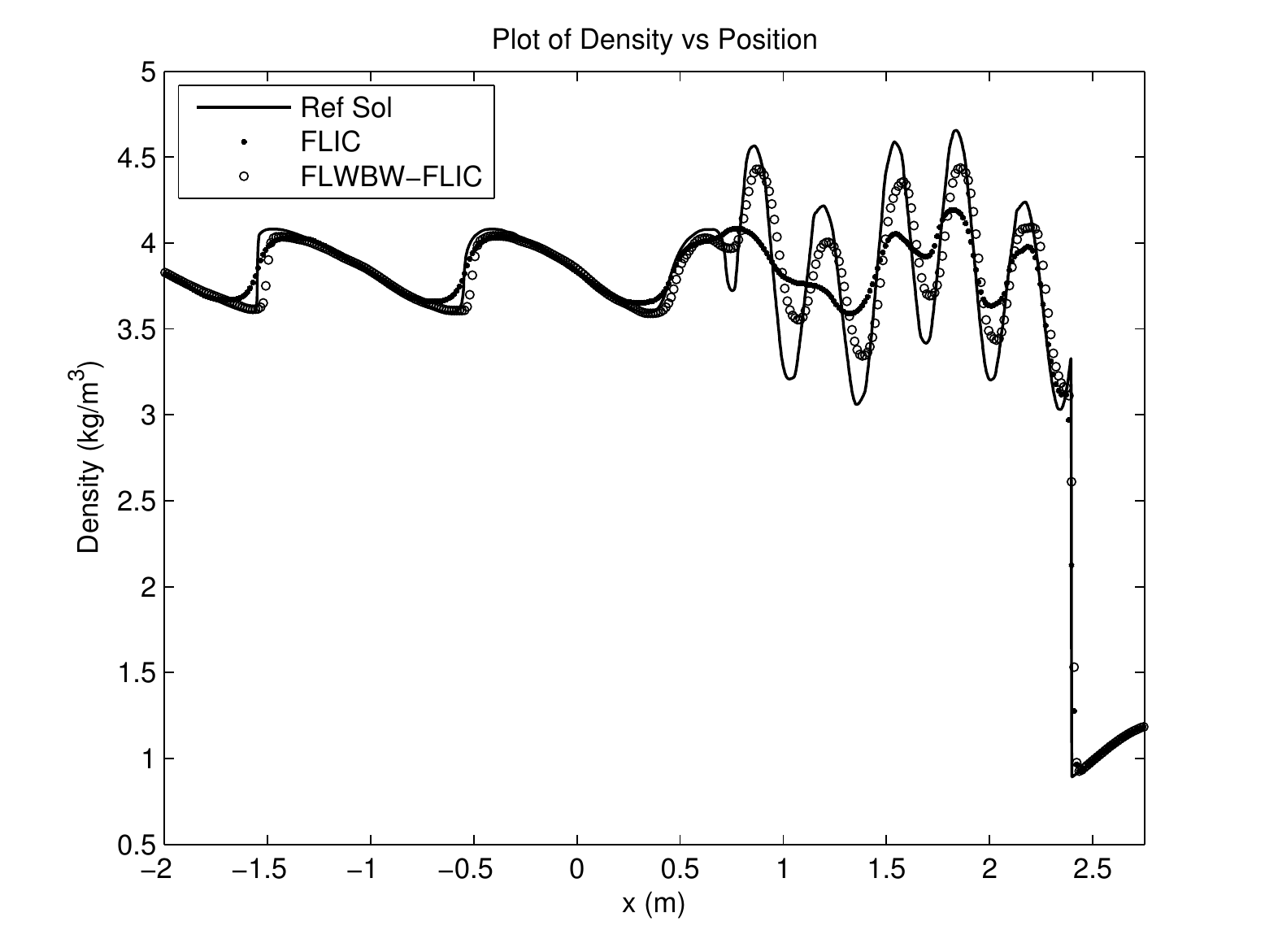}\\
(a) & (b)\\
\end{tabular}
\caption{\label{FigShuosher1DEuler} Numerical solution of shock
  entropy wave interaction $CFL=0.8, N=800$ using shock switch
  parameters $\epsilon=1\times10^{-8}, \delta=0.8$: high resolution of
  smooth extrema and steep gradient region.}
\end{figure}
\subsubsection{Sod test tube}
\begin{equation}
(\rho,u,p)=\begin{array}{ll} (1\; kg/m^3 , 0\; m/s, 100,000\; N/m^2)  & x<0\\
(0.125\; kg/m^3,  0\;m/s , 10,000\; N/m^2) & x\geq 0; \label{SodIC}
\end{array}, x\in[-10,10].
\end{equation}
This test problem has no sonic point but the contact and shock are
very close which cause a smeared approximation to the middle contact
discontinuity. In Figure \ref{Fig1Sod}, numerical results are given and
for different choice of shock switch parameters and compared with
FLIC. Results show that proposed FLWBW-FLIC crisply captures the
smooth rarefaction and contact discontinuity and shock more accurately
than high order TVD scheme FLIC with Minbee limiter.
\begin{figure}[!htb]
\begin{tabular}{cc}
  \includegraphics[scale=0.4]{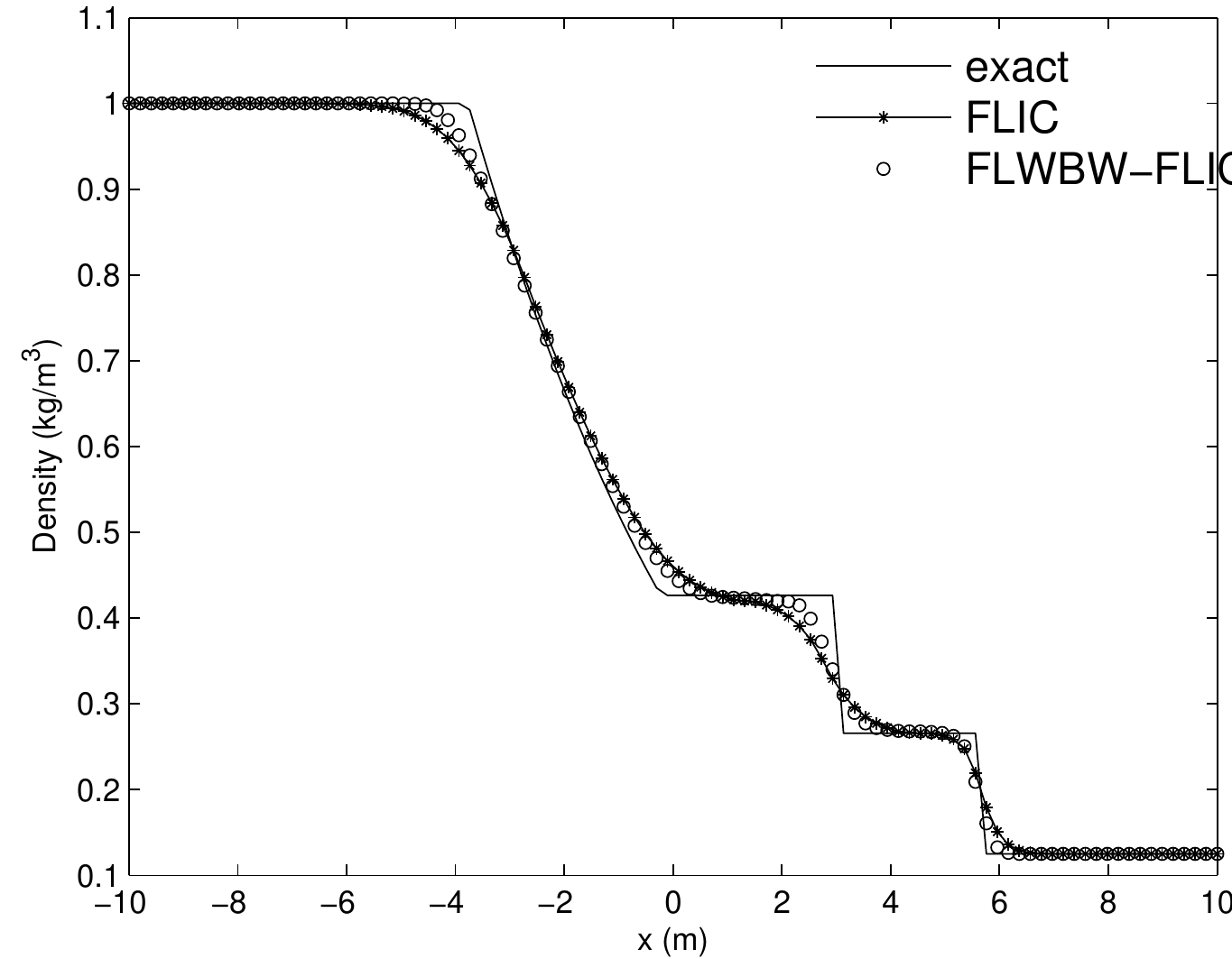} & 
  \includegraphics[scale=0.4]{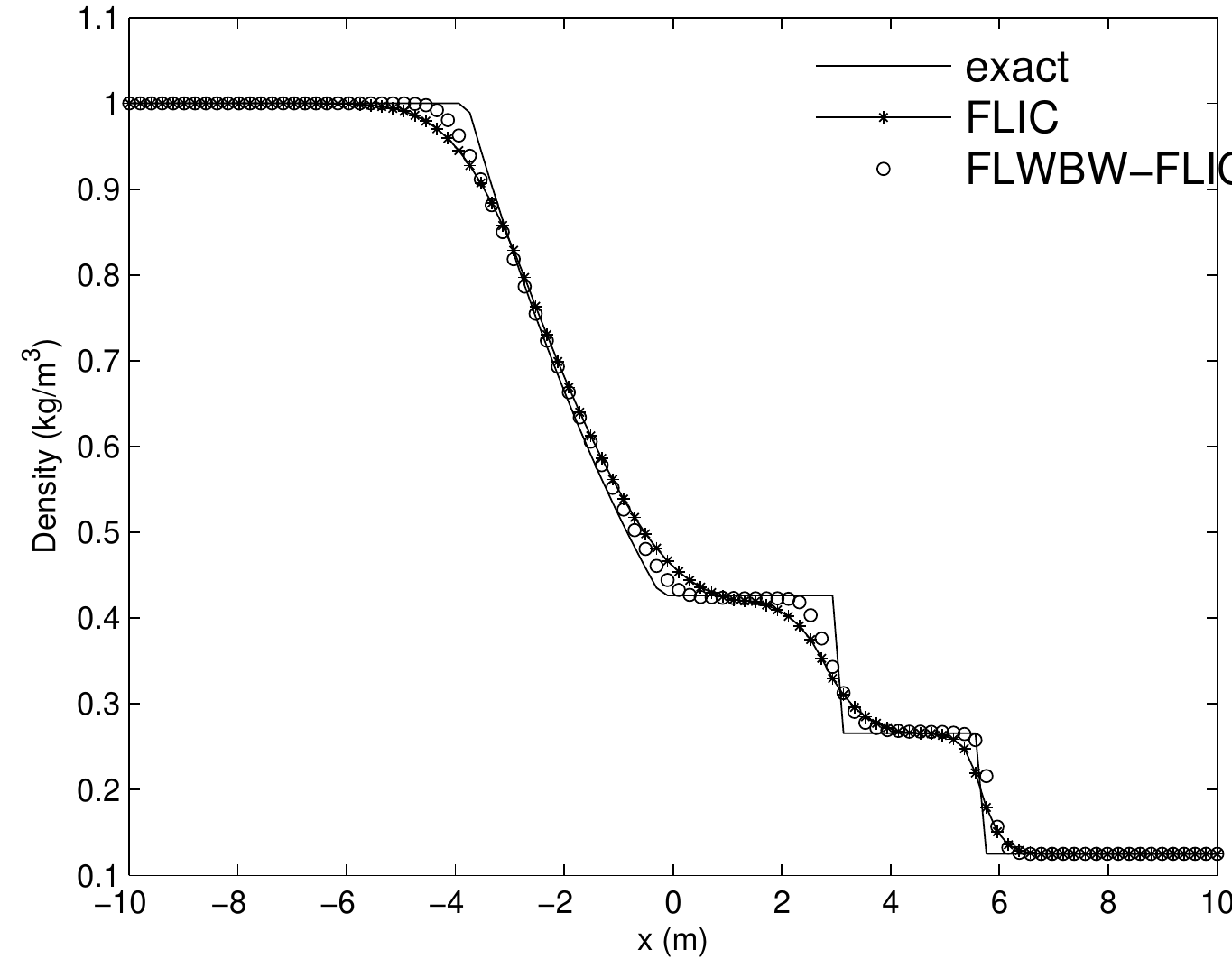}\\
  \includegraphics[scale=0.4]{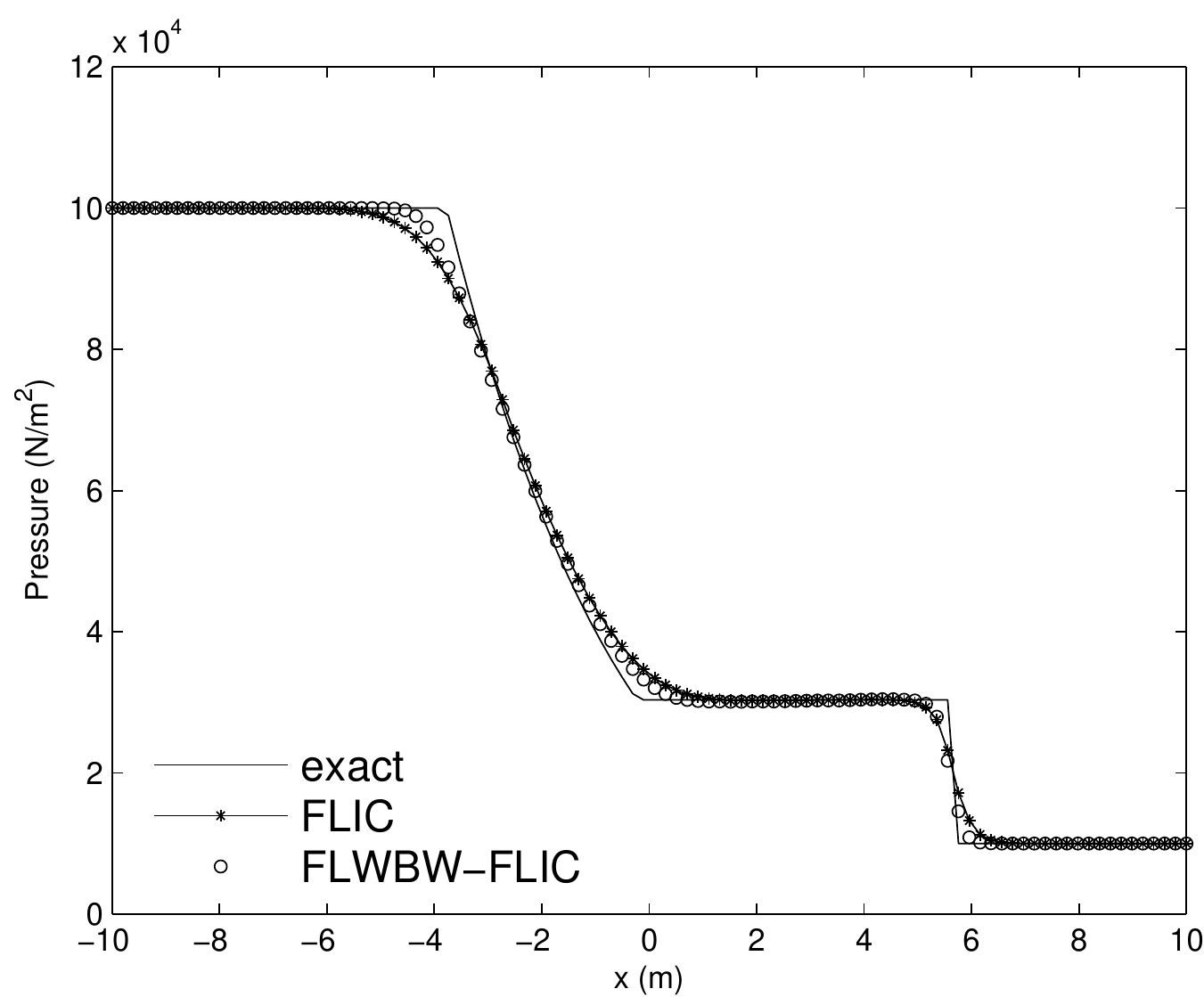}&
  \includegraphics[scale=0.4]{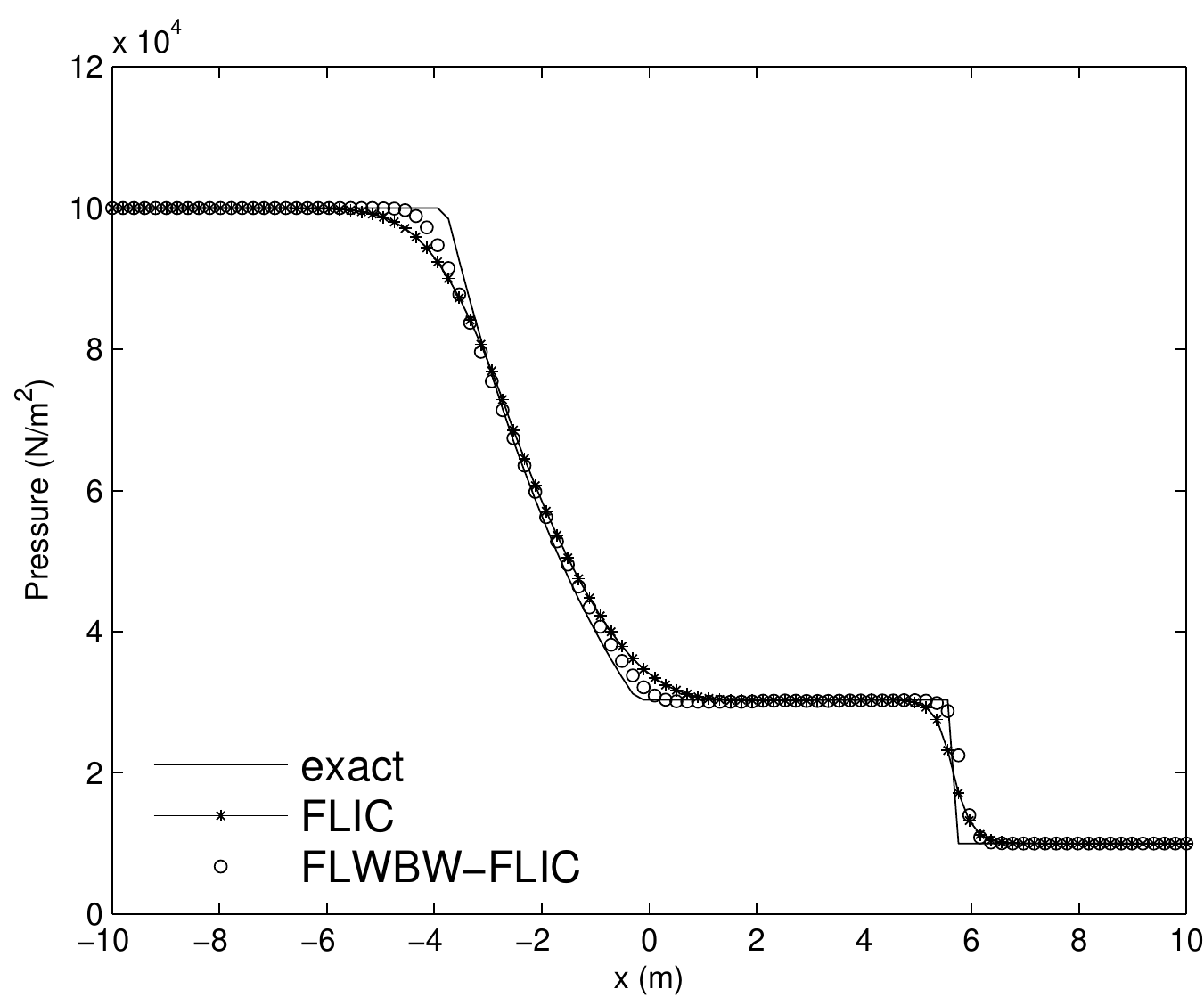}\\
  \includegraphics[scale=0.4]{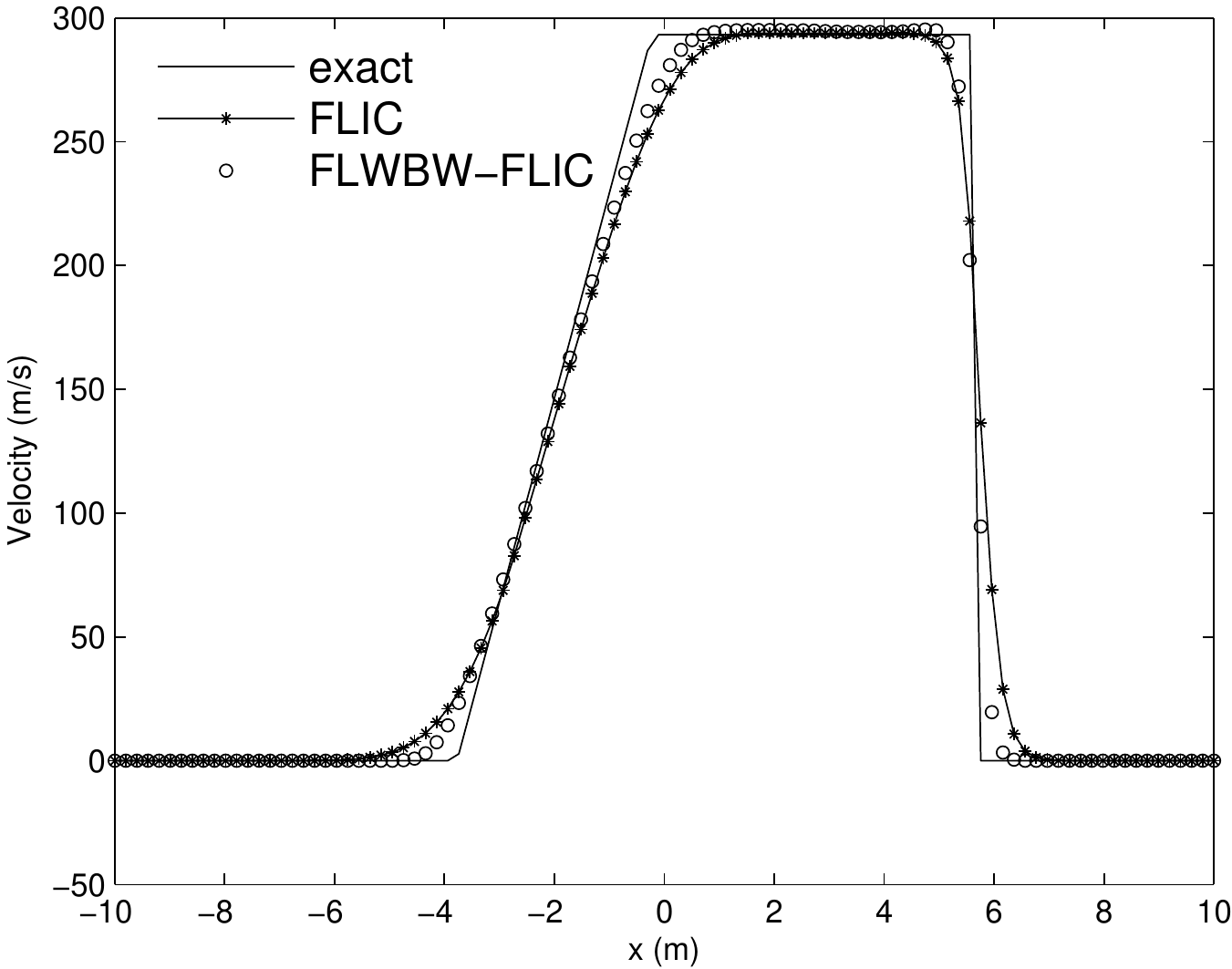}&
  \includegraphics[scale=0.4]{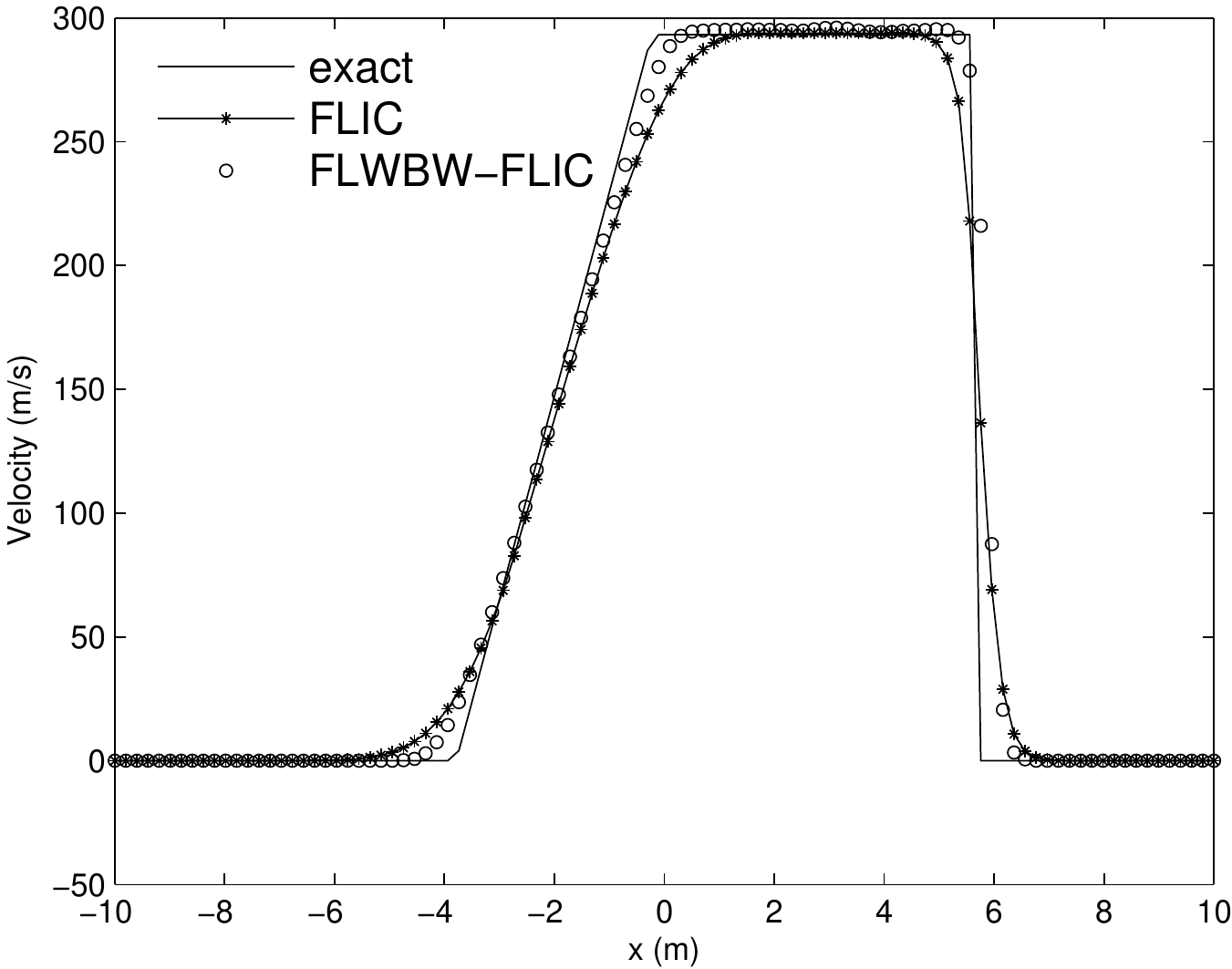}\\
  \includegraphics[scale=0.4]{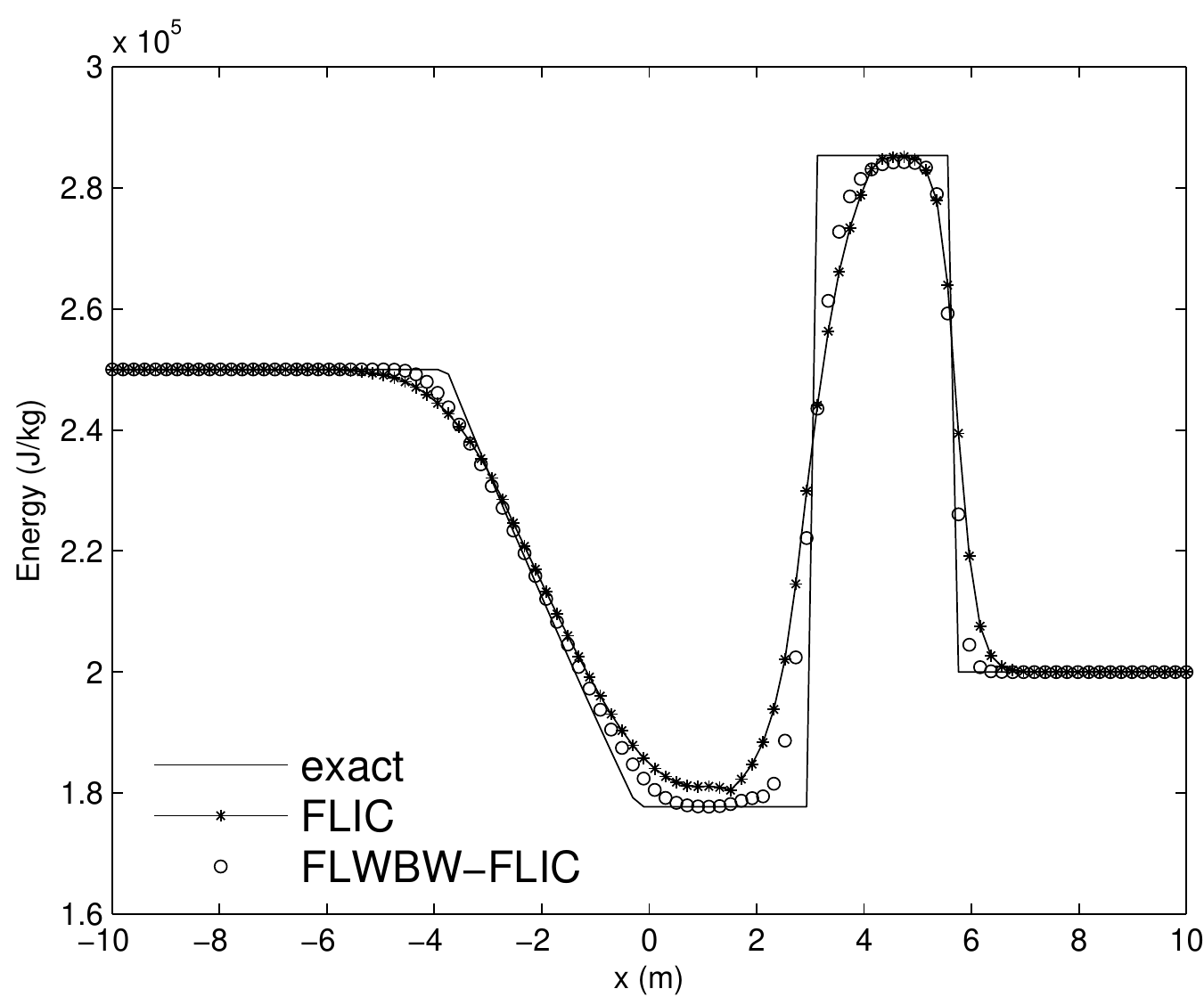}&
  \includegraphics[scale=0.4]{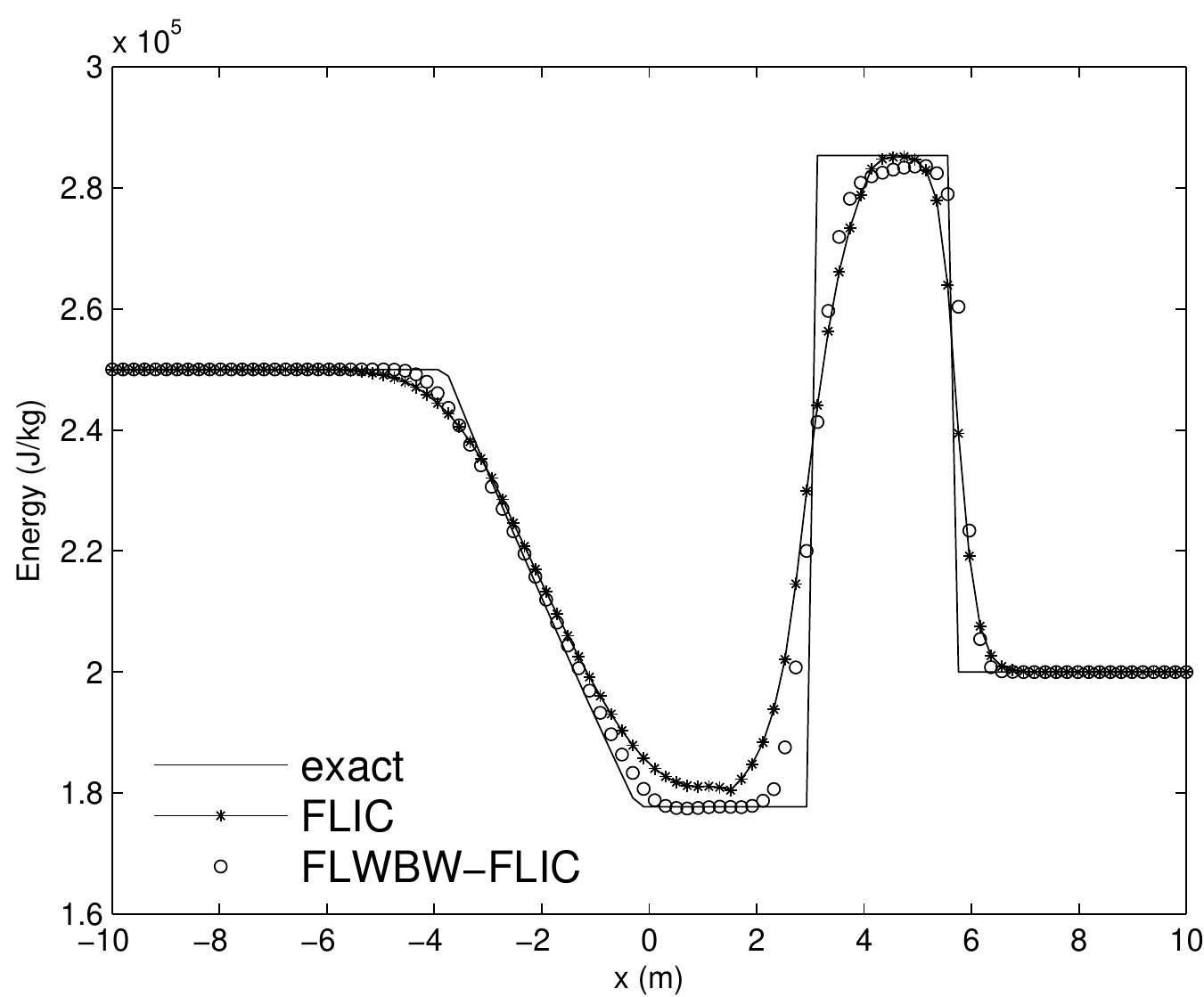}\\
(a) & (b) \\
\end{tabular}
  \caption{\label{Fig1Sod} Solution Sod shock tube, $N=100, CFL=0.5$
    after $67$ time steps at $T= 0.01$ second with shock switch
    parameters $\epsilon=1\times10^{-8}$, (a) $\delta=0.4$, (b)  $\delta=0.9$ }
\end{figure}
\subsubsection{Lax Tube}
\begin{equation}
(\rho,u,p)=\begin{array}{ll} (0.445\; kg/m^3 , 0.698\; m/s, 3.528\; N/m^2)  & x<1,\\
(0.5\; kg/m^3,  0\;m/s , 0.571\; N/m^2) & x\geq 1; \label{LaxIC}
\end{array}  x\in[0,2].
\end{equation}
Compared to Sod tube the shock in this case is very strong and mostly
use to check the robustness of any schemes. In Figure \ref{FigLax}
numerical results obtained by FLWBW-FLIC are given. It can be seen
that that the method capture the contact and the rarefaction wave with higher resolution  
shock compared to FLIC for various choices of shock parameters.

\begin{figure}[!htb]
\begin{tabular}{ccc}
\raisebox{3cm}{\bf (a)} &  \includegraphics[scale=0.36]{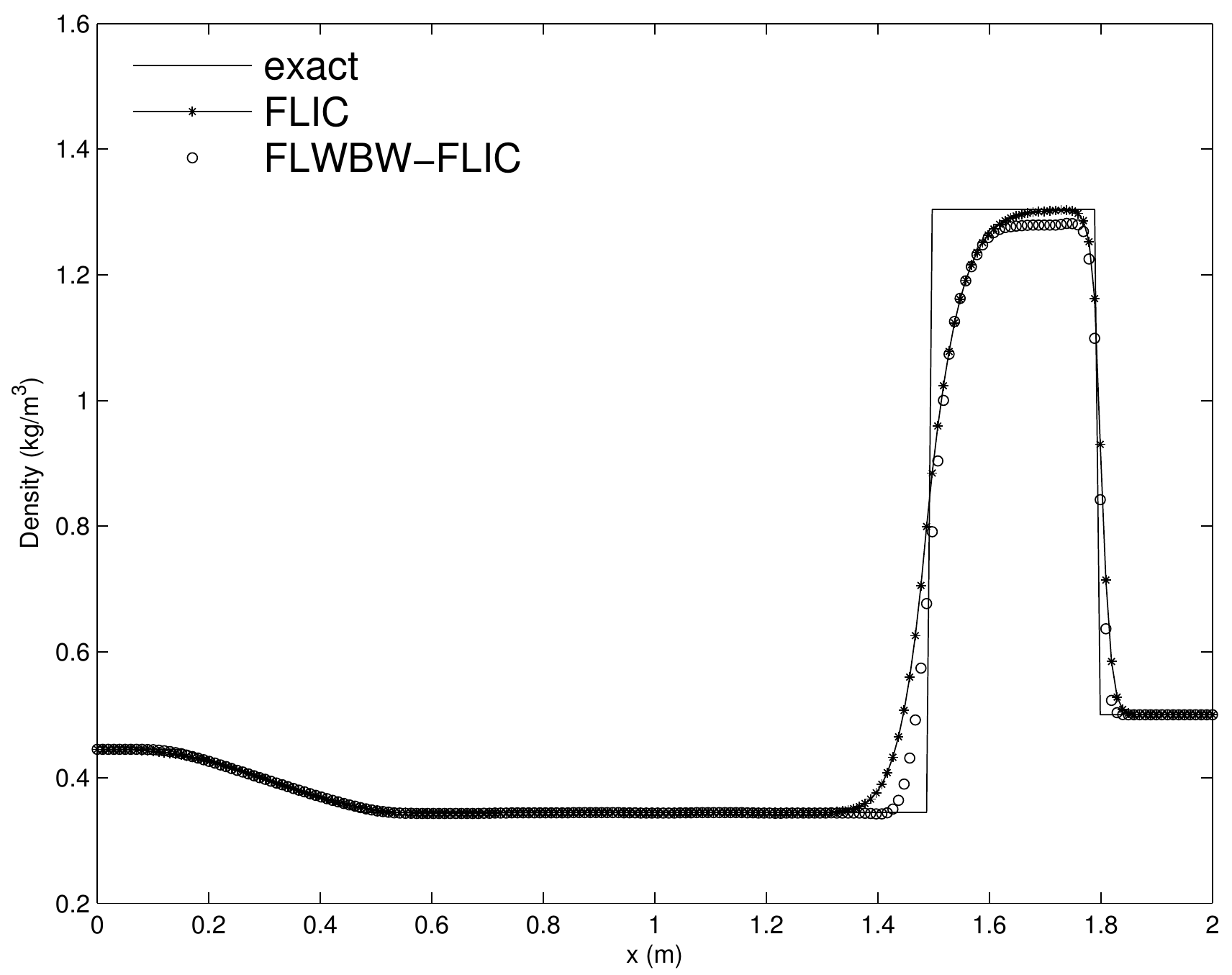}& \includegraphics[scale=0.36]{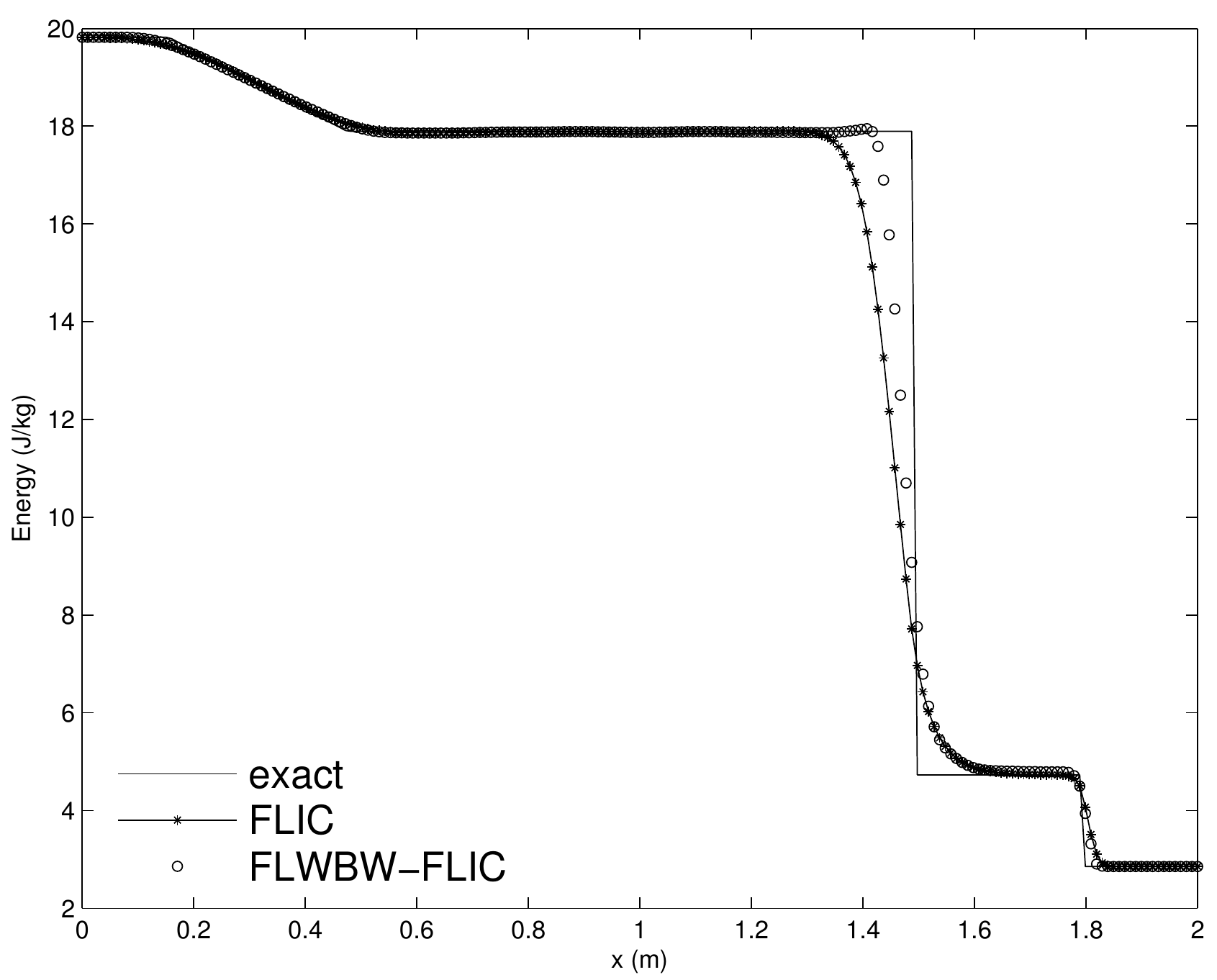}\\
\raisebox{3cm}{\bf (b)} &  \includegraphics[scale=0.36]{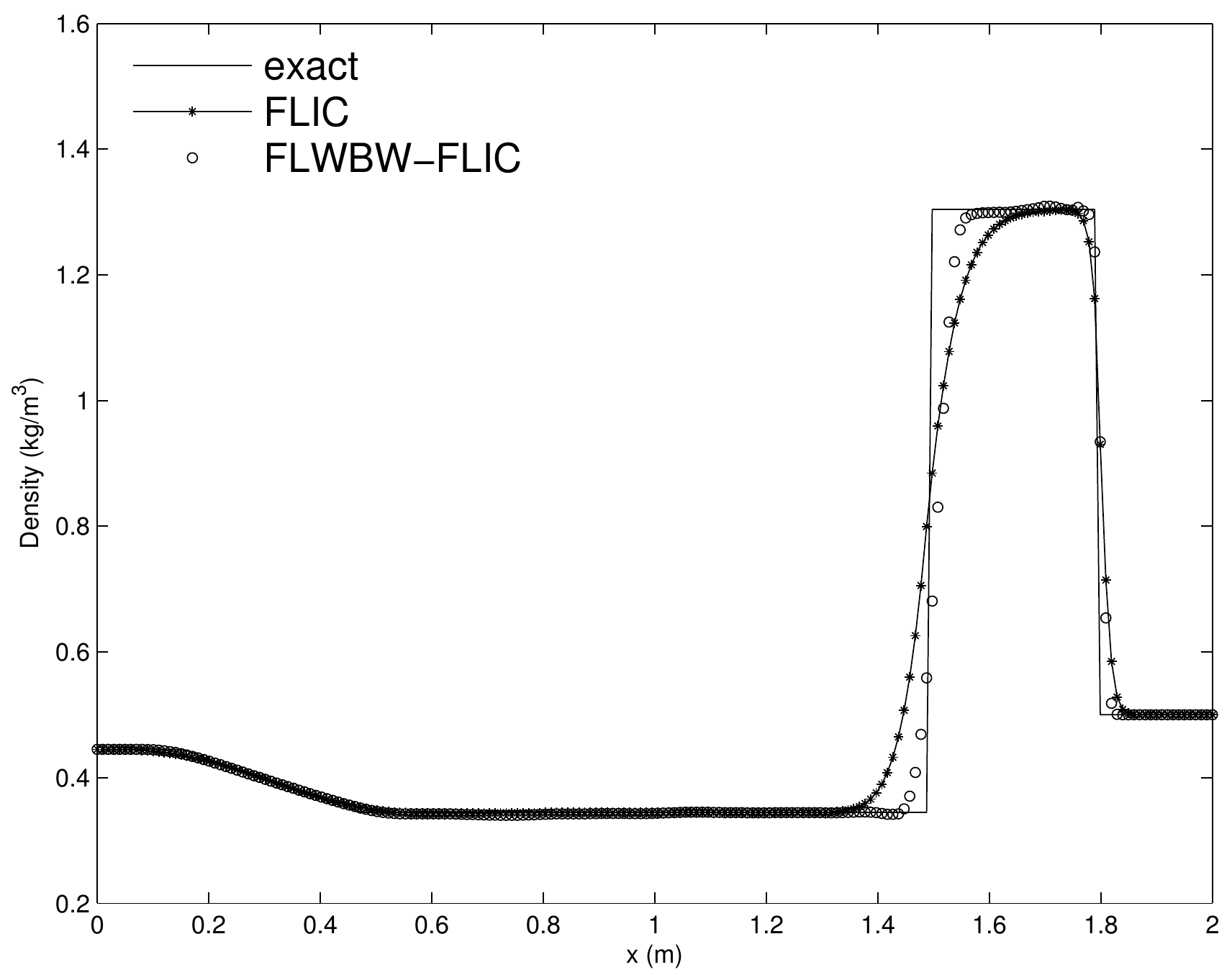} &  \includegraphics[scale=0.36]{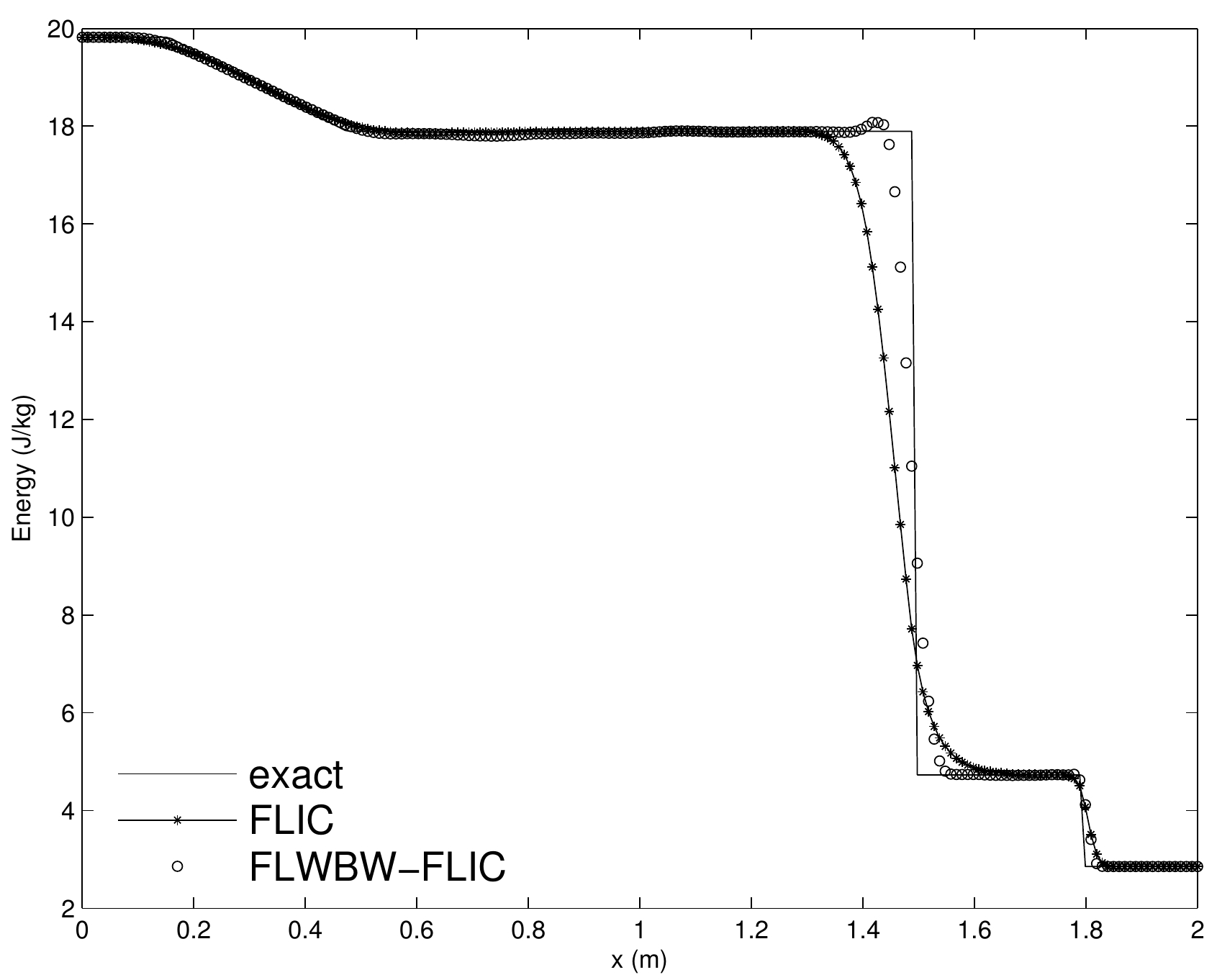}\\
\raisebox{3cm}{\bf (c)}&  \includegraphics[scale=0.36]{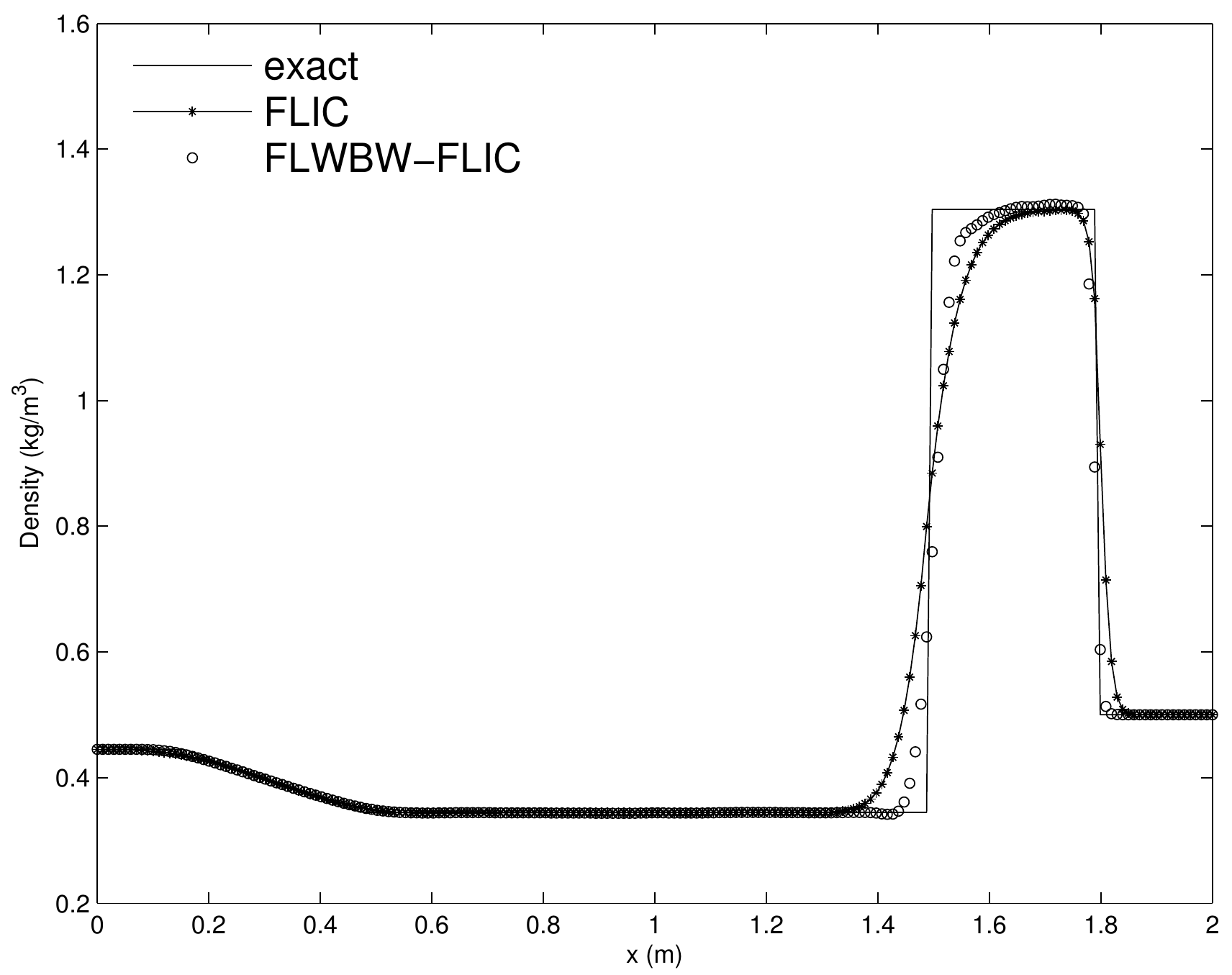}&   \includegraphics[scale=0.36]{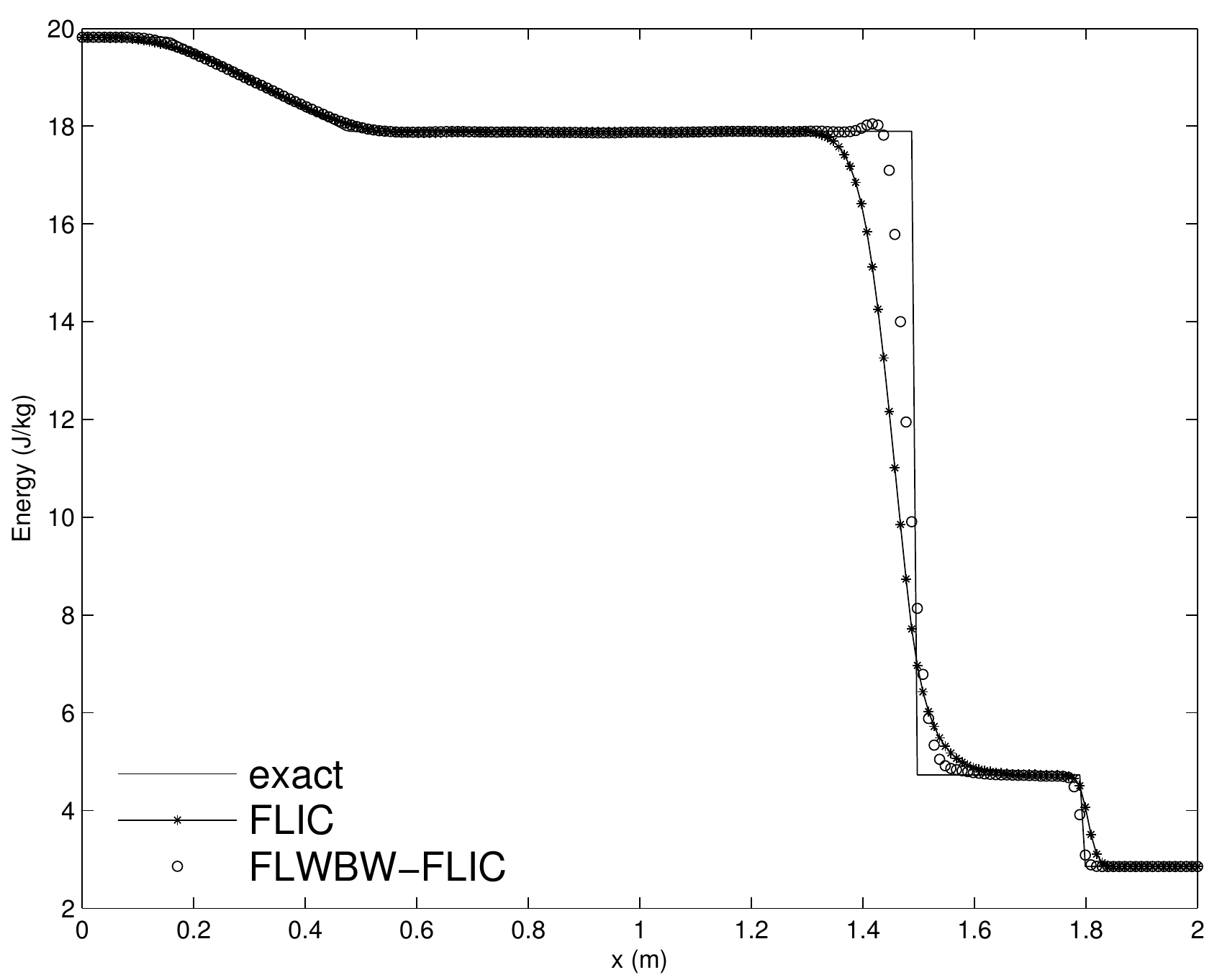}\\
\raisebox{3cm}{\bf (d)}&  \includegraphics[scale=0.36]{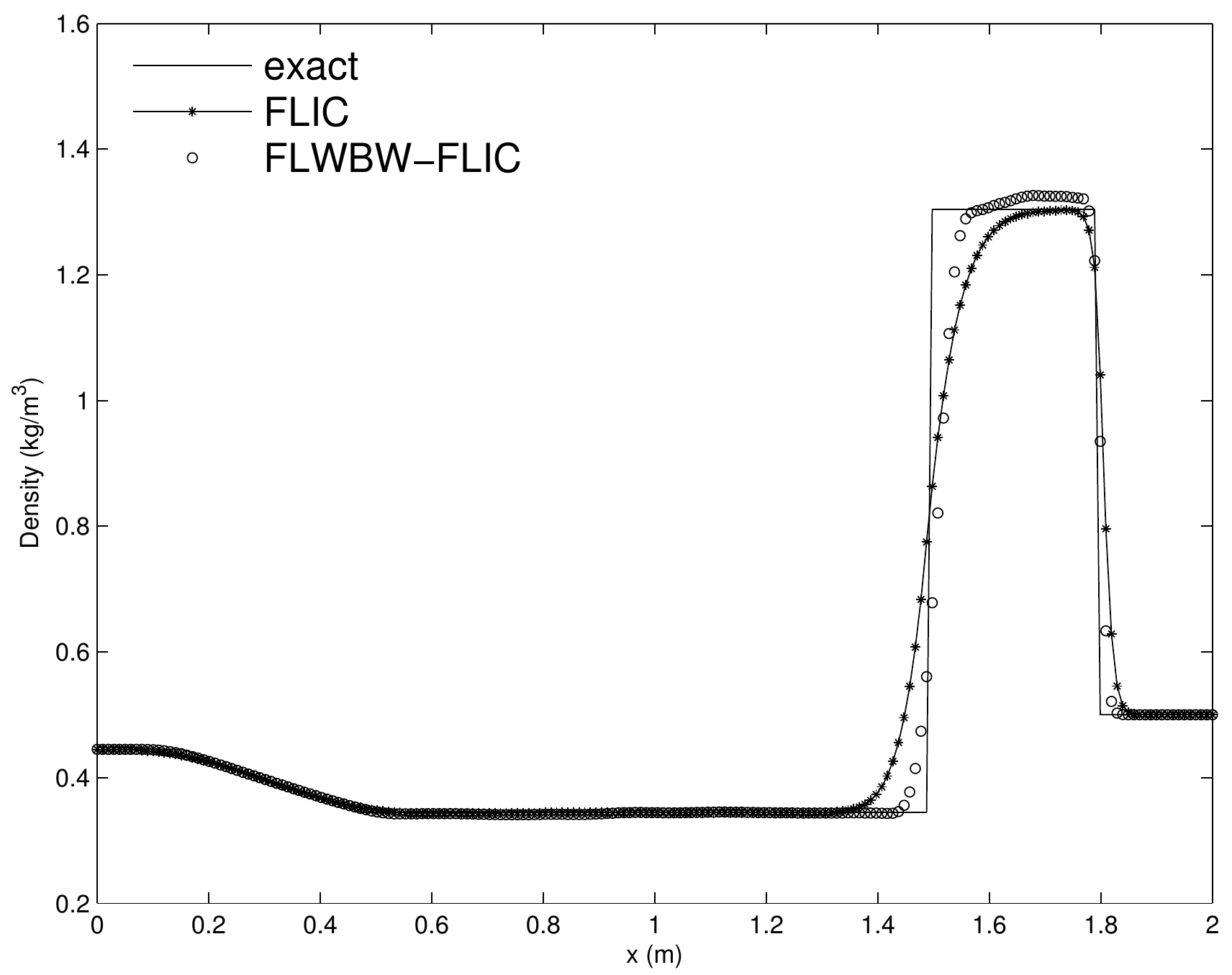}&   \includegraphics[scale=0.36]{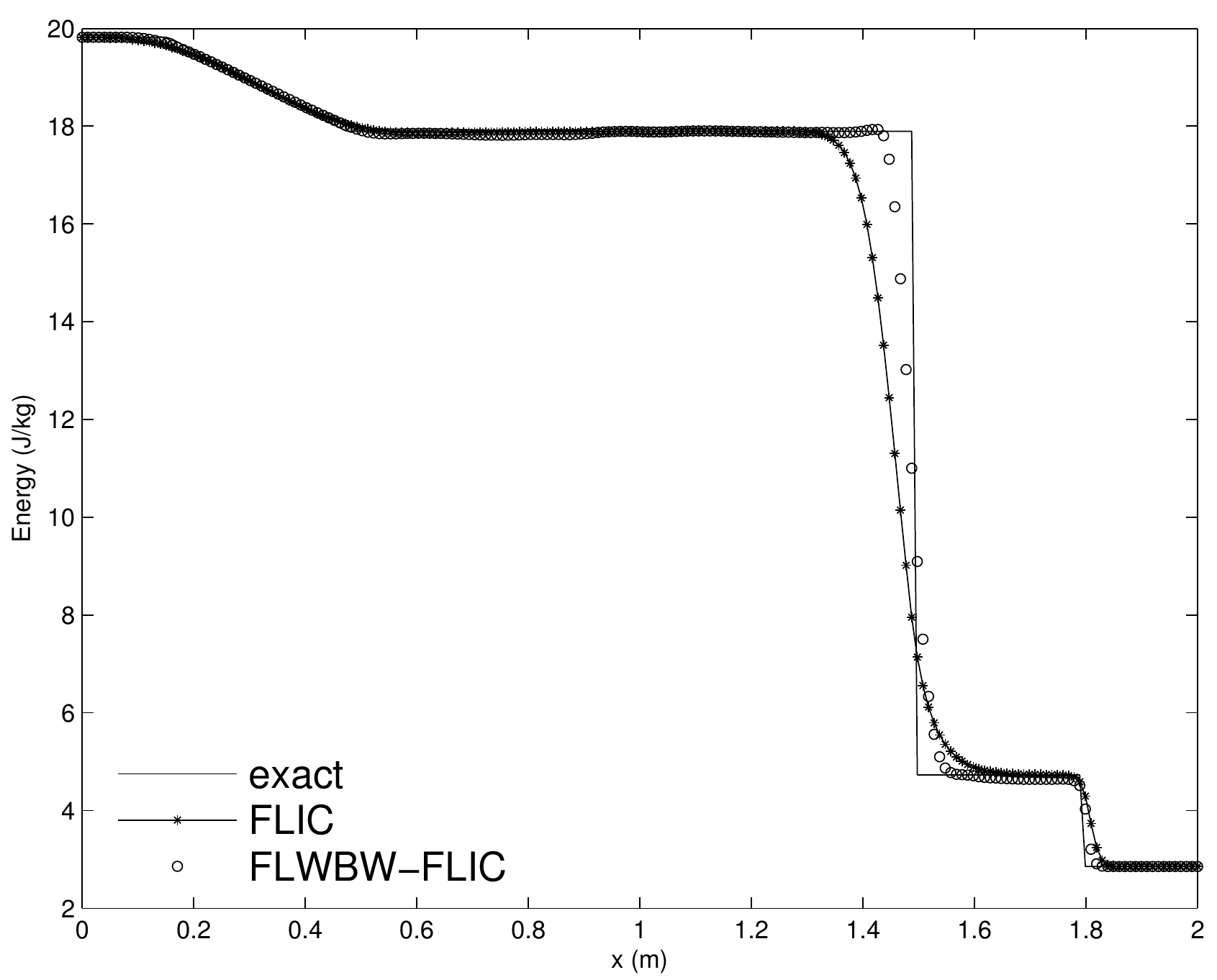}\\
\end{tabular}
\caption{\label{FigLax} Density and energy plots Lax Shock tube case 2, $N=200,
  CFL=0.8$ after $187$ time steps at $T= 0.32$ second with shock
  switch parameters in row (a) $\epsilon=1\times10^{-2}, \delta=0.2$, (b)
  $\epsilon=1\times 10^{-2}, \delta=0.8$, (c) $\epsilon=1\times
  10^{-4}, \delta=0.4$ and (d) $\epsilon=1\times 10^{-8}, \delta=0.8$.}
\end{figure}

\subsubsection{Laney Test \cite{Laney}}
\begin{equation}
(\rho,u,p)=\begin{array}{ll} (1\; kg/m^3 , 0\; m/s, 100,000\; N/m^2)  & x<0,\\
(0.01\; kg/m^3,  0\;m/s , 1,000\; N/m^2) & x\geq 0; \label{LaneyIC}
\end{array}  x\in[-10,15].
\end{equation}
In this test the density and pressure state on the right side of
initial discontinuity is much smaller compared to the left
state. Therefore computationally, even small oscillations can lead to negative
density or pressure which results in to non-physical imaginary speed
of sound $c=\sqrt{\frac{\gamma p}{\rho}}$. This makes it an important
test to check the non-oscillatory nature of any numerical scheme. In
Figure \ref{FigLaney}, the density and pressure plots obtained by
FLWBW-FLIC are given and compared for shock switch parameters
$\epsilon,\delta$.
\begin{figure}[!htb]
\begin{tabular}{cc}
  \includegraphics[scale=0.55]{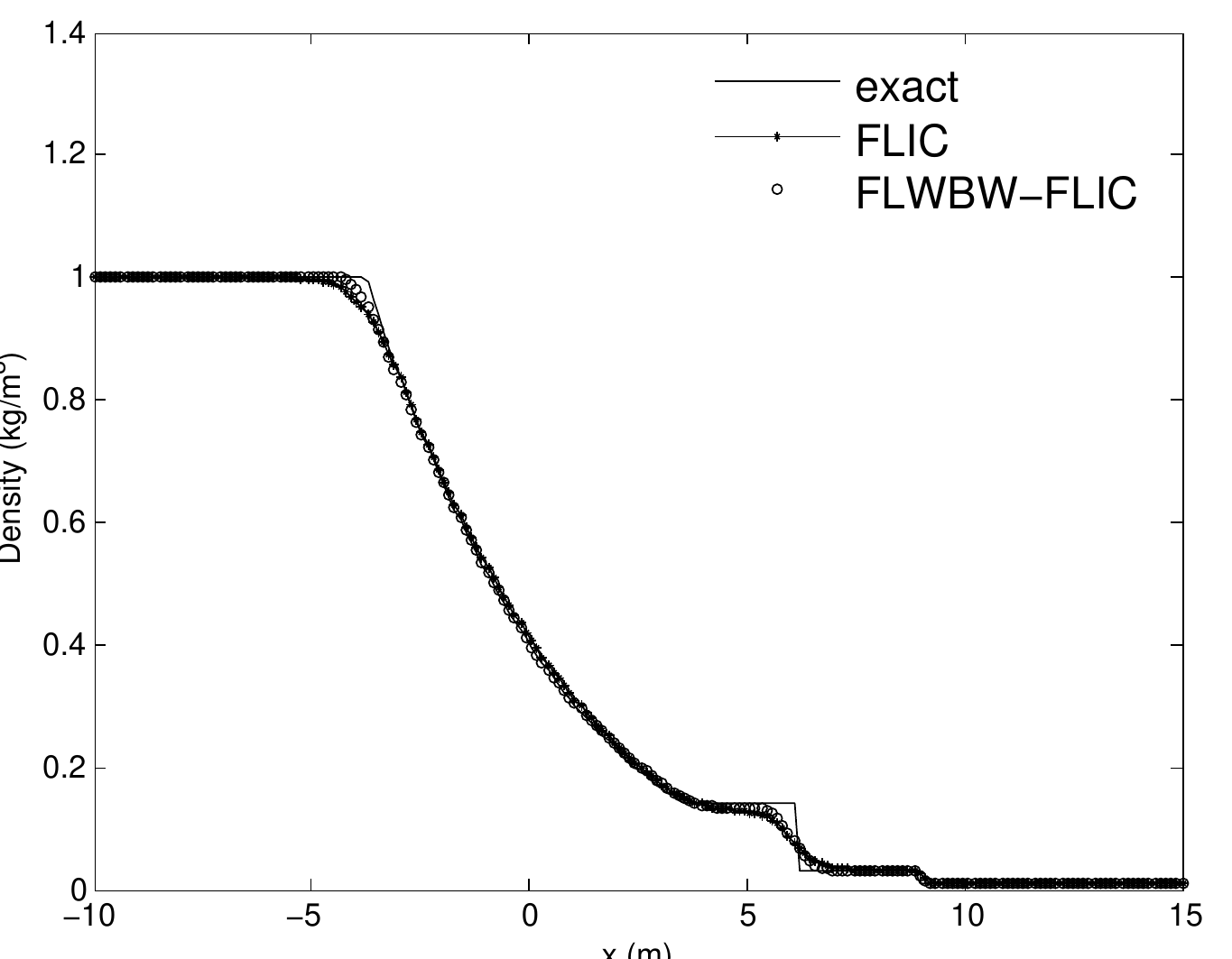}&\includegraphics[scale=0.55]{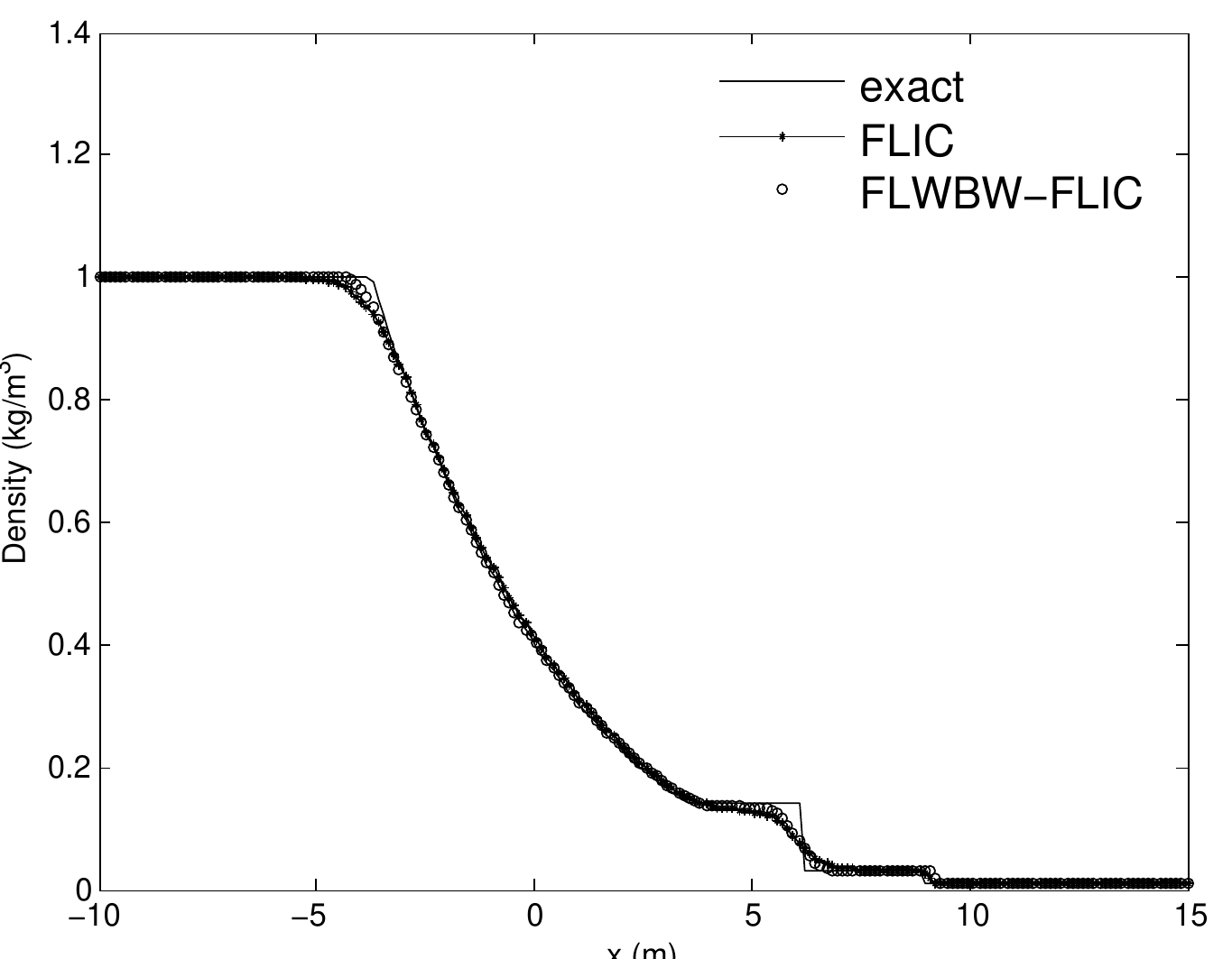}\\
  \includegraphics[scale=0.55]{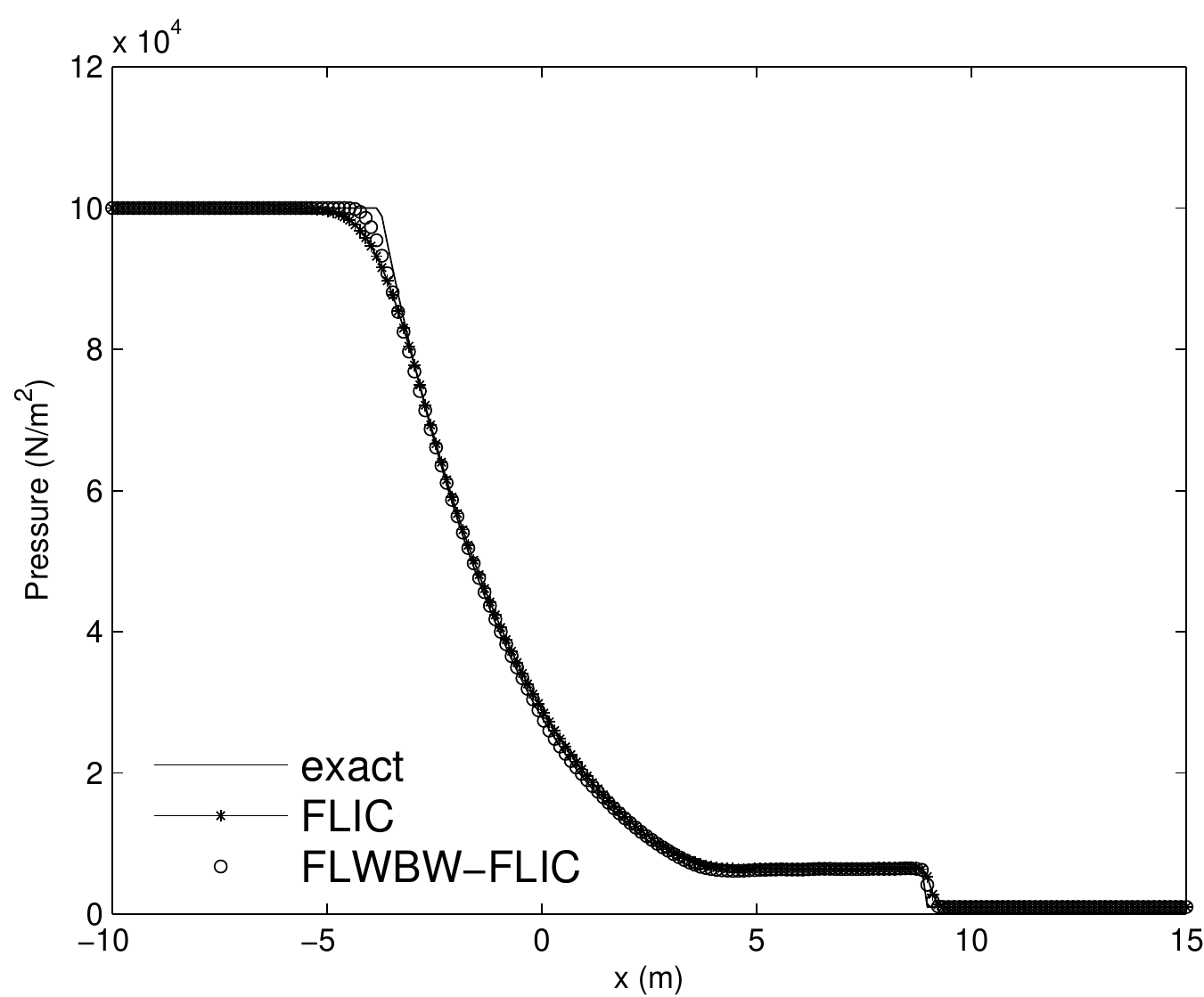}&\includegraphics[scale=0.55]{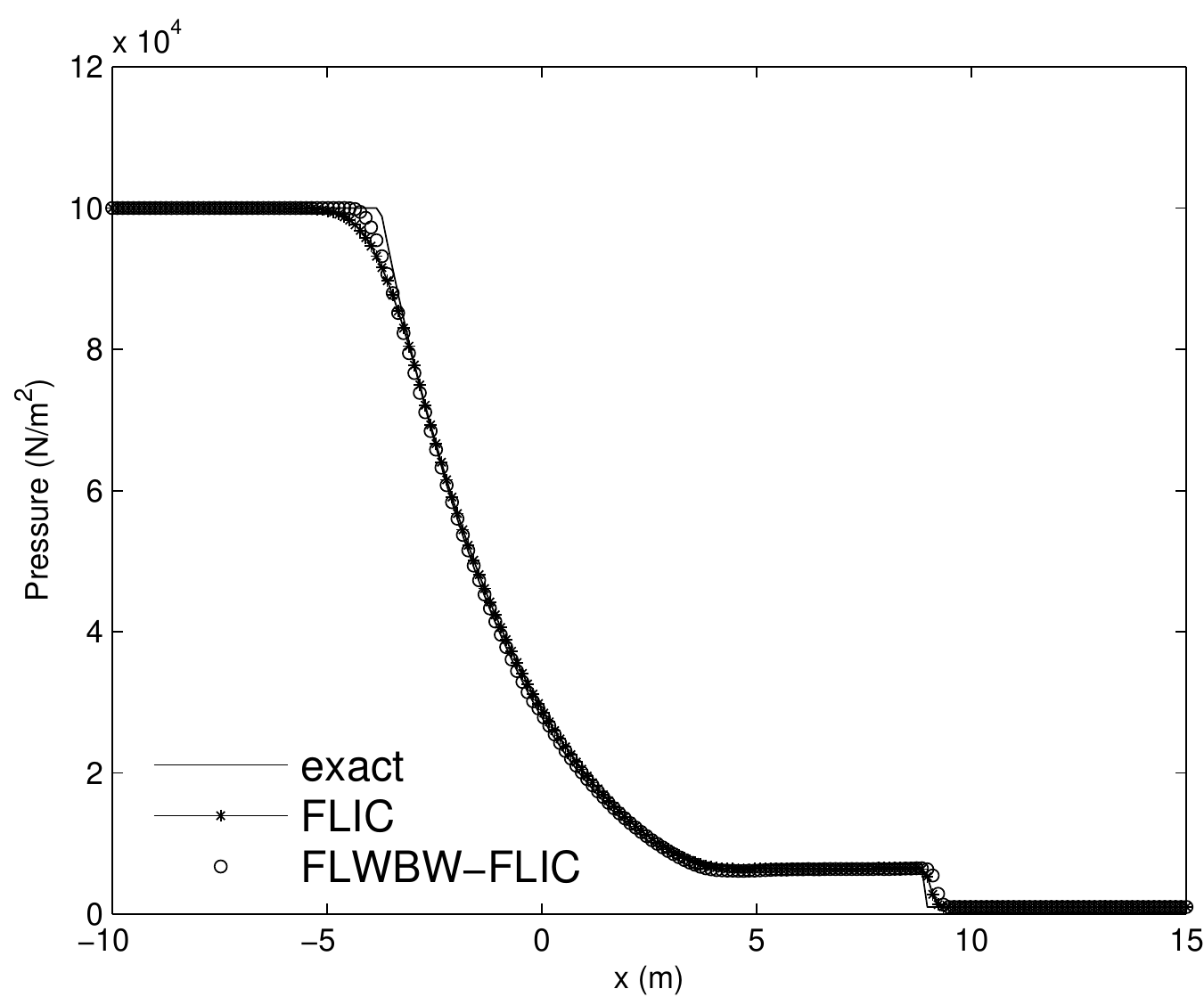}\\
  \includegraphics[scale=0.55]{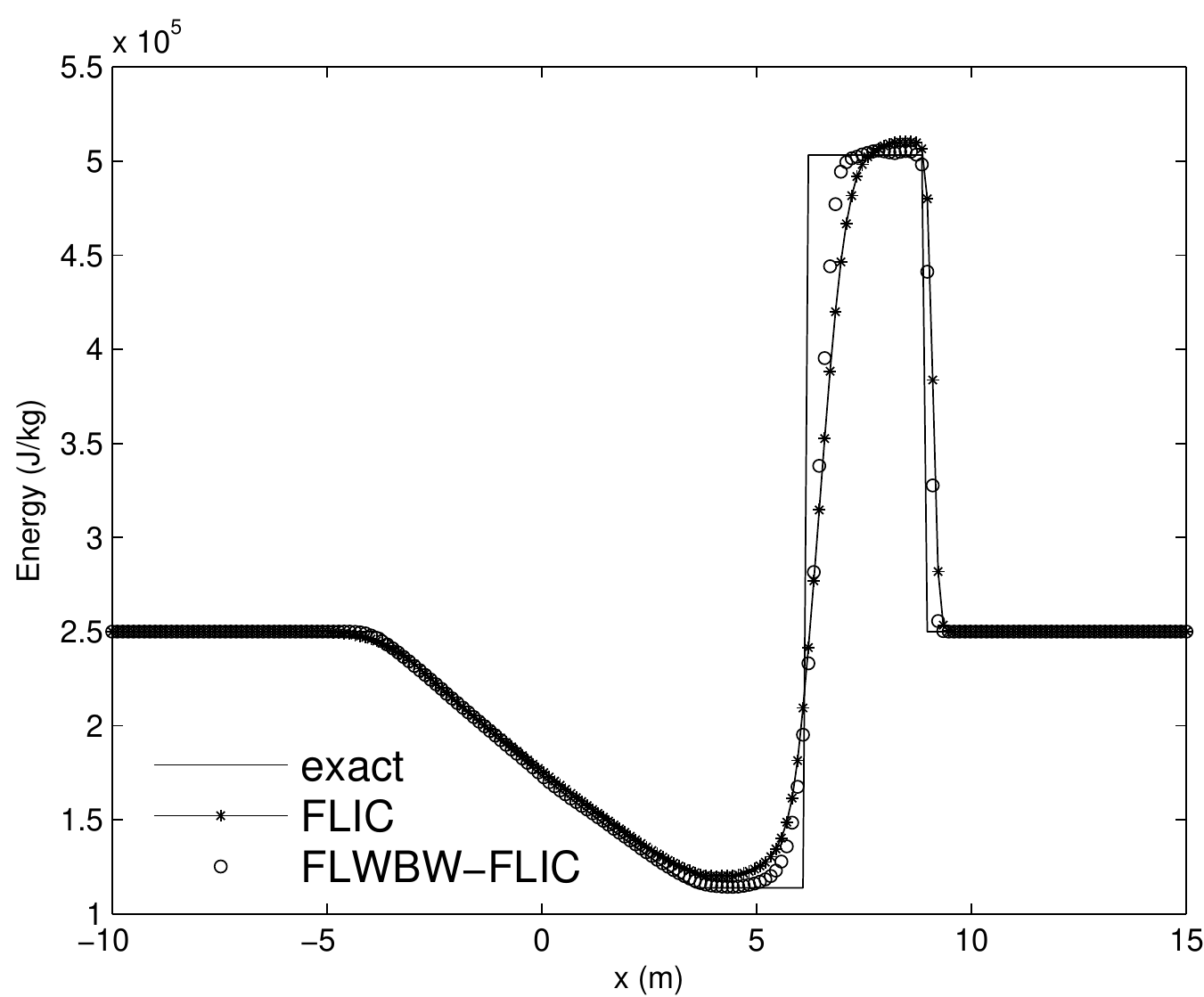}&\includegraphics[scale=0.55]{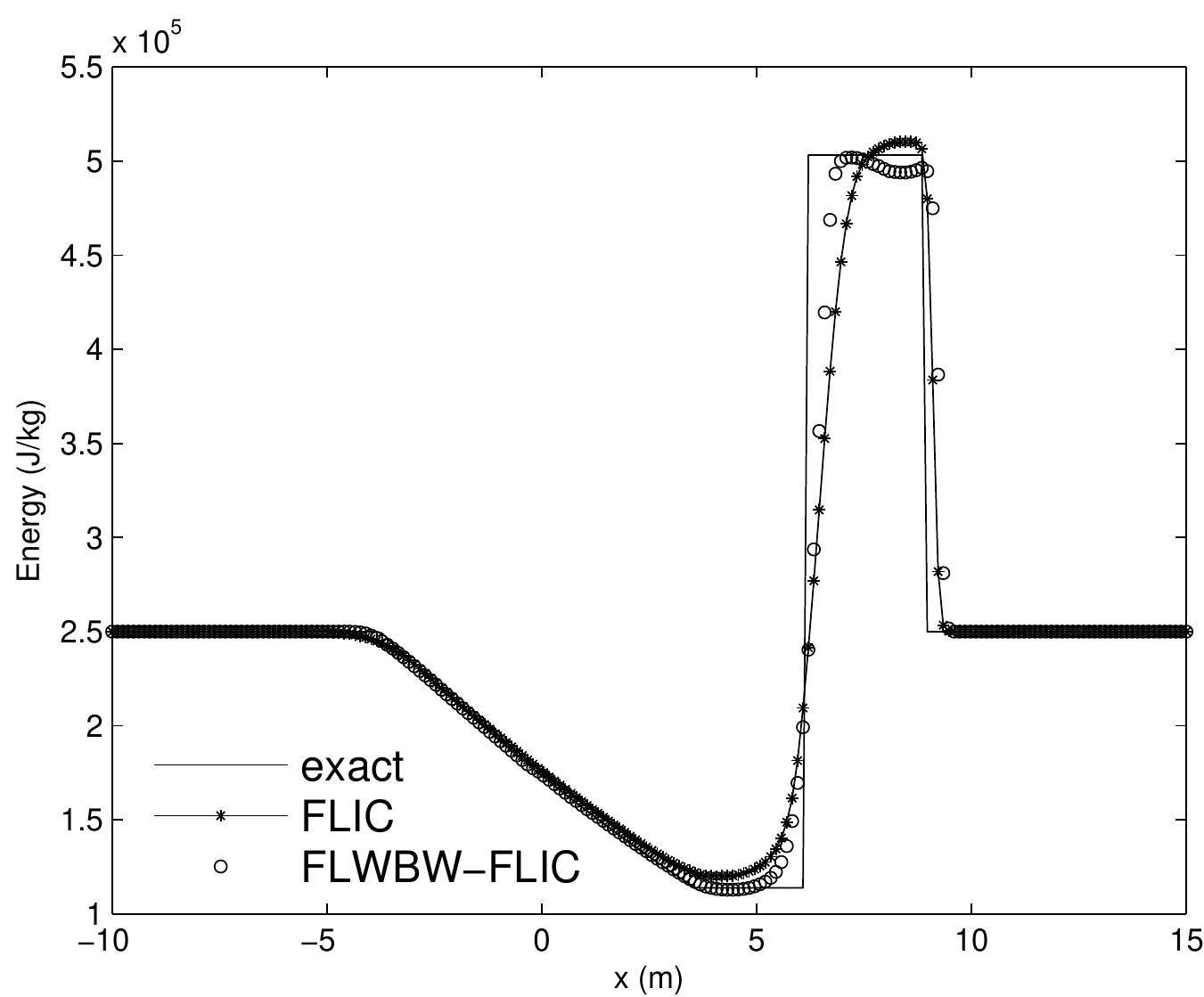}\\
(a) & (b)\\
\end{tabular}
  \caption{\label{FigLaney} Solution Laney test, $N=200, CFL=0.8$
    after $111$ time steps at $T= 0.01$ second with shock switch
    parameters $\epsilon=1e-8$ (a) $\delta=0.6$ (b) $\delta=0.9$ }
\end{figure}

Numerical results for 1D shock tube test problems show that the
corners of rarefaction and contact discontinuities are better resolved
by FLWBW-FLIC compared to centred flux limiter shock capturing TVD
scheme FLIC. The shock is also crisply captured compared to to FLIC
with Minbee limiter. The hybrid scheme presented in section
\ref{algo4sys} is robust and works for different choices of shock
parameters $\epsilon,\; \delta$.
\subsection{2D Euler Equation}
Consider the two-dimensional Euler equations for compressible gas
dynamics defined by the system
\begin{equation}
  U_t+F(U)_x+G(U)_y=0,\label{euler2d}
\end{equation}
where \begin{equation}
  U=\left( \begin{array}{cccc}
\rho \\
\rho u \\
\rho v \\
 e \end{array} \right),
 ~ F=\left( \begin{array}{cccc}
\rho u\\
\rho u^2+p\\
\rho uv\\
u(e+p)\end{array} \right),
~ G=\left( \begin{array}{cccc}
\rho v\\
\rho uv\\
\rho v^2+p\\
v(e+p) \end{array} \right)
\end{equation}
Here $\rho$ is the density, $u$ and $v$ are velocity components in $x$ and $y$
 direction respectively, $p$ is the pressure and $e$ is
 the energy defined by,
\begin{equation}
e=\frac{p}{\gamma-1}+\frac{\rho(u^2+v^2)}{2}.
\end{equation}
The Riemann problem for Euler equation (\ref{euler2d}) can be defined by
considering the constant initial data in each quadrant of unit square
$[0,1]\times[0,1]$ with center at $(0.5,0.5)$. More precisely consider
(\ref{euler2d}) with initial data\\
\begin{equation}
(p,\rho,u,v)(x,y,t=0) = \left\{\begin{array}{cc} 
(p_{1}, \rho_{1}, u_{1}, v_{1}) & \mbox{if}\; x>0.5, y>0.5\\ 
(p_{2}, \rho_{2}, u_{2}, v_{2}) & \mbox{if}\; x<0.5, y>0.5\\ 
(p_{3}, \rho_{3}, u_{3}, v_{3}) & \mbox{if}\; x<0.5, y<0.5\\ 
(p_{4}, \rho_{4}, u_{4}, v_{4}) & \mbox{if}\; x>0.5, y<0.5\\ 
\end{array}\right.
\end{equation}
In the 2D Riemann problems due to complex geometric wave pattern most
high resolution schemes experience problems in yielding oscillations
free crisp resolution to solution profile. Such 2D Riemann problem are
numerically solved by using positive scheme and Riemann solvers free
central schemes in \cite{LaxLiu1998} and \cite{Kurganov2002}
respectively. Recently some of the these Riemann problems are
considered to see the performance of a new finite volume adaptive
artificial viscosity method in \cite{Kurganov2012} and (with slight
changed geometry) a HLL Riemann solver in \cite{Vides2015643}
respectively. In this section numerical results are given for twelve
configurations. The one dimensional scheme presented in section
\ref{algo4sys} is extended to two dimensional Euler equation using the
Strang dimension by dimension splitting technique. In Figure
\ref{2DRP1} to Figure \ref{2DRP14} the contour plot of density are
given and compared with FORCE scheme for different test cases. In all
the figures contour plot by FORCE is given in column (a), and by
FLWBW-FORCE in column (b) with shock parameters $\epsilon =1\times
10^{-8}, \delta=0.6$. It is observed that small oscillations can occur
with if a flux limited TVD scheme is used in hybrid scheme in section
\ref{algo4sys} (Results are not shown here). Note that it is in
agreement with the comments in \cite{Kurganov2002}, that a over
compressive Minmod type limiter can lead to spurious oscillations for
the schemes proposed therein.
\begin{enumerate}
  \item[] Configuration 1
  \begin{equation}
    \begin{array}{llll}
        p_1=1    &p_2=0.4       &p_3=0.0439    &p4=0.15\\
        \rho_1=1 &\rho_2=0.5197 &\rho_3=0.1072 &\rho_4=0.2579\\
        u_1=0    &u_2=-0.7259   &u_3=-0.7259   &u_4=0.0\\
        v_1=0    &v_2=0         &v_3=-1.4045   &v_4=-1.4045\\
    \end{array}
  \end{equation}
\begin{figure}[!htb]
\begin{tabular}{cc}
\includegraphics[scale=0.5]{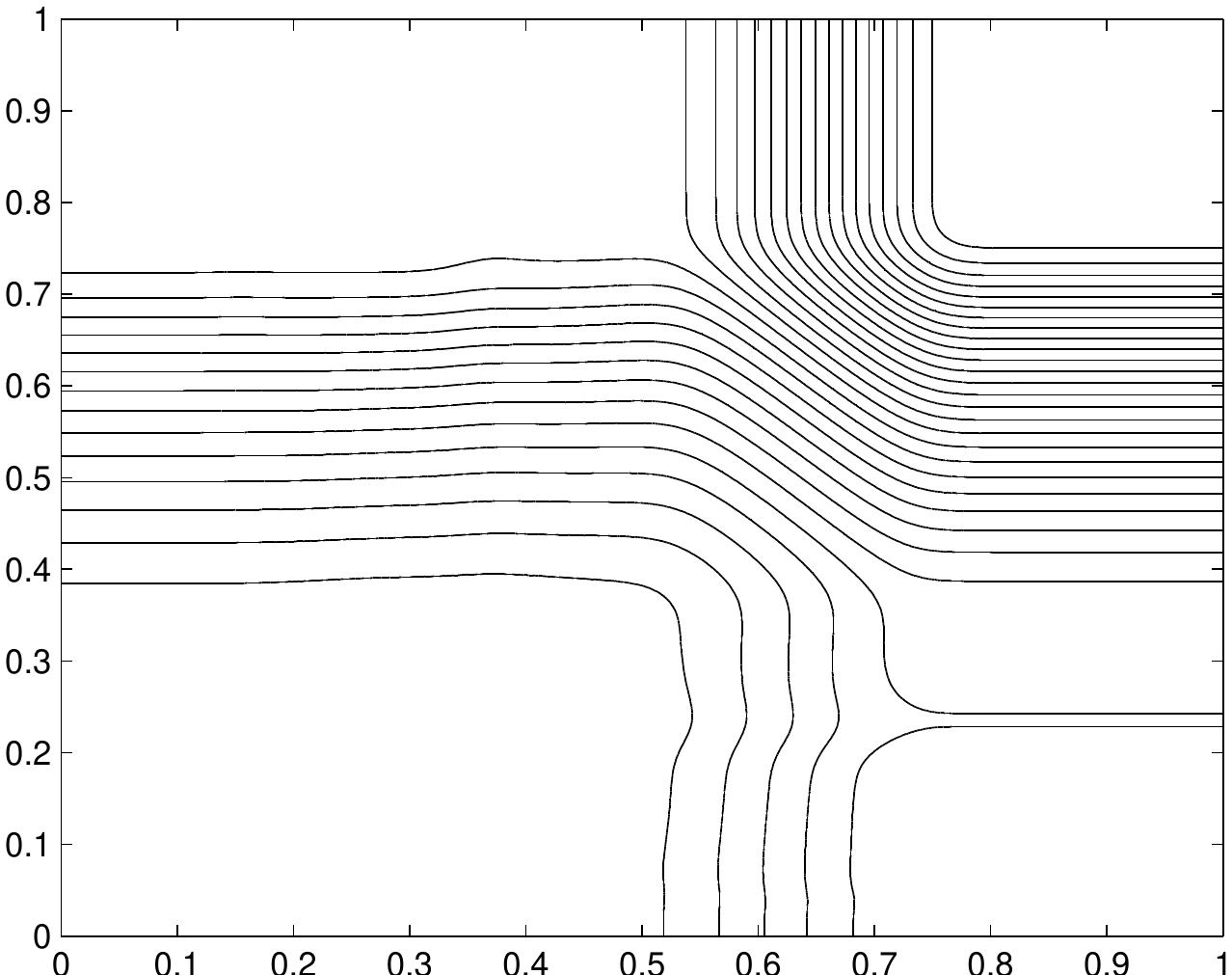} & \hspace{1cm} \includegraphics[scale=0.5]{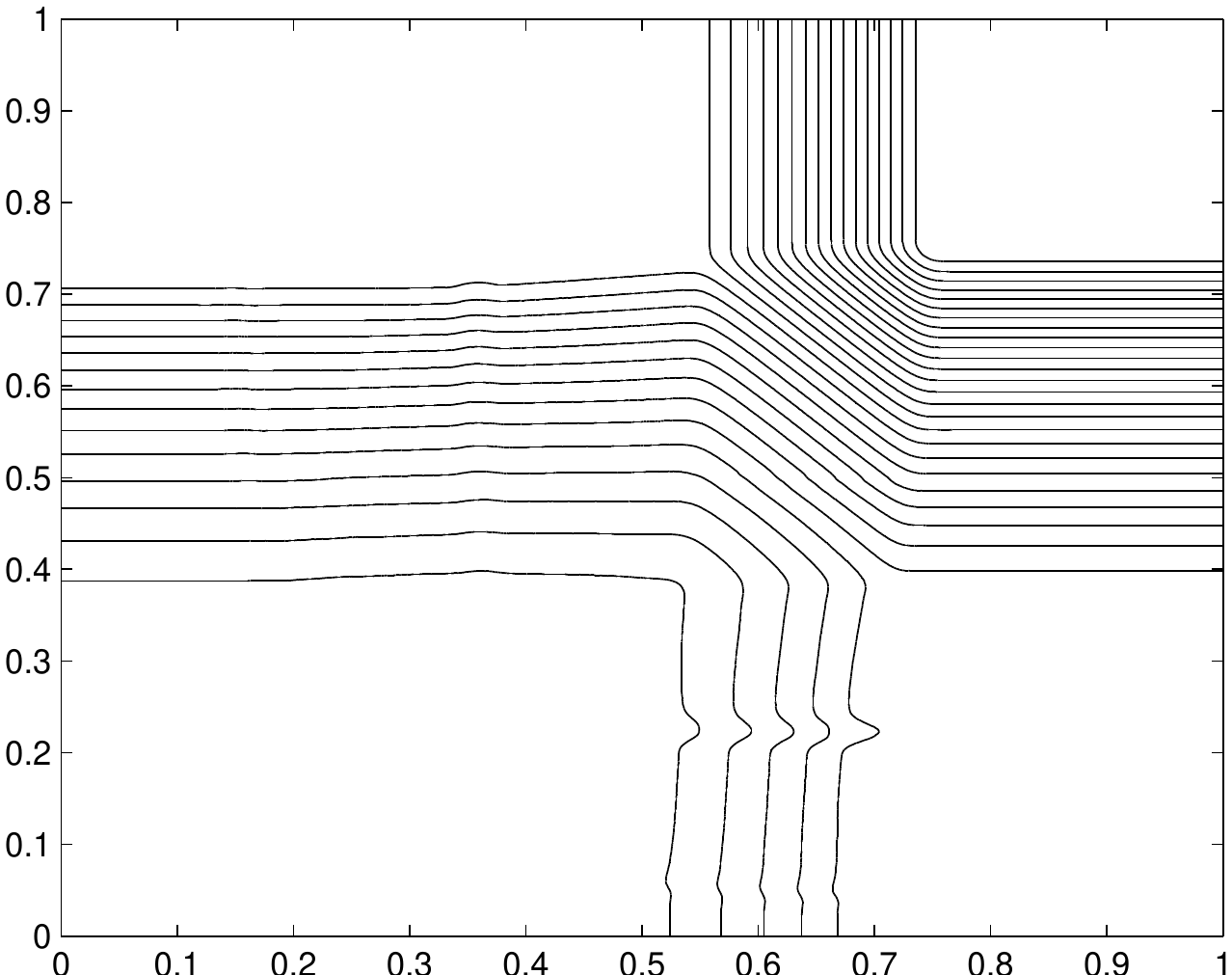}\\
(a)& (b)\\
\end{tabular} 
\caption{\label{2DRP1} Configuration 1: Density contour plot (30 lines)
  at $T=0.2$, sharp resolution for rarefaction and ripple in lower
  rarefaction are sharply resolved in (b) compared to FORCE (a).}
\end{figure}
 \item[] Configuration 2
 \begin{equation}
   \begin{array}{llll}
        p_1=1    &p_2=0.4       &p_3=1       &p_4=0.4     \\
        \rho_1=1 &\rho_2=0.5197 &\rho_3=1    &\rho_4=0.5197 \\
        u_1=0    &u_2=-0.7259   &u_3=-0.7259 &u_4=0       \\
        v_1=0    &v_2=0         &v_3=-0.7259 &v_4=-0.7259 
   \end{array}
 \end{equation} 
\begin{figure}[!htb]
\begin{tabular}{cc}
\includegraphics[scale=0.5]{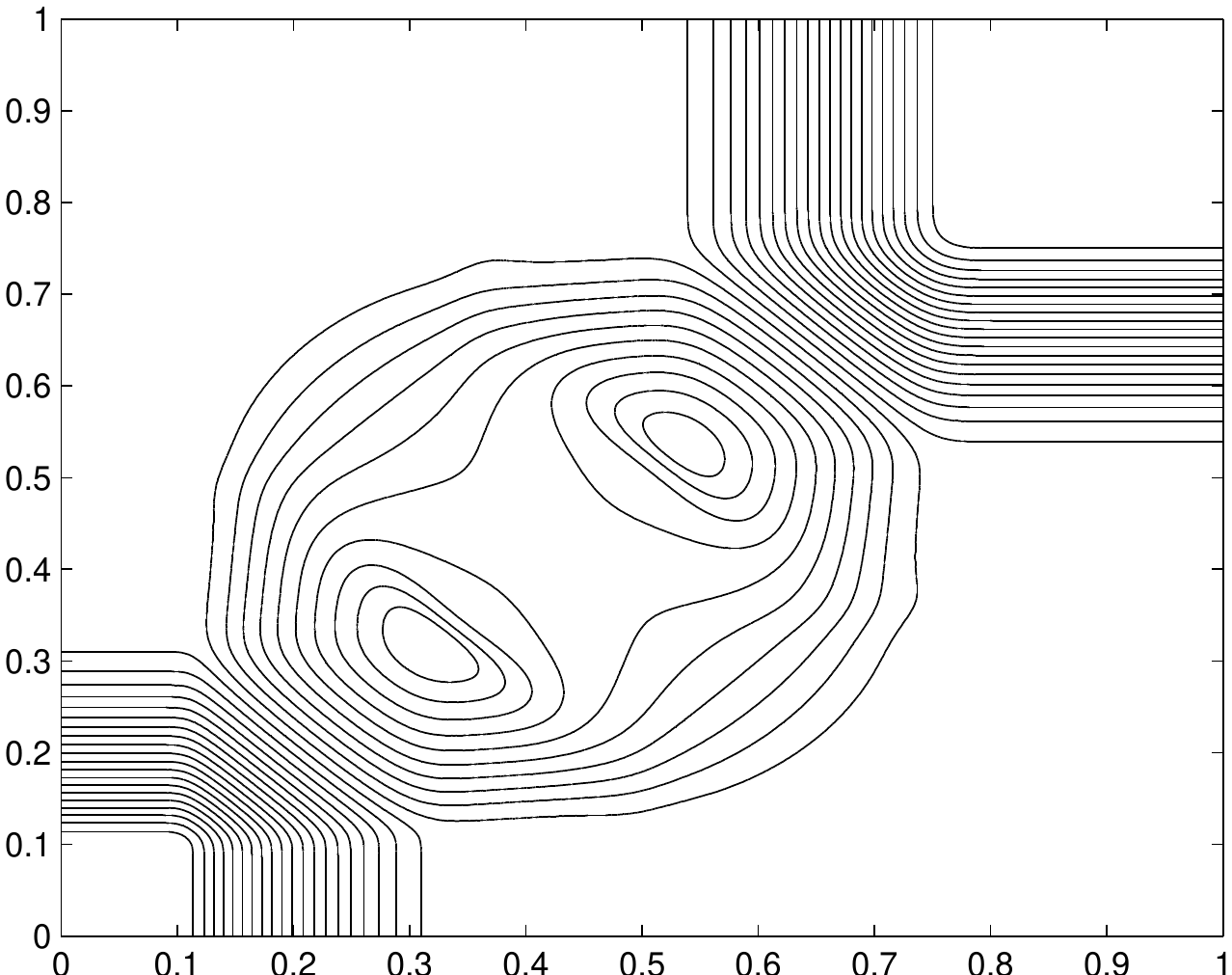} &\hspace{1cm} \includegraphics[scale=0.5]{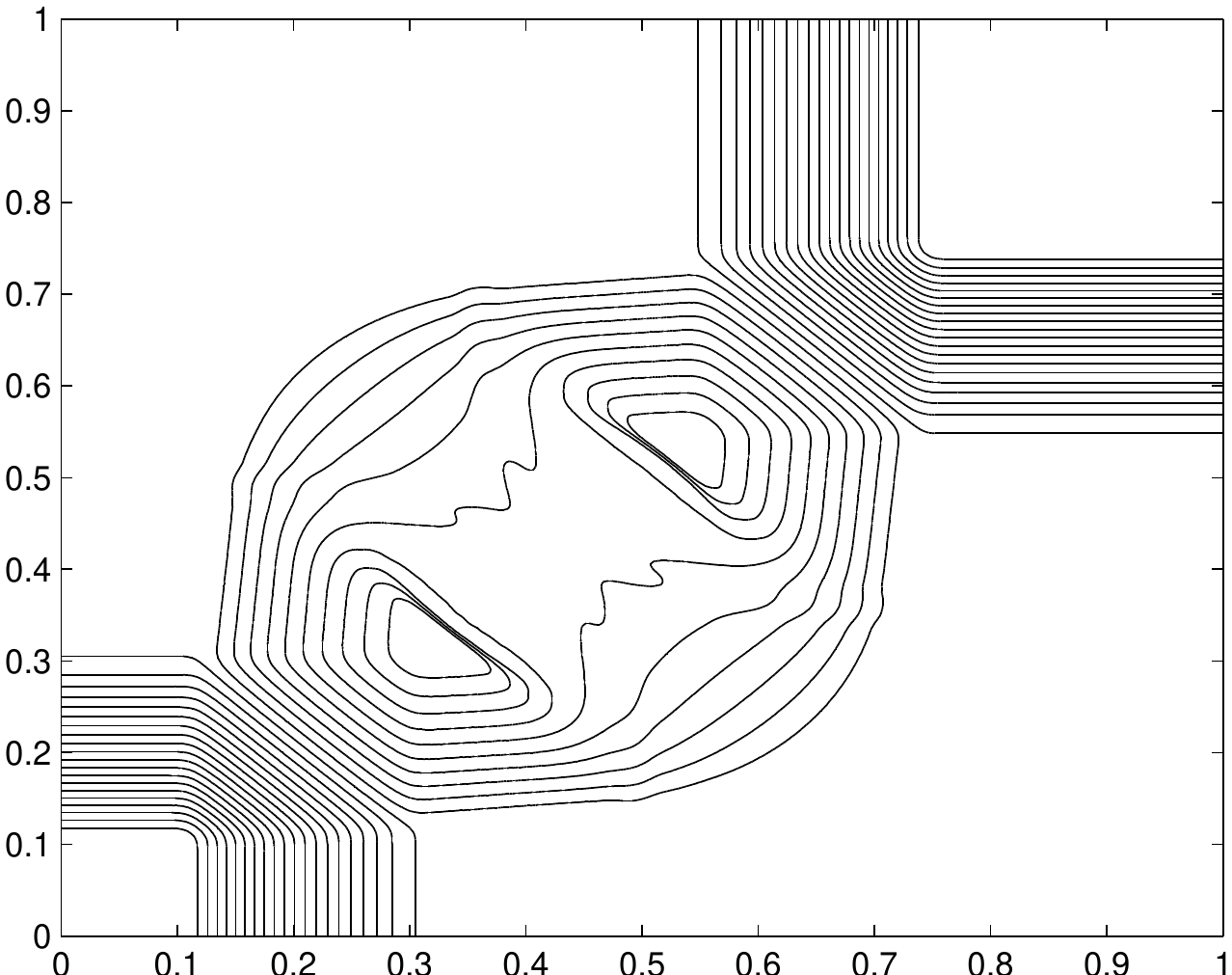}\\
(a)& (b)\\
\end{tabular}
\caption{\label{2DRP2} Configuration 2: Density contour plot (30
  lines) at $T=0.25$ non-oscillatory sharp resolution for rarefaction
  waves and corners in (b) compared to FORCE in (a)}
\end{figure}
 \item[] Configuration 3
 \begin{equation}
   \begin{array}{llll}
     p_1=1.5    &p_2=0.3      &p_3=0.029    &p_4=0.3\\
        \rho_1=1.5 &\rho2=0.5323 &\rho_3=0.138 &\rho4=0.5323\\
        u_1=0      &u_2=1.206    &u_3=1.206    &u_4=0\\
        v_1=0      &v_2=0        &v_3=1.206    &v_4=1.206
   \end{array}
 \end{equation}

\begin{figure}[!htb]
\begin{tabular}{cc}
\includegraphics[scale=0.5]{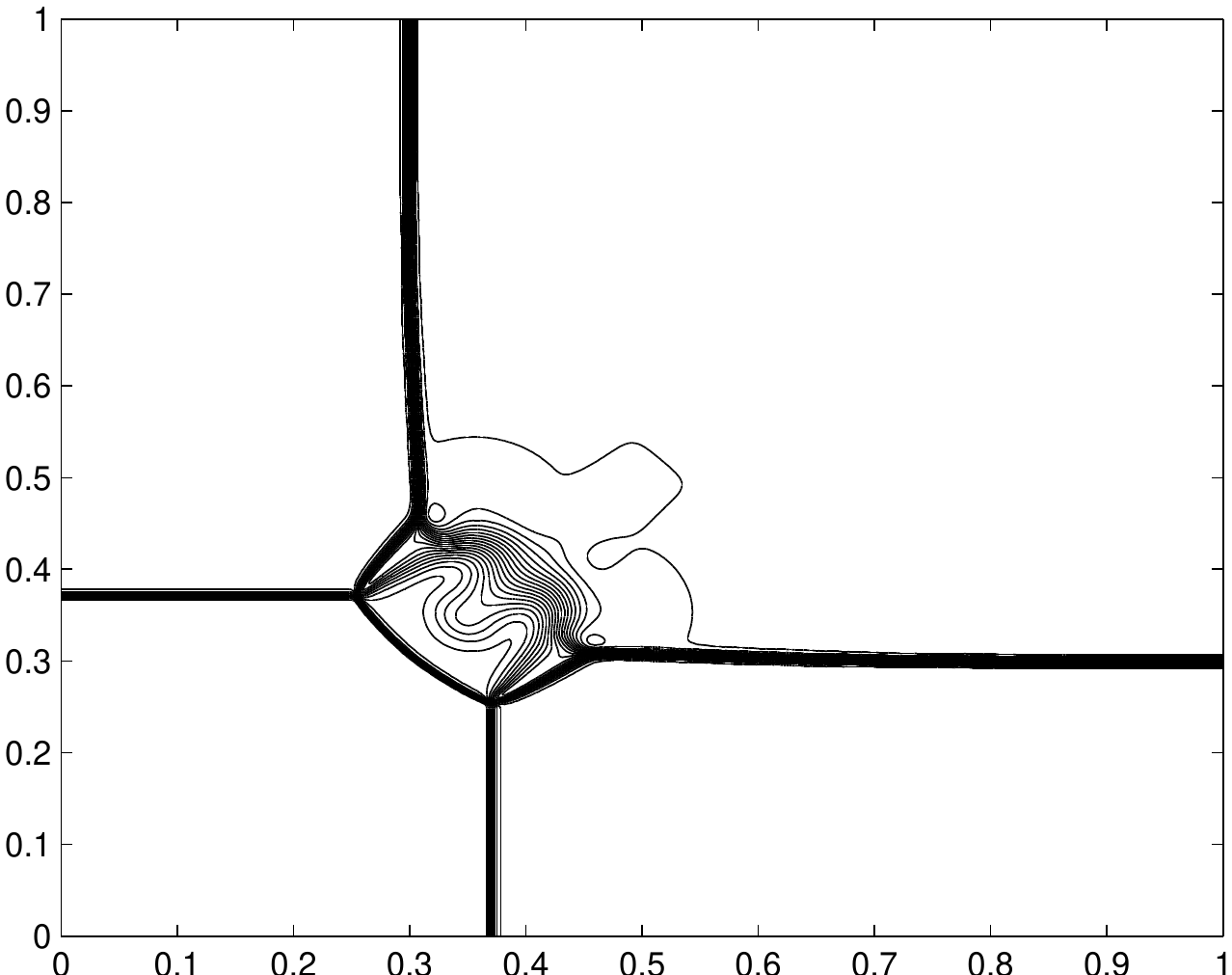} & \hspace{1cm}\includegraphics[scale=0.5]{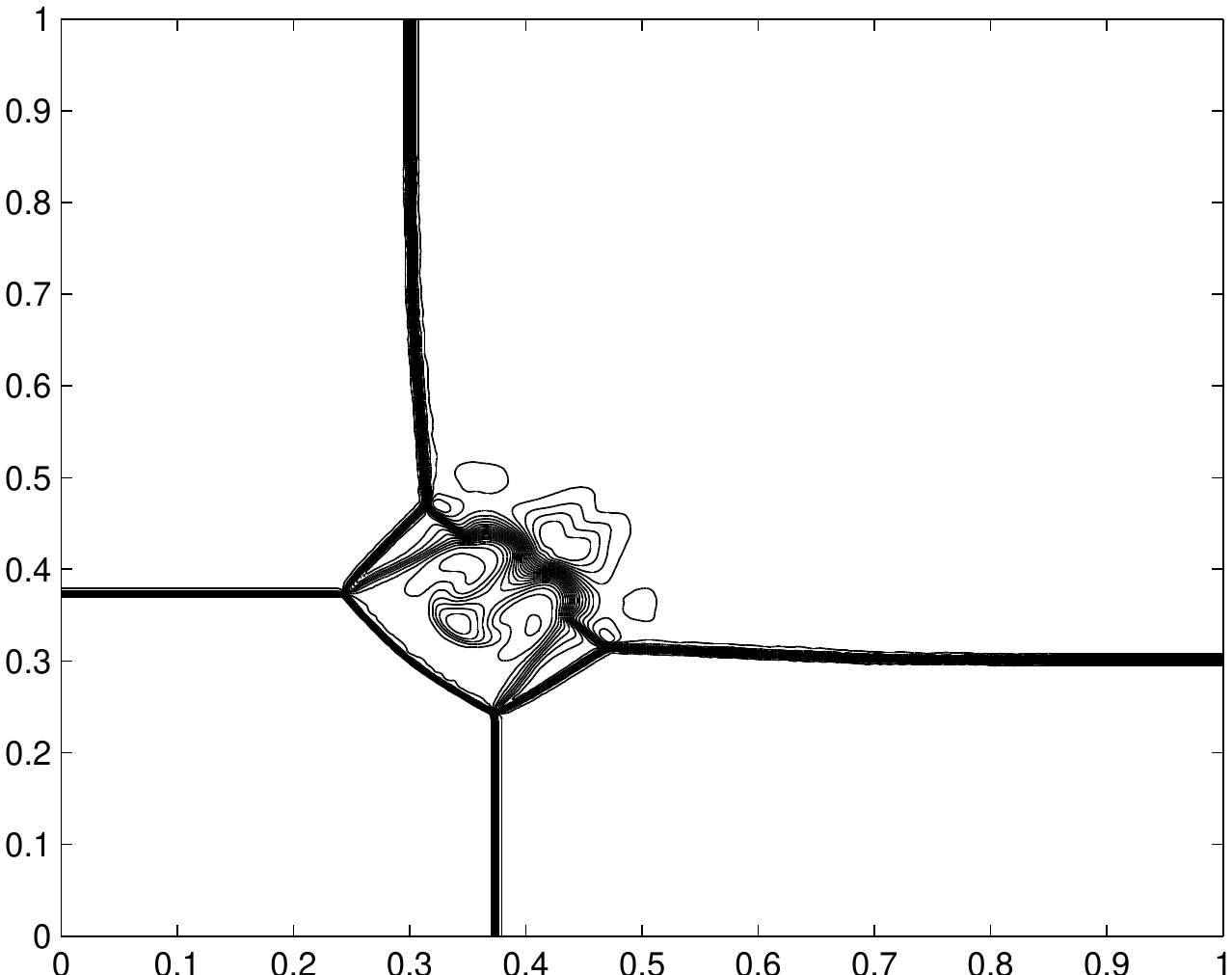}\\
(a)& (b)\\
\end{tabular}
\caption{\label{2DRP3} Configuration 3: Density contour plot (32 lines), All four shocks are sharply resolved with out induced oscillations in FLWBW-FORCE in (b) compared to FORCE in (a).}
\end{figure}
 \item[] Configuration 4
 \begin{equation}
   \begin{array}{llll}
     
        p_1=1.1    &p_2=0.35      &p_3=1.1    &p_4=0.35\\
        \rho_1=1.1 &\rho_2=0.5065 &\rho_3=1.1 & \rho_4=0.5065\\
        u_1=0      &u_2=0.8939    &u_3=0.8939 &u_4=0\\
        v_1=0      &v_2=0         &v_3=0.8939 &v_4=0.8939
   \end{array}
 \end{equation}
\begin{figure}[!htb]
\begin{tabular}{cc}
\includegraphics[scale=0.5]{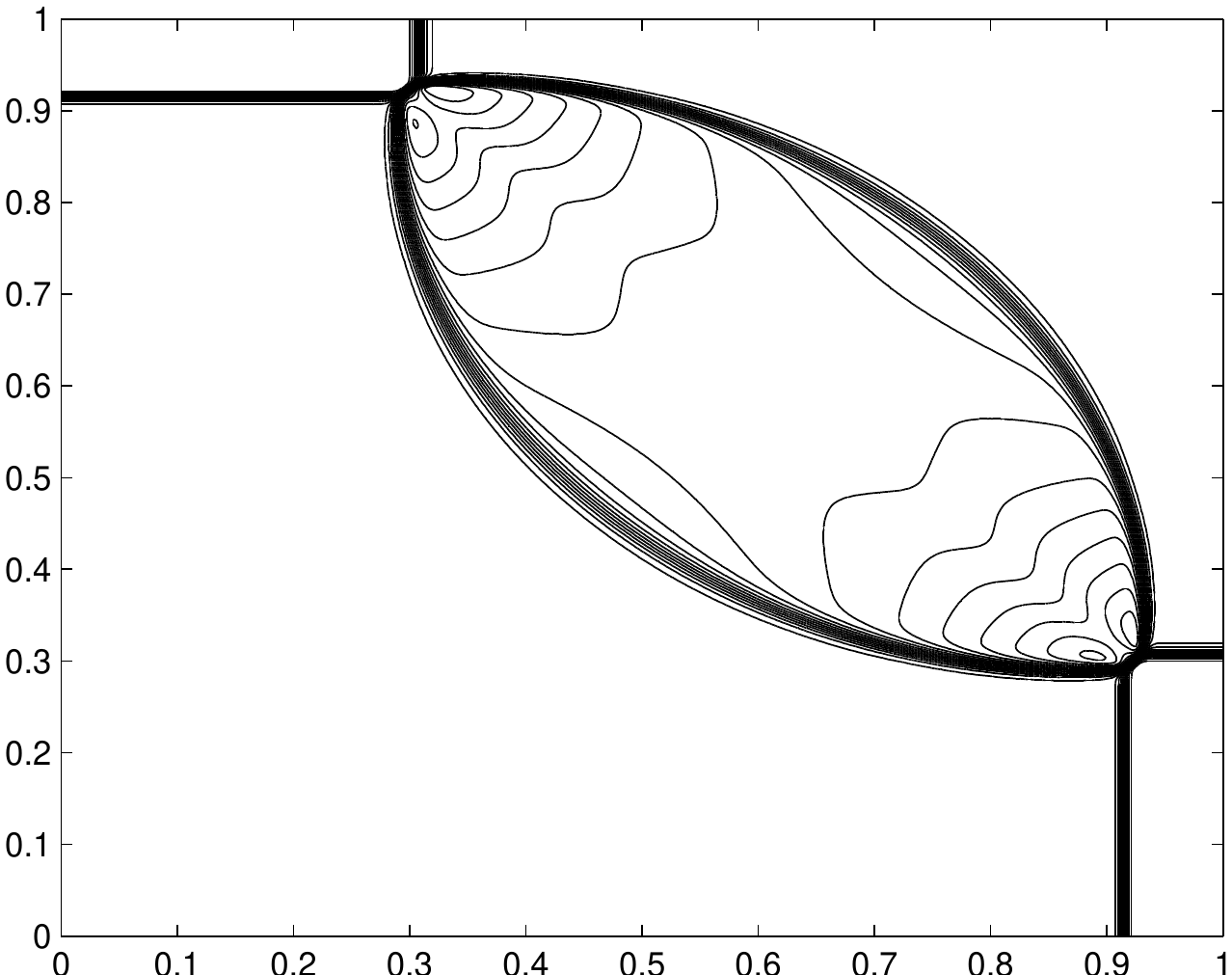} & \hspace{1cm} \includegraphics[scale=0.5]{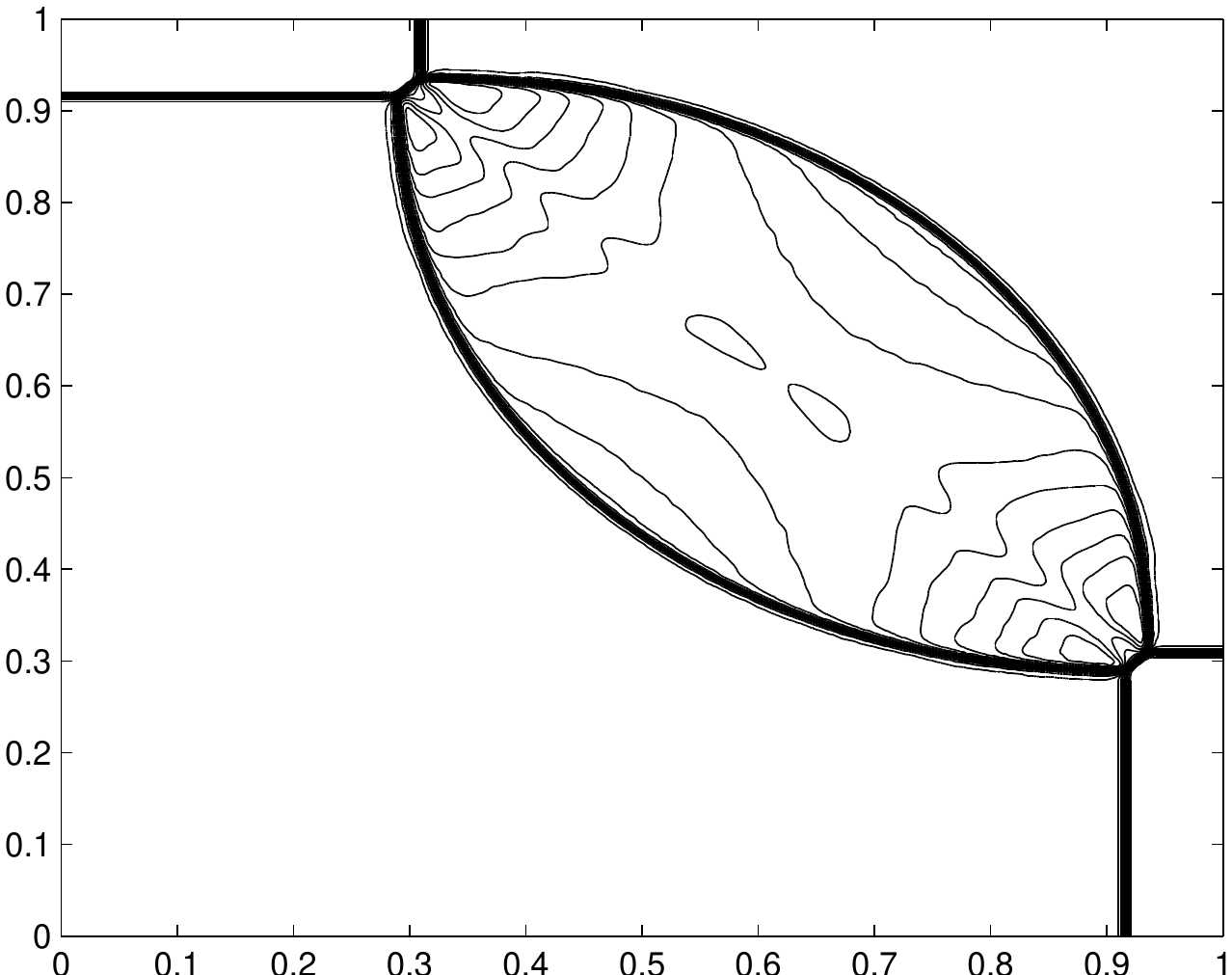}\\
(a)& (b)\\
\end{tabular}
\caption{\label{2DRP4} Configuration 4: Density contour plot (30
  lines) non-oscillatory, FLWBW-FORCE yield sharp resolution for shocks in (b)  compared to FORCE in (a)}
\end{figure}
\item[] Configuration 5
 \begin{equation}
   \begin{array}{llll}
     p_1=1    &p_2=1      &p_3=1    &p_4=1\\
        \rho_1=1 &\rho_2=2 &\rho_3=1 & \rho_4=3\\
        u_1=-0.75      &u_2=-0.75    &u_3=0.75 &u_4=0.75\\
        v_1=-0.5      &v_2=0.5         &v_3=0.5 &v_4=-0.5
   \end{array}
 \end{equation}
\begin{figure}[!htb]
\begin{tabular}{cc}
\includegraphics[scale=0.5]{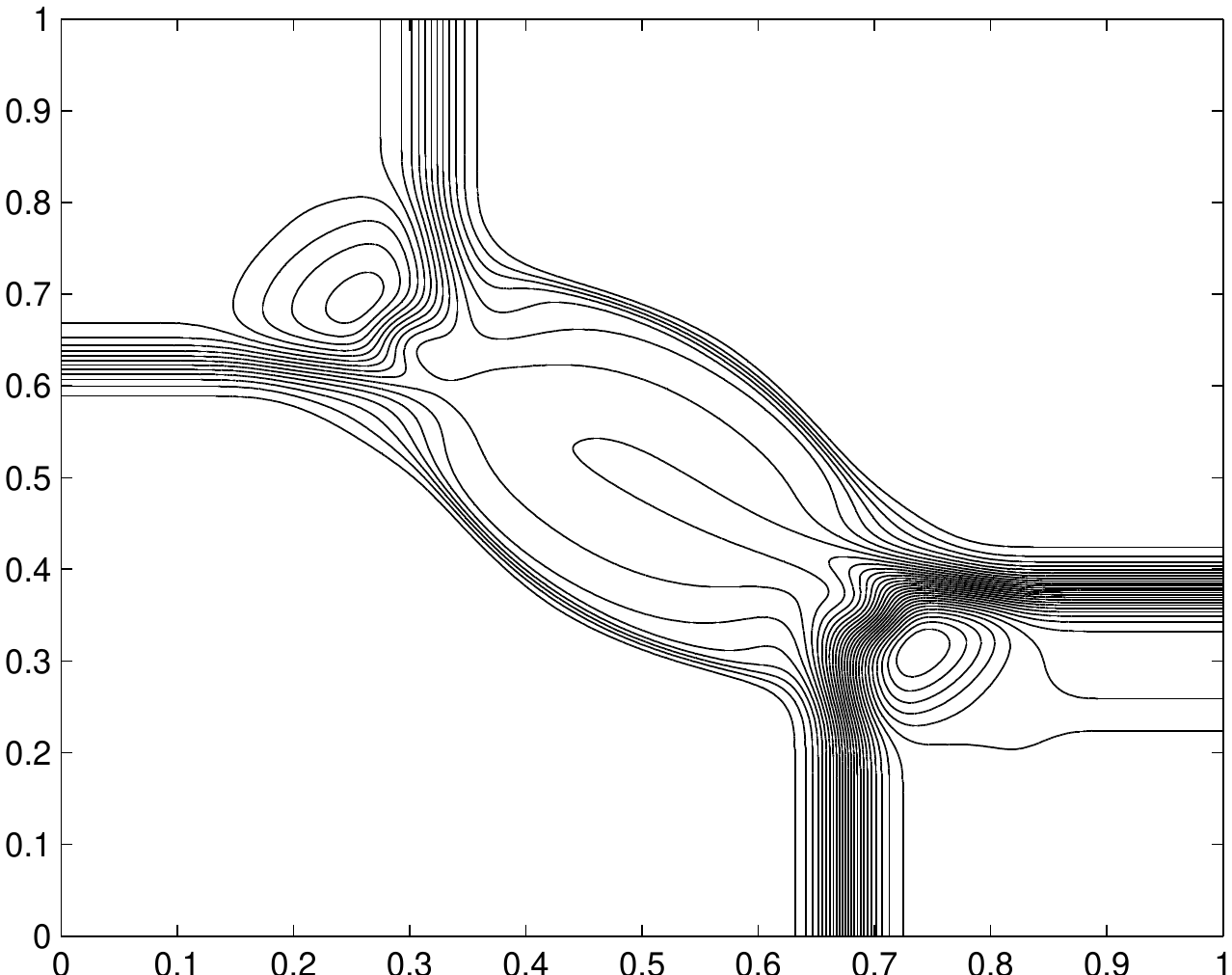} & \hspace{1cm} \includegraphics[scale=0.5]{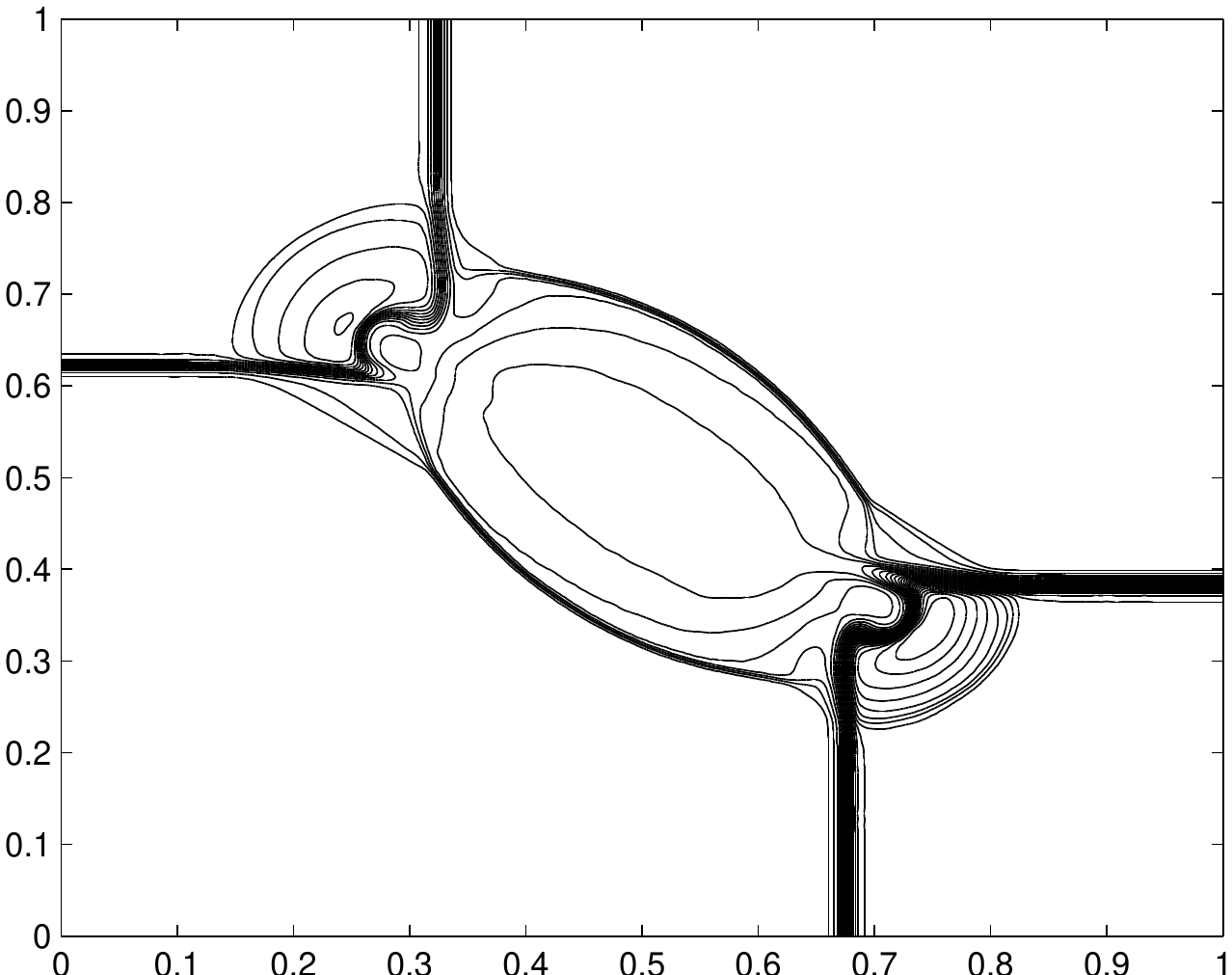}\\
(a)& (b)\\
\end{tabular}
\caption{\label{2DRP5} Configuration 5: Density plot (30 Lines) all four contacts are poorly resolved
  by FORCE (a), though FLWBW-FORCE yield sharp resolution to contacts with out oscillations in (b).}
\end{figure}
\item[] Configuration 6
 \begin{equation}
   \begin{array}{llll}
     p_1=1         &p_2=1           &p_3=1           &p_4=1\\
        \rho_1=1       &\rho_2=2         &\rho_3=1         &\rho_4=3\\
        u_1=0.75      &u_2=0.75       &u_3=-0.75       &u_4=-0.75\\
        v_1=-0.5      &v_2=0.5         &v_3=0.5         &v_4=-0.5
   \end{array}
 \end{equation}
\begin{figure}[!htb]
\begin{tabular}{cc}
\includegraphics[scale=0.5]{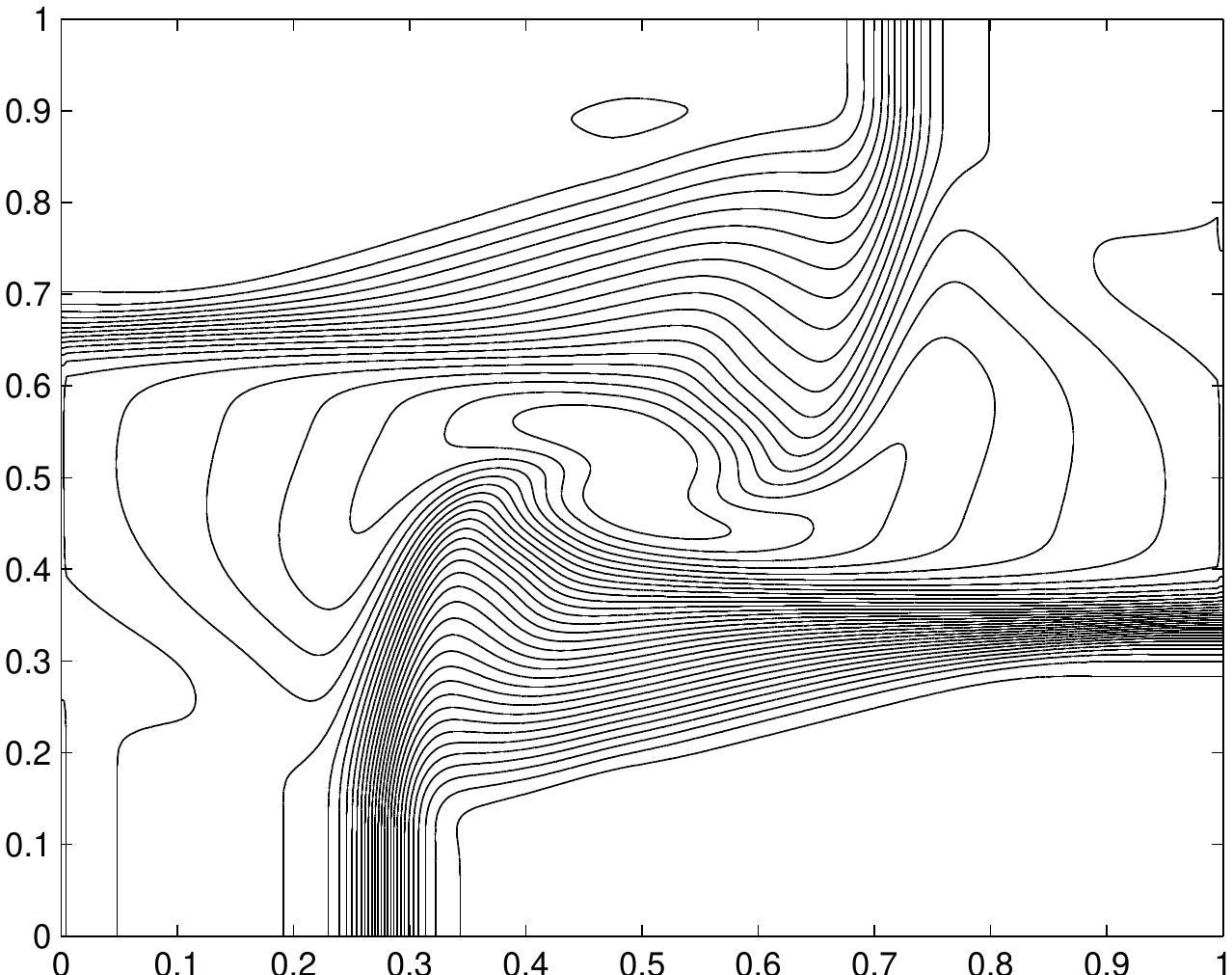} &\hspace{1cm} \includegraphics[scale=0.5]{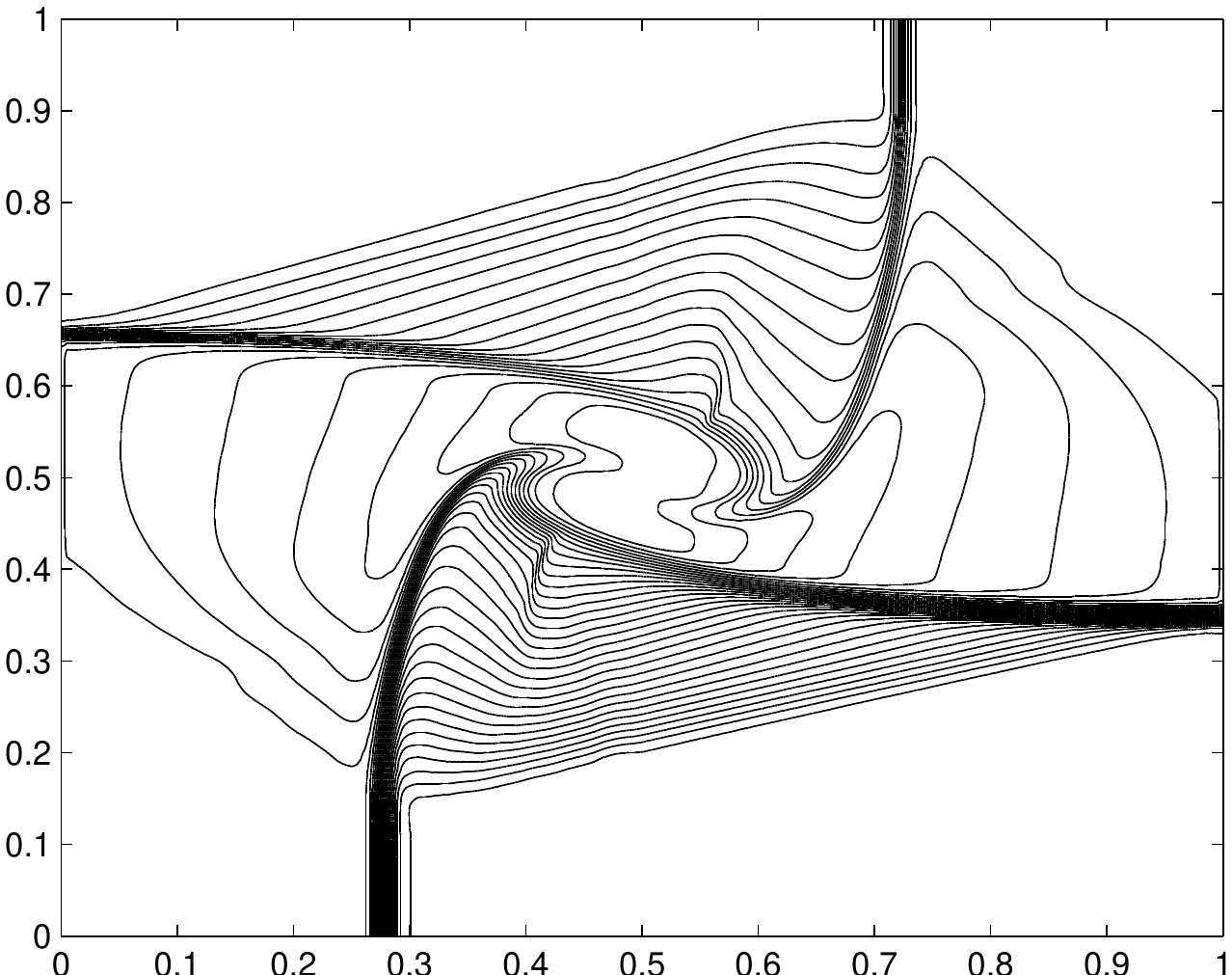} \\
(a)& (b)\\
\end{tabular}
\caption{\label{2DRP6} Configuration 6: The ripples in NE and SW
  quadrants are captured with comparable resolution with the one in
  \cite{Kurganov2002} using FLWBW-FORCE (b) though the resolution for
  contacts is little diffusive but much sharper compared to FORCE (a).}
\end{figure}
  \item[] Configuration 7
 \begin{equation}
   \begin{array}{llll}
     p_1=1      &p_2=0.4            &p_3=0.4           &p_4=0.4\\
    \rho_1=1    &\rho_2=0.5197       &\rho_3=0.8         &\rho_4=0.5197\\
     u_1=0.1    &u_2=-0.6259       & u_3=0.1          & u_4=0.1\\
     v_1=0.1    &v_2=0.1            &v_3=0.1           &v_4=-0.6259\\
   \end{array}
 \end{equation}
\begin{figure}[!htb]
\begin{tabular}{cc}
\includegraphics[scale=0.5]{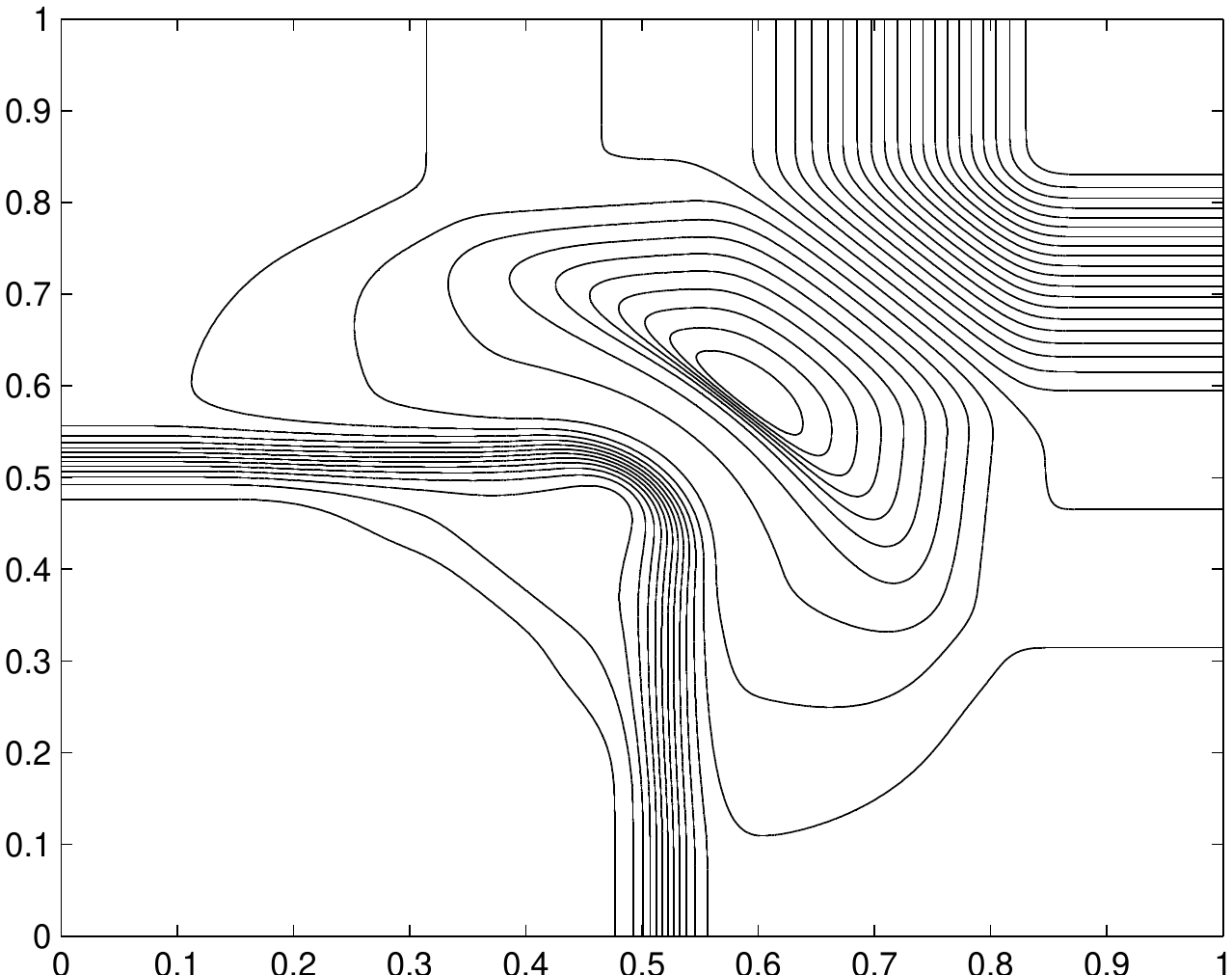} &\hspace{1cm} \includegraphics[scale=0.5]{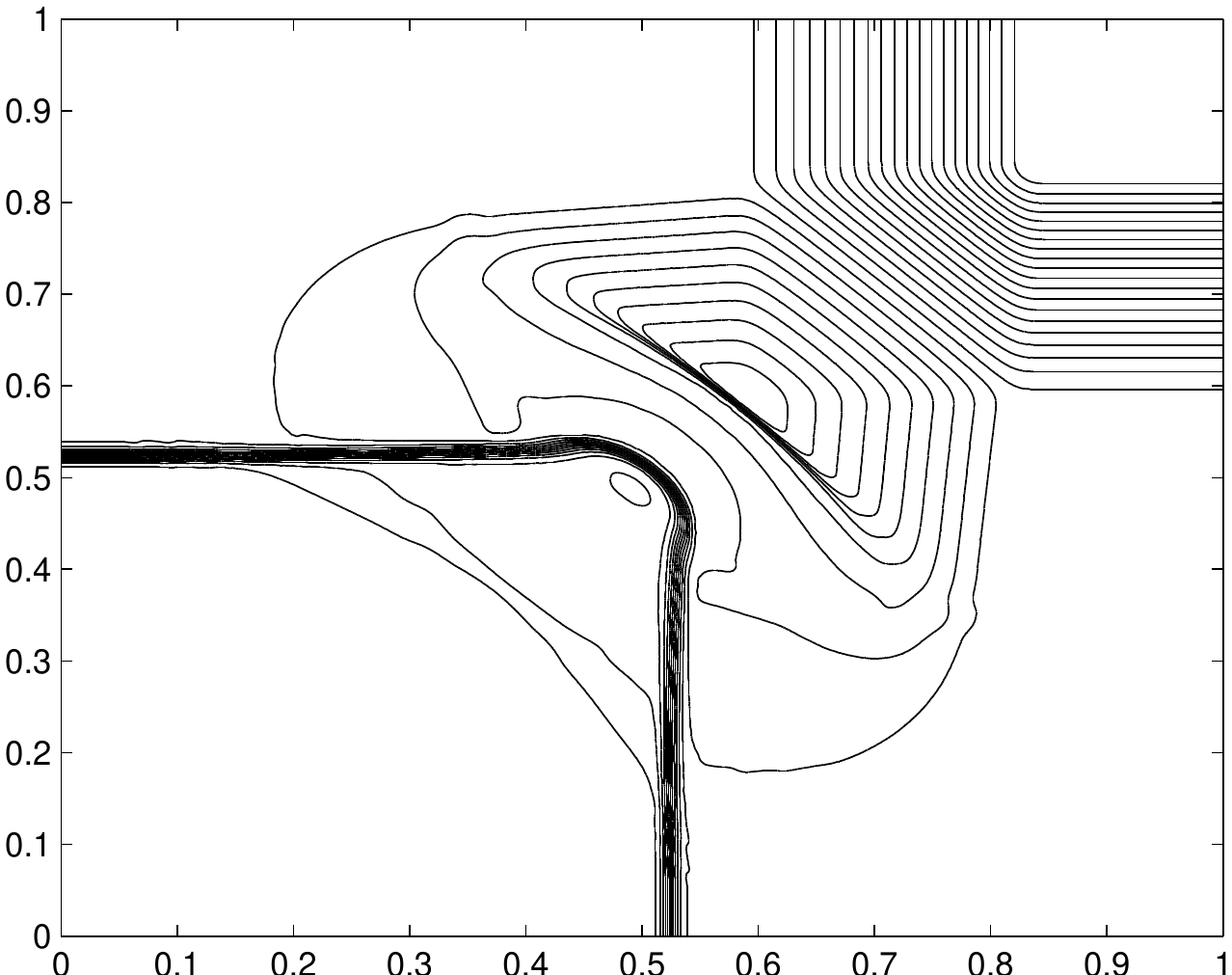}\\
(a)& (b)\\
\end{tabular}
\caption{\label{2DRP7} Configuration 7: The contacts in South and West
  are crisply resolved by FLWBW-FORCE (b). Moreover the rarefaction
  corners in NE quadrants are significantly sharper than FORCE (a).}
\end{figure}
\item[] Configuration 8
 \begin{equation}
   \begin{array}{llll}
     p_1=0.4       &p_2=1       &p_3=1      &p_4=1\\
        \rho_1=0.5197 &\rho_2=1    &\rho_3=0.8 &\rho_4=1\\
        u_1=0.1       &u_2=-0.6259 &u_3=0.1    &u_4=0.1\\
        v_1=0.1       &v_2=0.1     &v_3=0.1    &v_4=-0.6259
   \end{array}
 \end{equation}
\begin{figure}[!htb]
\begin{tabular}{cc}
\includegraphics[scale=0.5]{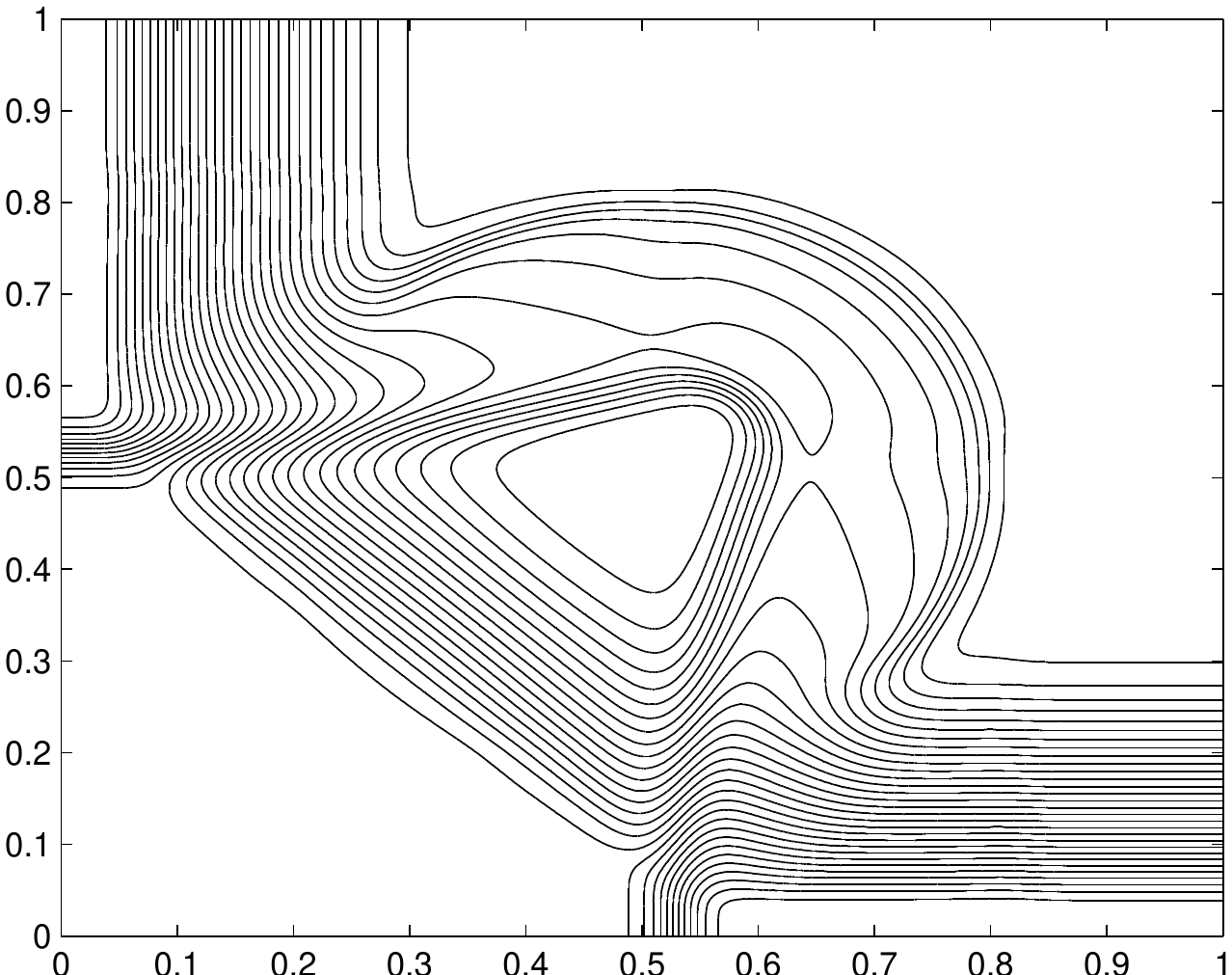} & \hspace{1cm} \includegraphics[scale=0.5]{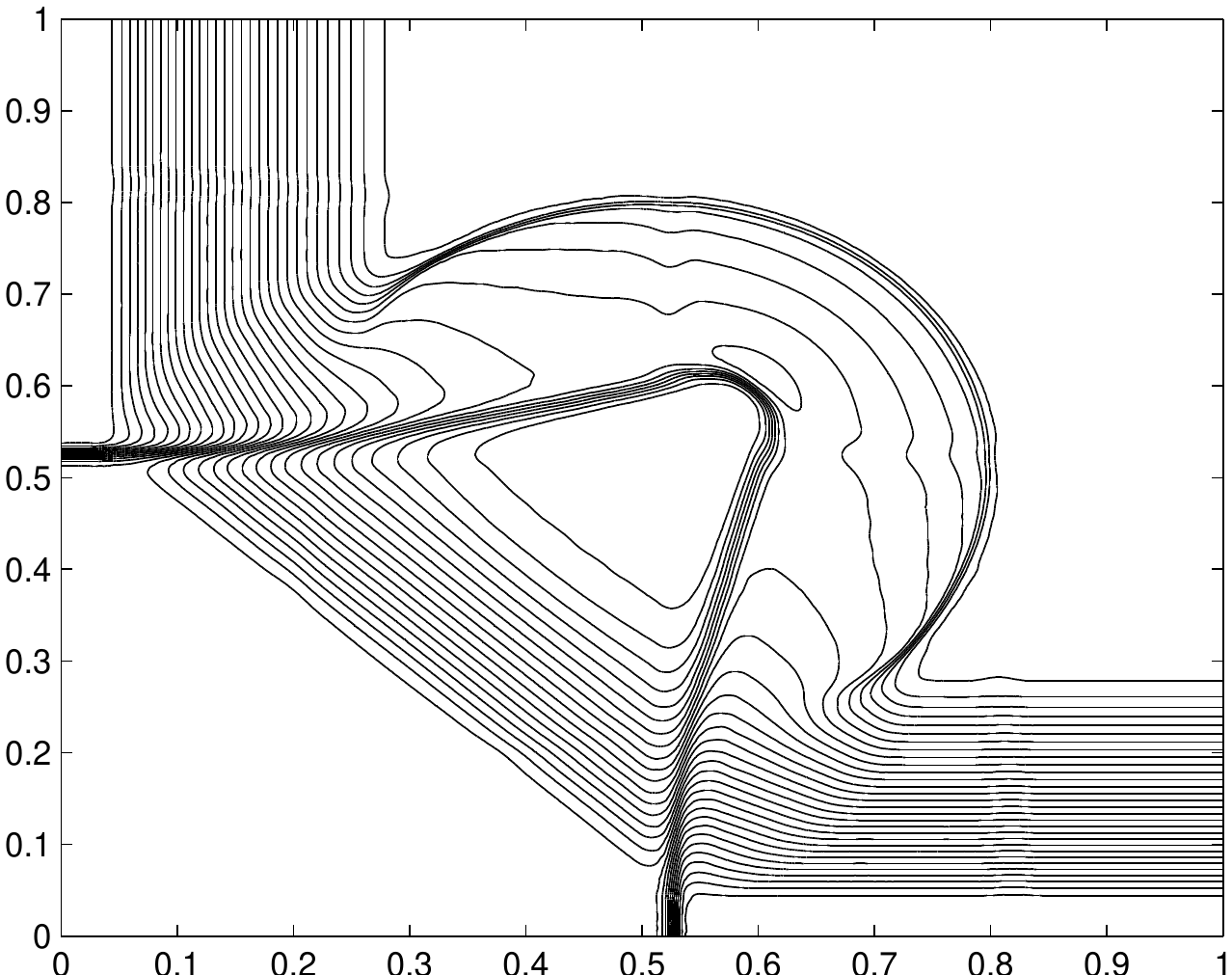} \\
(a)& (b)\\
\end{tabular}
\caption{\label{2DRP8} Configuration 8: The semi-circular wave front in NE is sharply resolved and resolution for contacts are comparable with the one in \cite{Kurganov2002}.}

\end{figure}
\item[] Configuration 9
 \begin{equation}
   \begin{array}{llll}
     p_1=1&           p_2=1&            p_3=0.4&           p_4=0.4\\
     \rho1=1&        \rho2=2&          \rho3=1.039&       \rho4=0.5197\\
     u_1=0&           u_2=0.0&          u_3=0&             u_4=0.0\\
     v_1=0.3&         v_2=-0.3&         v_3=-0.8133&       v_4=-0.4259\\
   \end{array}
 \end{equation}
\begin{figure}[!htb]
\begin{tabular}{cc}
\includegraphics[scale=0.5]{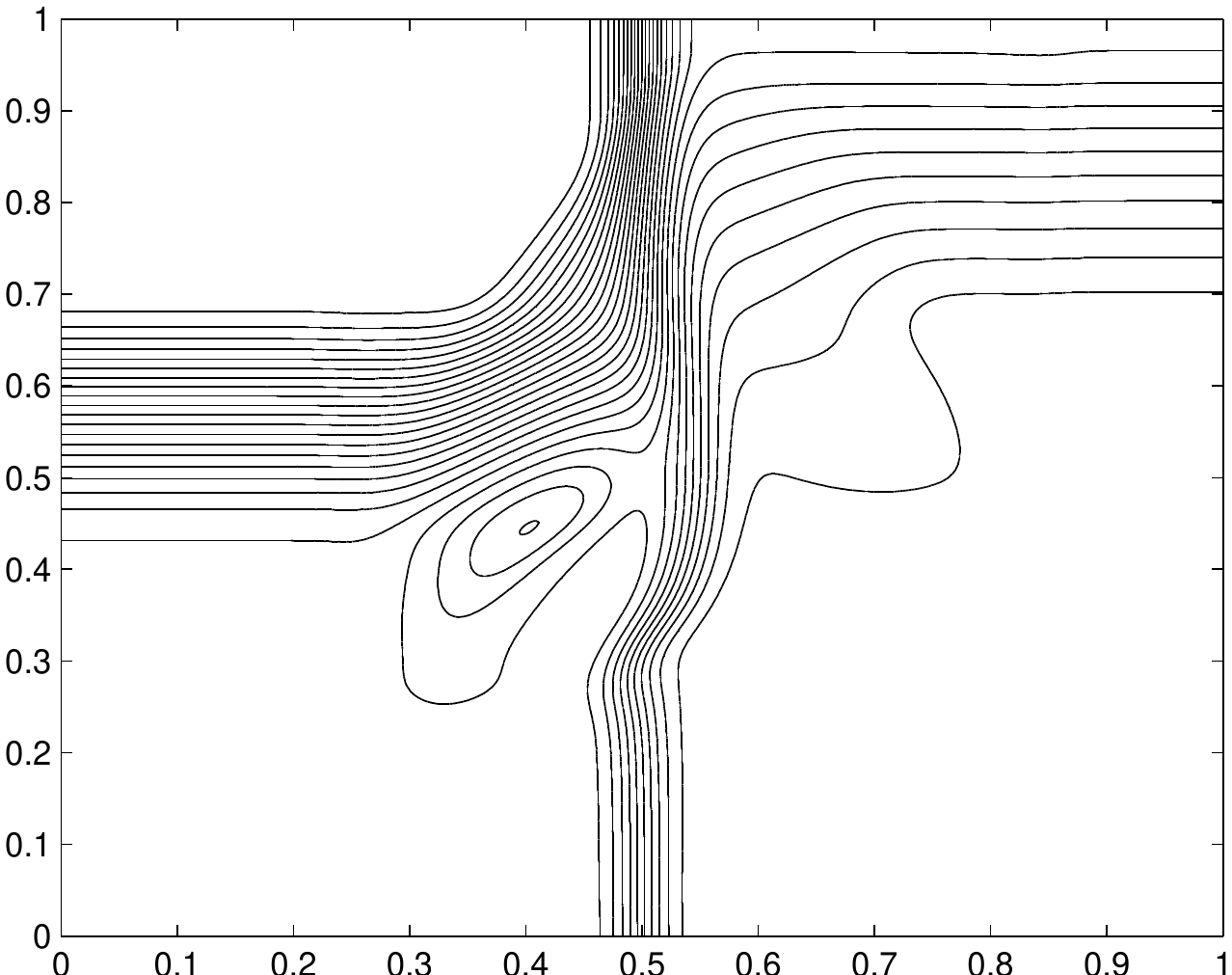} &\hspace{1cm} \includegraphics[scale=0.5]{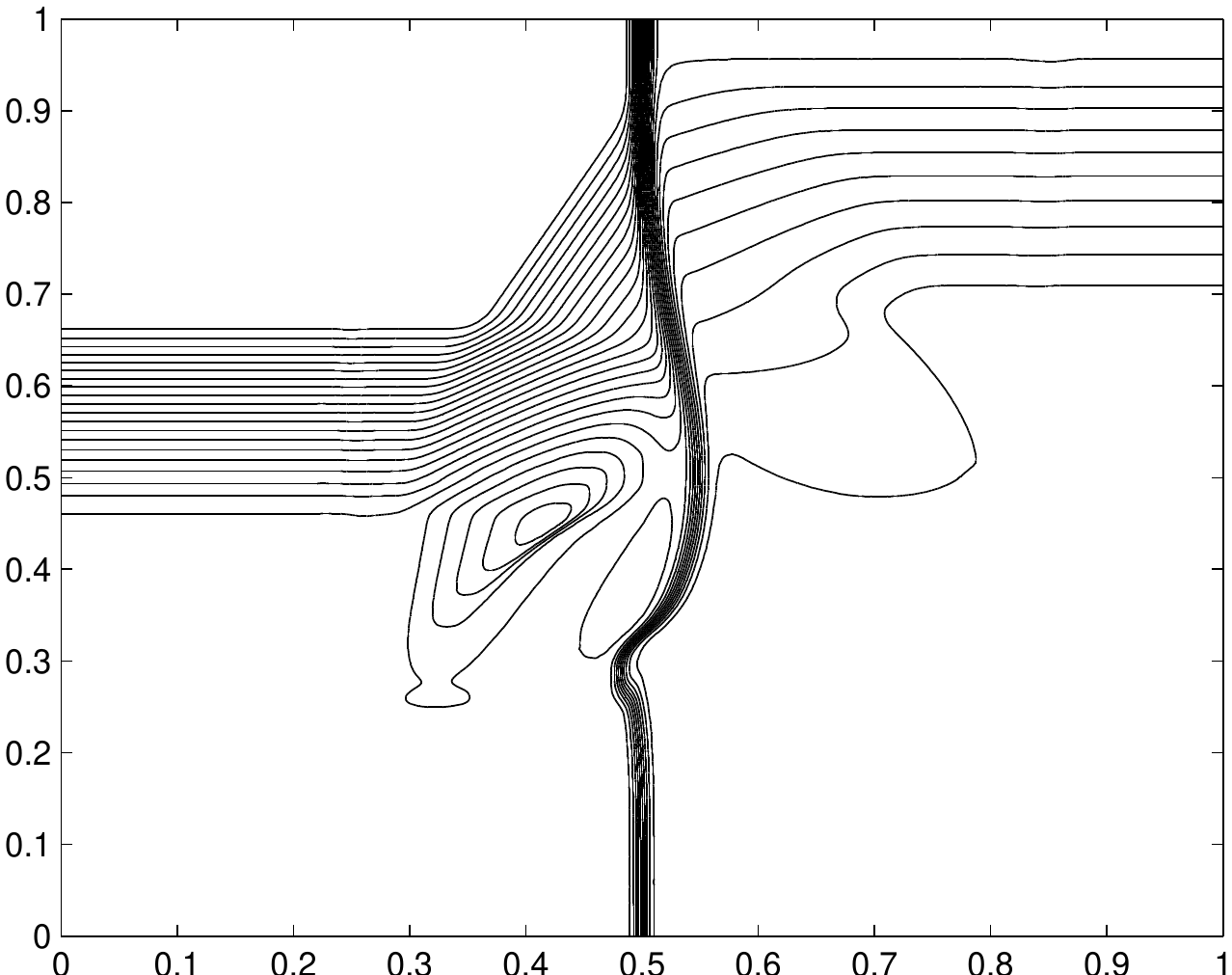} \\
(a)& (b)\\
\end{tabular}
\caption{\label{2DRP9} Configuration 9: Again rarefaction and corners are resolved sharper than FORCE (a). The vertical contact is crisply captured by FLWBW-FORCE.}
\end{figure}
\item[] Configuration 10
 \begin{equation}
   \begin{array}{llll}
     p_1=1       &p_2=1           &p_3=0.3333           &p_4=0.3333\\
    \rho_1=1     &\rho_2=0.5      &\rho_3=0.2281         &\rho_4=0.4562\\
     u_1=0       &u_2=0           &u_3=0                &u_4=0\\
     v_1=0.4297  &v_2=0.6076      &v_3=-0.6076          &v_4=-0.4297
   \end{array}
 \end{equation}
\begin{figure}[!htb]
\begin{tabular}{cc}
\includegraphics[scale=0.5]{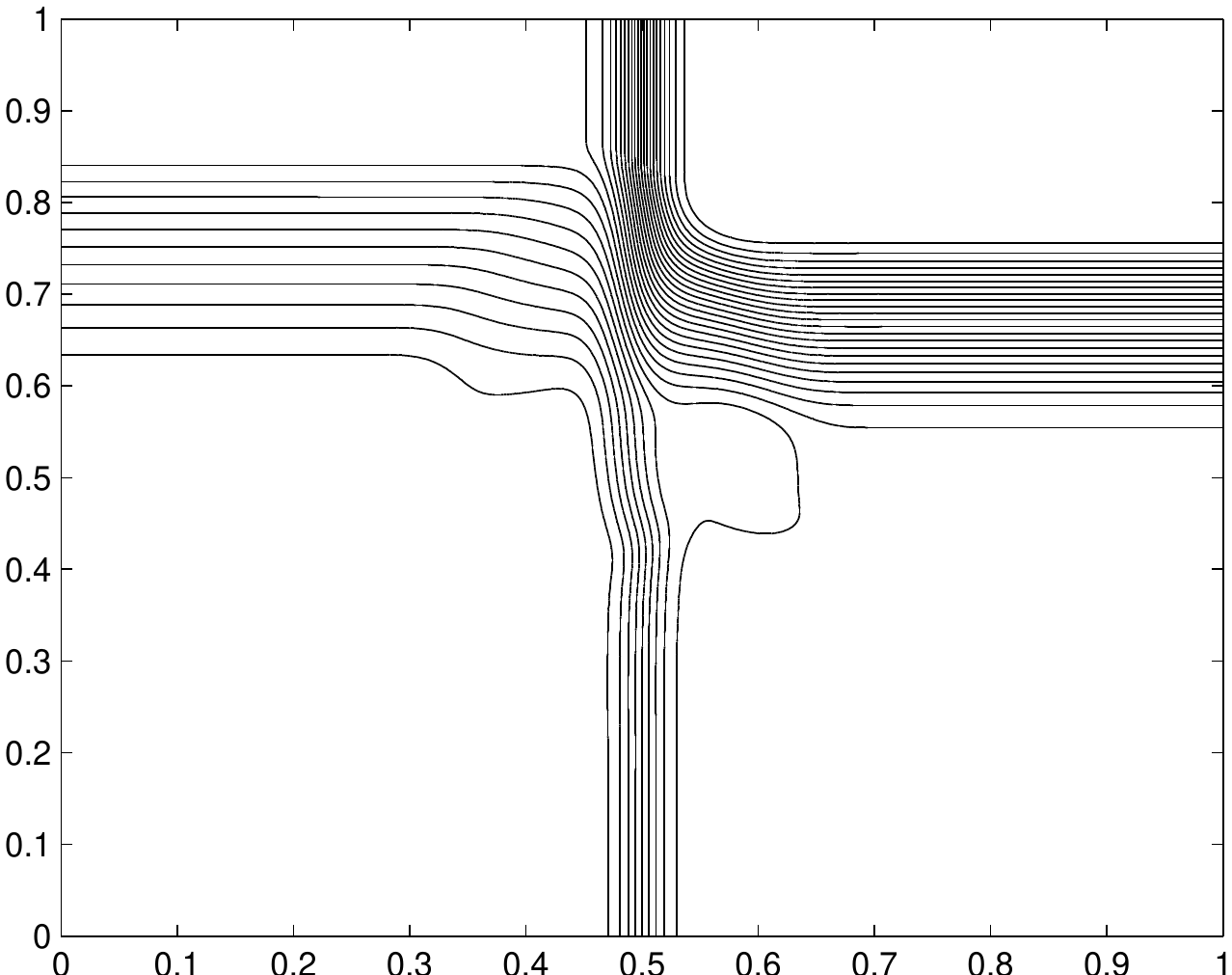} & \hspace{1cm}\includegraphics[scale=0.5]{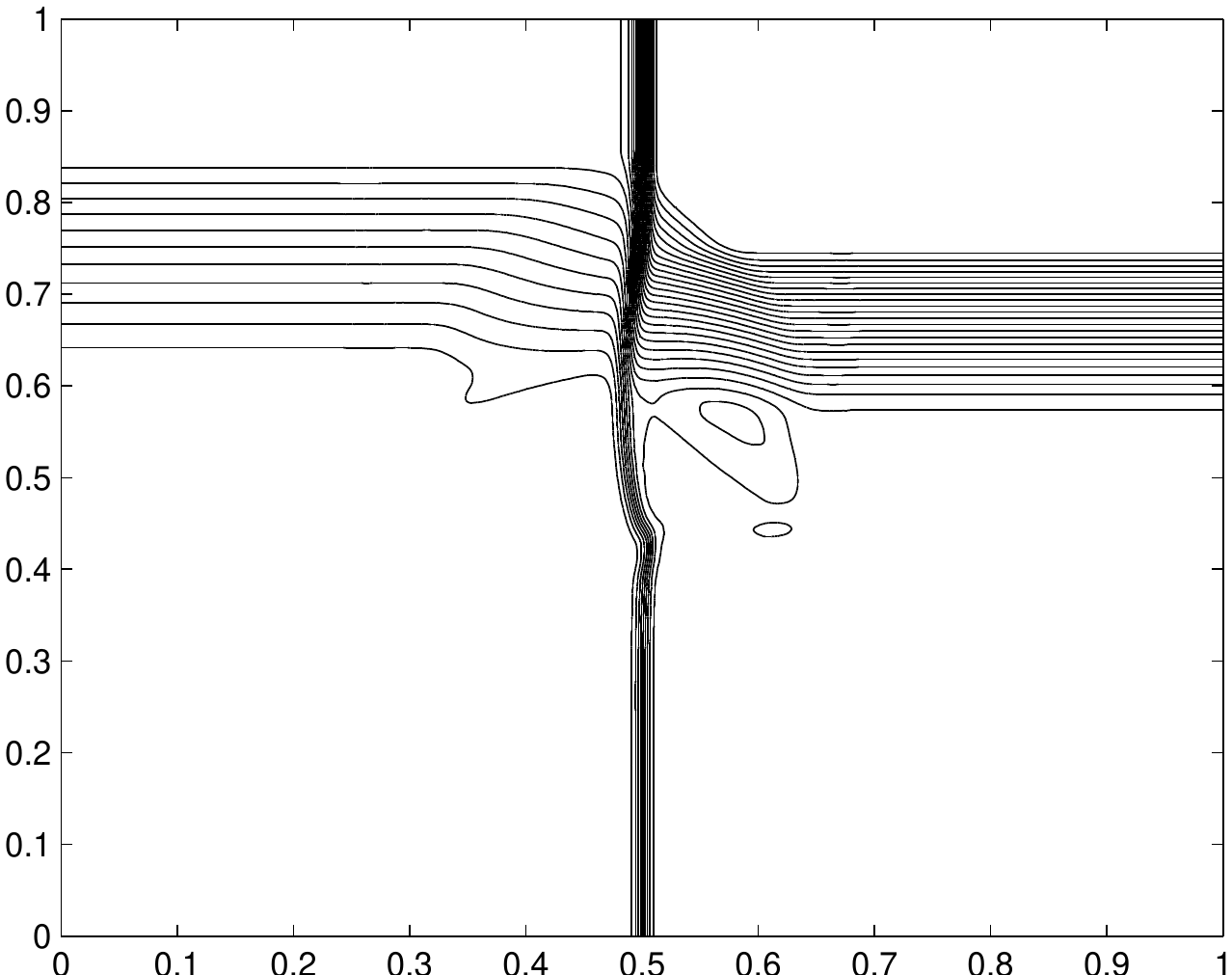} \\
(a)& (b)\\
\end{tabular}
\caption{\label{2DRP10} Configuration 10: Resolution of contacts by FLWBW-FORCE is comparable with the one in \cite{Kurganov2002}.}
\end{figure}
  \item[] Configuration 11
 \begin{equation}
   \begin{array}{llll}
     p_1=1           &p_2=0.4         &p_3=0.4           &p_4=0.4\\
   \rho_1=1          &\rho_2=0.5313    &\rho_3=0.8       &\rho_4=0.5313\\
        u_1=0.1         &u_2=0.8276      &u_3=0.1           &u_4=0.1\\
        v_1=0.0         &v_2=0.0         &v_3=0.0           &v_4=0.7276
   \end{array}
 \end{equation}

\begin{figure}[!htb]
\begin{tabular}{cc}
\includegraphics[scale=0.5]{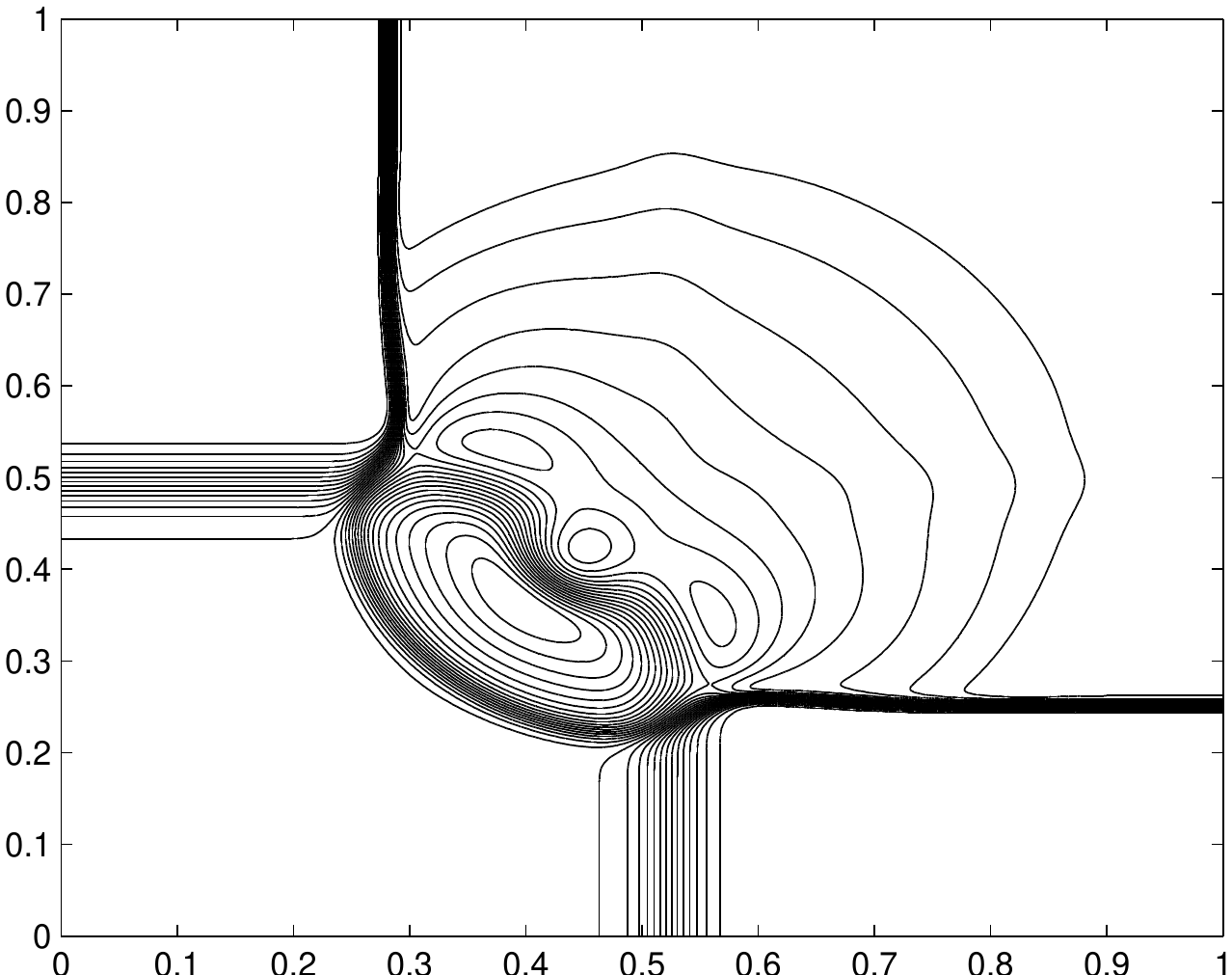} & \hspace{1cm} \includegraphics[scale=0.5]{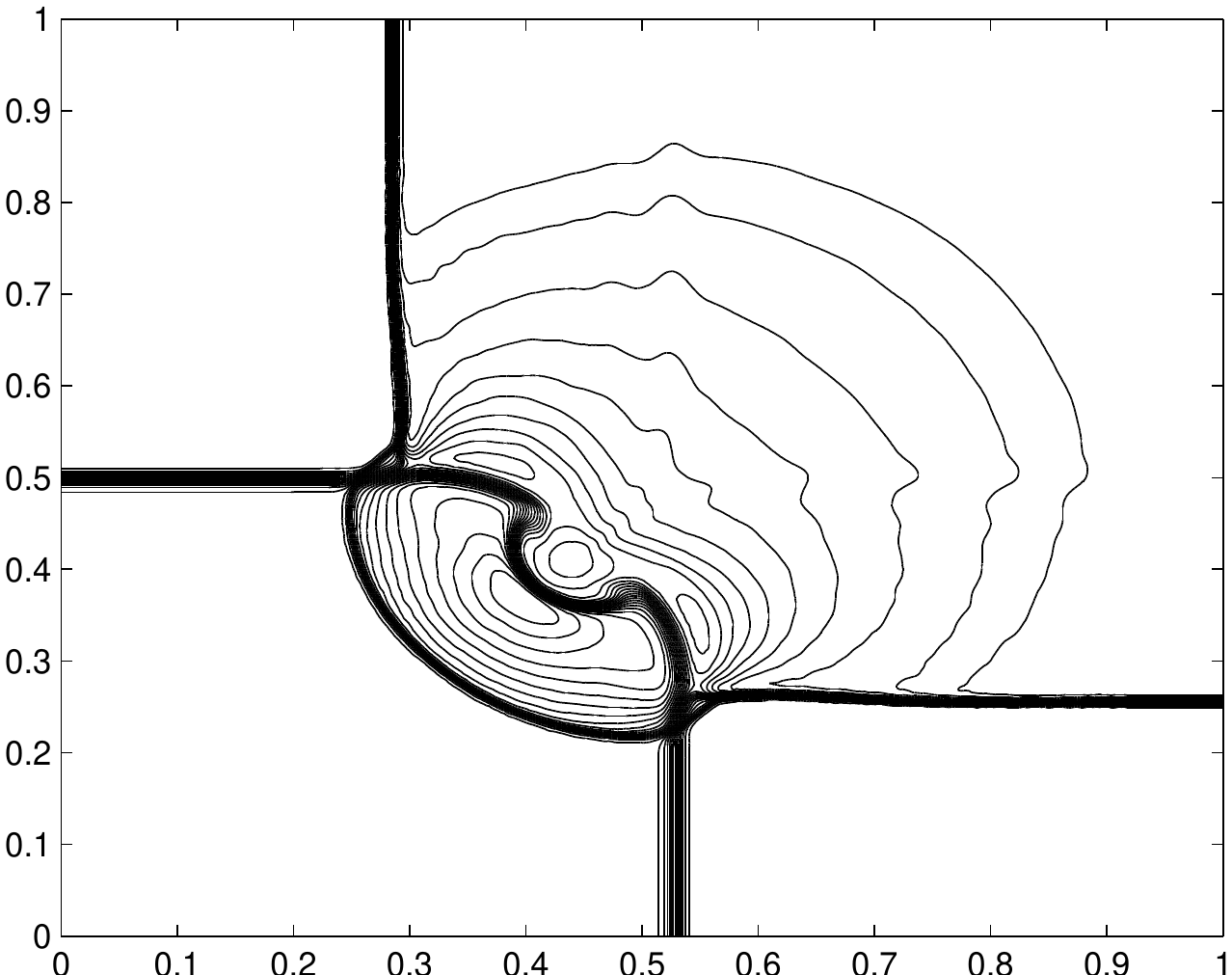} \\
(a)& (b)\\
\end{tabular}
\caption{\label{2DRP11} Configuration 11: Resolution of contact is SW quadrant, Shocks in SE and NW quadrants by FLWBW-FORCE (b) is better than that of the one in \cite{Kurganov2002, LaxLiu1998}.}
\end{figure}

\item[] Configuration 12
 \begin{equation}
   \begin{array}{llll}
     p_1=0.4 &         p_2=1&            p_3=1.0&           p_4=1.0\\
     \rho1=0.5313&     \rho2=1.0222&     \rho3=0.8&         \rho4=1.0\\
        u_1=0.1&          u_2=-0.6179&      u_3=0.1&         u_4=0.1\\
        v_1=0.1 &         v_2=0.1&          v_3=0.1&         v_4=0.8276
   \end{array}
 \end{equation}

\begin{figure}[!htb]
\begin{tabular}{cc}
\includegraphics[scale=0.5]{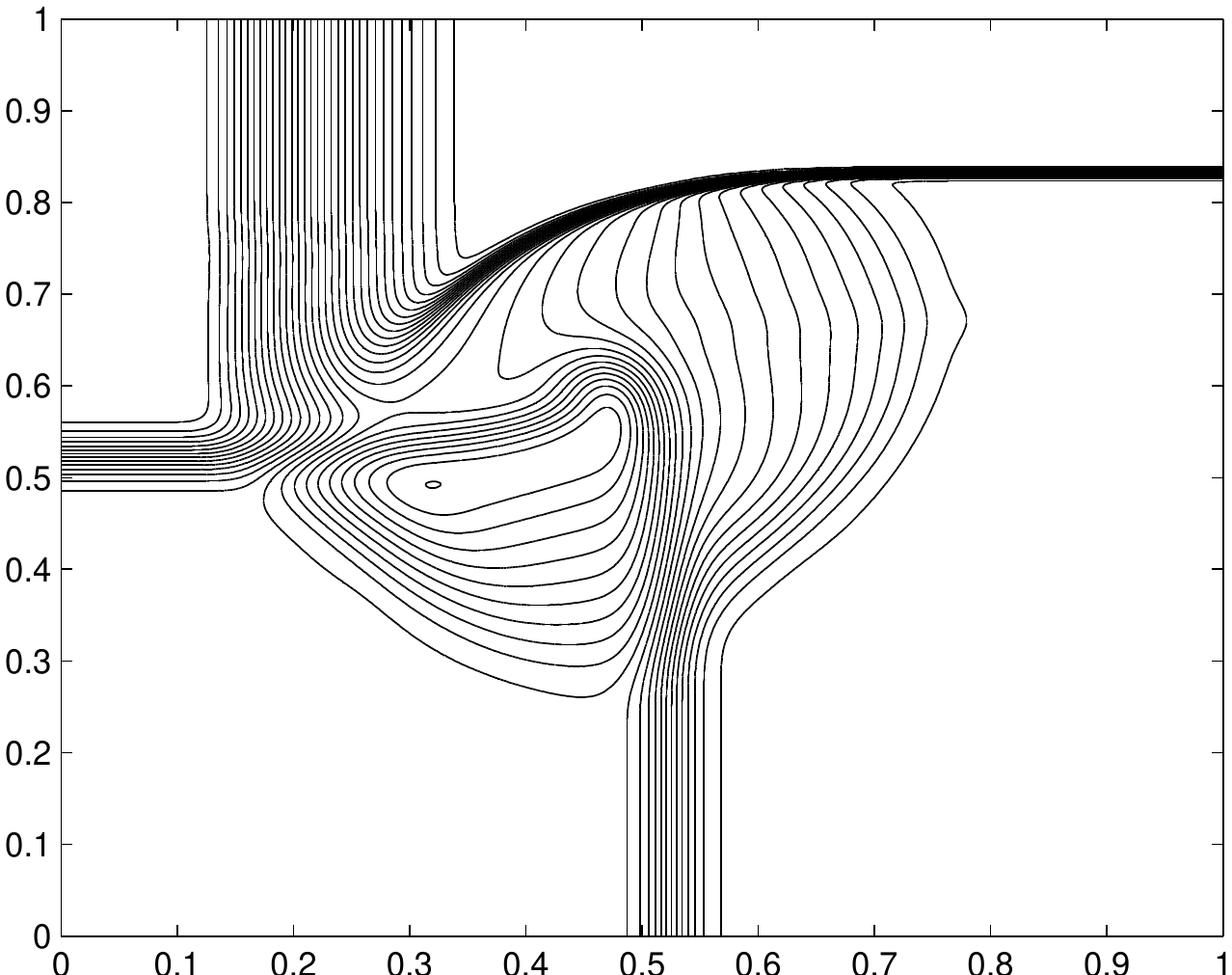} &\hspace{1cm} \includegraphics[scale=0.5]{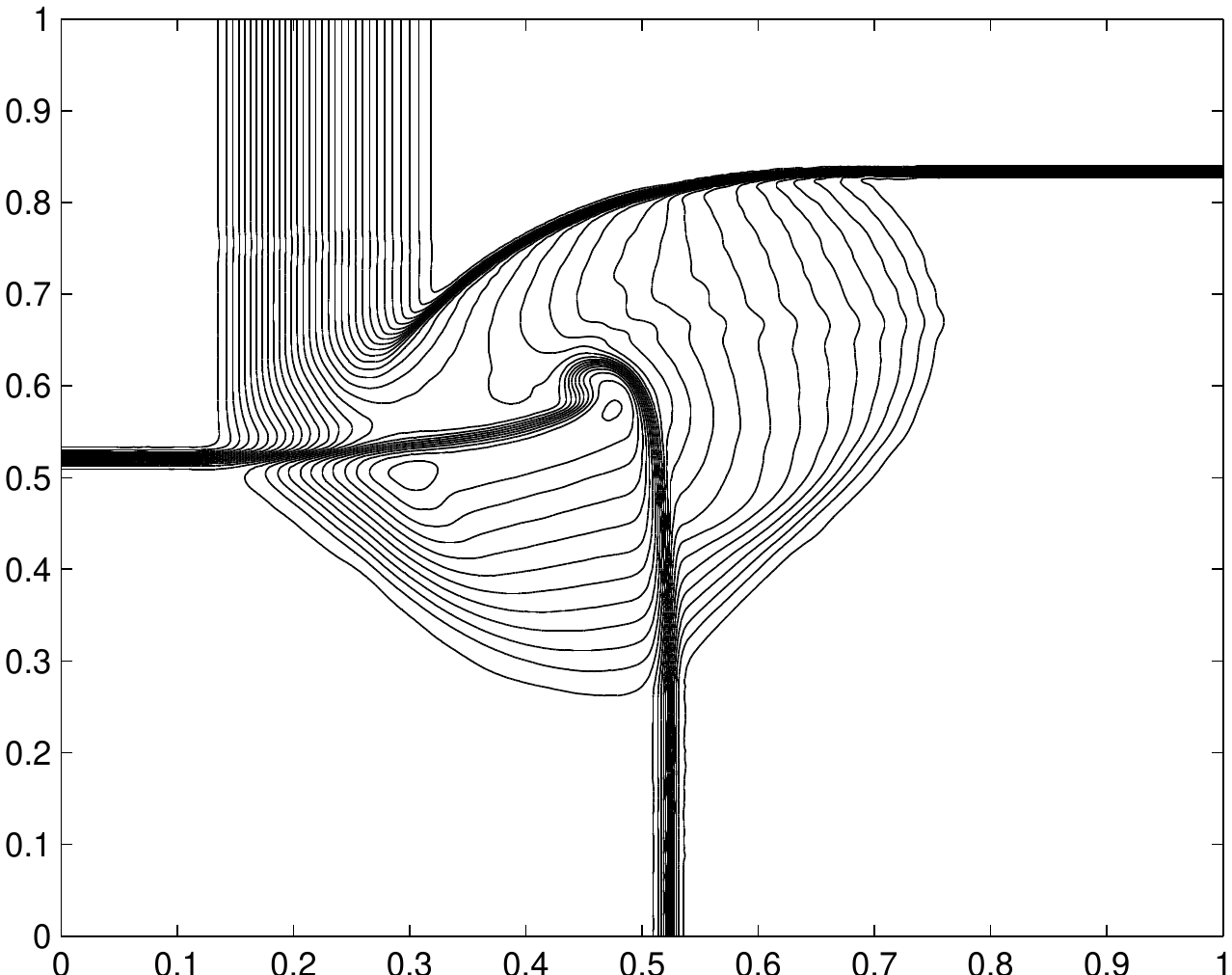} \\
(a)& (b)\\
\end{tabular}
\caption{\label{2DRP14} Configuration 12: FLWBW-FORCE recovers the
  ripples between NE shock and contact waves. The resolution for shock
  and contacts are comparable with the second order scheme results in
  \cite{Kurganov2002}.}
\end{figure}
\end{enumerate}
\textbf{Comments}


\section{Conclusion and Future work}\label{sec7}
In this work LMP/TVD bounds are obtained for uniformly second order
accurate schemes in non-conservative form. These bound show that
higher than second order TVD accuracy can be achieved at extrema and
steep gradient region in limiting sense i.e., when $r\rightarrow
0^{-}$. Based on the LMP/TVD bounds hybrid local maximum principle
satisfying schemes are constructed and applied on various benchmark
test problems. Numerical results show improvement in TVD approximation
of solution region with extreme points, smooth rarefaction as well
contact discontinuity compared to existing higher order TVD
method. For a separate work, the focus is on TVD bounds for multi-step
methods and efficient use of a shock detector. As, the algorithm
\ref{algo4} recovered the shock at right location for scalar case, it
would be interesting to devise a hybrid scheme for system by modifying
the wave speed choice in section \ref{algo4sys}, with out a shock
detector.




\begin{thebibliography}{10}
\bibitem{Breuss2010}
Michael Breuss.
\newblock An analysis of the influence of data extrema on some first and second
  order central approximations of hyperbolic conservation laws.
\newblock {\em ESAIM: Mathematical Modelling and Numerical Analysis},
  39(5):965--993, 3 2010.

\bibitem{jameson1995}
Jameson, Antony.
\newblock Positive schemes and shock modelling for compressible flows.
\newblock {\em International Journal for Numerical Methods in Fluids},20:743-776, 1995.

\bibitem{shuosher}
Shu Chi-Wang and Stanley Osher.
\newblock Efficient implementation of essentially non-oscillatory
  shock-capturing schemes, ii.
\newblock {\em Journal of Computational Physics}, 83:32--78, 1989.

\bibitem{Crandall}
Michael~G. Crandall and Andrew Majda.
\newblock Monotone difference approximations for scalar conservation laws.
\newblock {\em Mathematics of Computation}, 34(149):1--21, jan 1980.

\bibitem{davis1987}
S.~F. Davis.
\newblock A simplified tvd finite difference scheme via artificial viscosity.
\newblock {\em SIAM J. Sci. and Stat. Comput.}, 8(1):pp. 1--18, 1987.

\bibitem{despres}
Bruno Després and Frédéric Lagoutière.
\newblock Contact discontinuity capturing schemes for linear advection and
  compressible gas dynamics.
\newblock {\em Journal of Scientific Computing}, 16(4):479--524, 2001.

\bibitem{riteshEJDE}
Ritesh~Kumar Dubey.
\newblock Total variation stability and second-order accuracy at extrema.
\newblock In Goddard Jerome and Zhu Jianping, editors, {\em Ninth MSU-UAB
  Conference on Differential Equations and Computational Simulations, Electron.
  J. Diff. Eqns. Conf.}, pages 53--63, sept 2012.

\bibitem{Dubey2013}
Ritesh~Kumar Dubey.
\newblock Flux limited schemes: Their classification and accuracy based on
  total variation stability regions.
\newblock {\em Applied Mathematics and Computation}, 224(0):325 -- 336, 2013.


\bibitem{godunov}
S.~K. Godunov.
\newblock A difference scheme for numerical solution of discontinuous solution
  of hydrodynamic equations, translated us joint publ. res. service, jprs 7225
  nov. 29, 1960.
\newblock {\em Math. Sbornik}, 47:271--306, 1959.

\bibitem{goodman1988}
Jonathan~B. Goodman and Randall~J. LeVeque.
\newblock A geometric approach to high resolution tvd schemes.
\newblock {\em SIAM J. Numer. Anal.}, 25:268--284, April 1988.

\bibitem{harten1983}
Ami Harten.
\newblock High resolution schemes for hyperbolic conservation laws,.
\newblock {\em Journal of Computational Physics}, 49(3):357 -- 393, 1983.

\bibitem{Harten1989}
Ami Harten.
\newblock Eno schemes with subcell resolution.
\newblock {\em Journal of Computational Physics}, 83(1):148 -- 184, 1989.


\bibitem{harten1984}
Ami Harten and Peter~D. Lax.
\newblock On a class of high resolution total-variation-stable
  finite-difference schemes.
\newblock {\em SIAM Journal on Numerical Analysis}, 21(1):pp. 1--23, 1984.


\bibitem{Hou1994}
Thomas~Y. Hou and Philippe~G. Lefloch.
\newblock Why nonconservative schemes converge to wrong weak solutions: Error
  analysis.
\newblock {\em Mathematics of Computation}, 62(206), 1994.

\bibitem{jairaghu}
S.~Jaisankar and S.V.~Raghurama Rao.
\newblock A central rankine–hugoniot solver for hyperbolic conservation laws.
\newblock {\em Journal of Computational Physics}, 228(3):770 -- 798, 2009.


\bibitem{rkd2007}
M.K. Kadalbajoo and Ritesh Kumar.
\newblock A high resolution total variation diminishing scheme for hyperbolic
  conservation law and related problems.
\newblock {\em Applied Mathematics and Computation}, 175(2):1556 -- 1573, 2006.
\bibitem{rkdejde1}
Ritesh Kumar Dubey. 
\newblock A hybrid semi-primitive shock capturing scheme for conservation laws.
\newblock{\em  Electron. J. Diff. Eqns.,} Eighth Mississippi State - UAB Conference on Differential Equations and Computational Simulations. Conference 19 (2010), pp. 65-73.








\bibitem{Kurganov2012}
Alexander Kurganov and Yu~Liu.
\newblock New adaptive artificial viscosity method for hyperbolic systems of
  conservation laws.
\newblock {\em Journal of Computational Physics}, 231(24):8114 -- 8132, 2012.

\bibitem{Kurganov2002}
Alexander Kurganov and Eitan Tadmor.
\newblock Solution of two-dimensional riemann problems for gas dynamics without
  riemann problem solvers.
\newblock {\em Numerical Methods for Partial Differential Equations},
  18(5):584--608, 2002.

\bibitem{Laney}
Culbert~B. Laney.
\newblock {\em {Computational gasdynamics}}.
\newblock Cambridge University Press, 1998.

\bibitem{LaxLiu1998}
Peter~D. Lax and Xu-Dong Liu.
\newblock Solution of two-dimensional riemann problems of gas dynamics by
  positive schemes.
\newblock {\em SIAM J. Sci. Comput.}, 19(2):319--340, March 1998.

\bibitem{Lefloach99}
Philippe~G. Lefloch and Jian~Guo Liu.
\newblock Generalized monotone schemes, discrete paths of extrema, and discrete
  entropy conditions.
\newblock {\em Mathematics of Computation}, 68:1025--1055, 1999.

\bibitem{Lele1992}
Sanjiva~K. Lele.
\newblock Compact finite difference schemes with spectral-like resolution.
\newblock {\em Journal of Computational Physics}, 103(1):16--42, November 1992.

\bibitem{LeVeque1992}
Randall~J. LeVeque.
\newblock {\em {Numerical Methods for Conservation Laws}}.
\newblock Lectures in mathematics ETH Z\"{u}rich. Birkh\"{a}user Basel, 2nd
  edition, February 1992.

\bibitem{jiequan2011}
J.~Li and Z.~Yang.
\newblock Heuristic modified equation analysis on oscillations in numerical
  solutions of conservation laws.
\newblock {\em SIAM Journal on Numerical Analysis}, 49(6):2386--2406, 2011.

\bibitem{jiequan2009}
Jiequan Li, Huazhong Tang, Gerald Warnecke, and Lumei Zhang.
\newblock Local oscillations in finite difference solutions of hyperbolic
  conservation laws.
\newblock {\em Math. Comput.}, 78(268):1997--2018, 2009.


\bibitem{Liushock}
M.~Oliveira, P.~Lu, X.~Liu, and C.~Liu.
\newblock A new shock/discontinuity detector.
\newblock {\em International Journal of Computer Mathematics},
  87(13):3063--3078, 2010.

\bibitem{chakra}
S.~Osher and S.~Chakravarthy.
\newblock High resolution schemes and entropy condition.
\newblock {\em SIAM J. Numer. Anal.}, 21:955--984, 1984.

\bibitem{tadmor1988}
Stanley Osher and Eitan Tadmor.
\newblock On the convergence of difference approximations to scalar
  conservation laws.
\newblock {\em Mathematics of Computation}, 50(181):pp. 19--51, 1988.

\bibitem{Piperno}
Serge Piperno and Sophie Depeyre.
\newblock Criteria for the design of limiters yielding efficient high
  resolution tvd schemes.
\newblock {\em Computers \& Fluids}, 27(2):183 -- 197, 1998.

\bibitem{Rider}
William~J. Rider.
\newblock A comparison of tvd lax-wendroff methods.
\newblock {\em Communications in Numerical Methods in Engineering},
  9(2):147--155, 1993.


\bibitem{roe1985some}
Philip~L Roe.
\newblock Some contributions to the modelling of discontinuous flows.
\newblock In {\em Large-scale computations in fluid mechanics}, volume~1, pages
  163--193, 1985.


\bibitem{Sanders1983}
Richard Sanders.
\newblock On convergence of monotone finite difference schemes with variable
  spatial differencing.
\newblock {\em Mathematics of Computation}, 40(161):91--106, Jan 1983.

\bibitem{sanders1988}
Richard Sanders.
\newblock A third-order accurate variation nonexpansive difference scheme for
  single nonlinear conservation laws.
\newblock {\em Mathematics of Computation}, 51(184):pp. 535--558, 1988.

\bibitem{cwshu1999}
Chi-Wang Shu.
\newblock High order eno and weno schemes for computational fluid dynamics.
\newblock In TimothyJ. Barth and Herman Deconinck, editors, {\em High-Order
  Methods for Computational Physics}, volume~9 of {\em Lecture Notes in
  Computational Science and Engineering}, pages 439--582. Springer Berlin
  Heidelberg, 1999.

\bibitem{cwshu2012}
Chi-Wang Shu.
\newblock Efficient algorithms for solving partial differential equations with
  discontinuous solutions.
\newblock {\em Notices of the AMS}, 59(5), 2012.


\bibitem{Sod1978}
G.A. Sod.
\newblock Survey of several finite difference methods for systems of nonlinear
  hyperbolic conservation laws.
\newblock {\em Journal of Computational Physics}, Apr 1978.

\bibitem{sweby1984}
P.~K. Sweby.
\newblock High resolution schemes using flux limiters for hyperbolic
  conservation laws.
\newblock {\em Siam Journal on Numerical Analysis}, 21(5):995--1011, 1984.

\bibitem{Thomas}
J~W. Thomas.
\newblock {\em Numerical partial differential equations- conservation laws and
  elliptic equations}.
\newblock Texts in Applied Mathematics 33. Springer-Verlag, 1999.

\bibitem{toro2000}
EF~Toro and SJ~Billett.
\newblock Centred tvd schemes for hyperbolic conservation laws.
\newblock {\em IMA Journal of Numerical Analysis}, 20(1):47--79, 2000.

\bibitem{toro2009}
Eleuterio~F. Toro.
\newblock {\em {Riemann Solvers and Numerical Methods for Fluid Dynamics: A
  Practical Introduction}}.
\newblock Springer, 3rd edition, April 2009.

\bibitem{vanLeer1974}
Bram van Leer.
\newblock Towards the ultimate conservative difference scheme. ii. monotonicity
  and conservation combined in a second-order scheme.
\newblock {\em Journal of Computational Physics}, 14(4):361 -- 370, 1974.

\bibitem{Vides2015643}
Jeaniffer Vides, Boniface Nkonga, and Edouard Audit.
\newblock A simple two-dimensional extension of the \{HLL\} riemann solver for
  hyperbolic systems of conservation laws.
\newblock {\em Journal of Computational Physics}, 280(0):643 -- 675, 2015.

\bibitem{yang1996}
H.~Yang.
\newblock On wavewise entropy inequalities for high-resolution schemes. i: the
  semidiscrete case.
\newblock {\em Math. Comp.}, 65:45–67, 1996.

\bibitem{Yee1987}
H.~C. Yee.
\newblock Construction of explicit and implicit symmetric tvd schemes and their
  applications.
\newblock {\em Journal of Computational Physics}, 68(1):151 -- 179, 1987.


\bibitem{zhang2011}
X.~Zhang and C-W Shu.
\newblock Maximum-principle-satisfying and positivity-preserving high-order
  schemes for conservation laws: survey and new developments.
\newblock {\em Proc. R. Soc. A}, 467:2752--2776, 2011.

\bibitem{zhang2010}
Xiangxiong Zhang and Chi-Wang Shu.
\newblock A genuinely high order total variation diminishing scheme for
  one-dimensional scalar conservation laws.
\newblock {\em SIAM J. Numerical Analysis}, 48(2):772--795, 2010.

\bibitem{zhangjcp2010}
Xiangxiong Zhang and Chi-Wang Shu.
\newblock On maximum-principle-satisfying high order schemes for scalar
  conservation laws.
\newblock {\em Journal of Computational Physics}, 229(9):3091 -- 3120, 2010.

\end{thebibliography}
\end{document}